\documentclass[12pt]{article}

\usepackage{amsmath}
\usepackage{amsthm}
\usepackage{amssymb}
\usepackage{amscd}
\usepackage{xspace}
\usepackage{verbatim}
\usepackage{geometry}
\usepackage{esint}
\newtheorem{theorem}{Theorem}[section]
\newtheorem{corollary}{Corollary}[section]
\newtheorem{lemma}{Lemma}[section]
\newtheorem{proposition}{Proposition}[section]
\newtheorem{example}{Example}[section]
\theoremstyle{definition}
\newtheorem{definition}{Definition}[section]
\newtheorem*{acknowledgement}{Acknowledgement}
\theoremstyle{remark}
\newtheorem{remark}{Remark}[section]
\numberwithin{equation}{section}

\newcommand{\ov}{\overline}

\newcommand{\e}{\varepsilon}

\newcommand{\N}{\mathbb N}

\renewcommand{\vec}[1]{\mathbf{#1}}

\newcommand{\R}{\mathbb{R}}

\newcommand{\er}{\eqref}

\DeclareMathOperator{\Div}{div} \DeclareMathOperator{\dist}{dist}

\DeclareMathOperator{\Supp}{Supp}
\DeclareMathOperator{\cp}{cap}

\newcommand{\Haus}{\mathcal H}

\newcommand{\Leb}{\mathcal L}

\renewcommand{\vec}[1]{\boldsymbol{#1}}

\date{}

\begin{document}
\title{Jumps in Besov spaces and fine properties of Besov and fractional Sobolev functions}
\author{Paz Hashash and Arkady Poliakovsky}
\date{\today}
\maketitle
\begin{abstract}
In this paper we analyse functions in Besov spaces $B^{1/q}_{q,\infty}(\mathbb{R}^N,\mathbb{R}^d),q\in (1,\infty)$, and functions in fractional Sobolev spaces $W^{r,q}(\mathbb{R}^N,\mathbb{R}^d),r\in (0,1),q\in [1,\infty)$. We prove for Besov functions $u\in B^{1/q}_{q,\infty}(\mathbb{R}^N,\mathbb{R}^d)$ the summability of the difference between one-sided approximate limits in power $q$, $|u^+-u^-|^q$, along the jump set $\mathcal{J}_u$ of $u$ with respect to Hausdorff measure
$\mathcal{H}^{N-1}$, and establish the best bound from above on the integral $\int_{\mathcal{J}_u}|u^+-u^-|^qd\mathcal{H}^{N-1}$ in terms of Besov constants. We show for functions $u\in B^{1/q}_{q,\infty}(\mathbb{R}^N,\mathbb{R}^d),q\in (1,\infty)$  that
\begin{equation}
\liminf\limits_{\varepsilon\to 0^+}\frac{1}{\varepsilon^N}\int_{B_{\varepsilon}(x)}
|u(z)-u_{B_{\varepsilon}(x)}|^qdz=0
\end{equation}
for every $x$ outside of a $\mathcal{H}^{N-1}$-sigma finite set. For fractional Sobolev functions $u\in W^{r,q}(\mathbb{R}^N,\mathbb{R}^d)$ we prove that
\begin{equation}
\lim_{\varepsilon\to
0^+}\frac{1}{\varepsilon^N}\int_{B_{\varepsilon}(x)}\frac{1}{\varepsilon^N}\int_{B_{\varepsilon}(x)}
|u\big(z\big)-u(y)|^qdzdy=0
\end{equation}
for $\mathcal{H}^{N-rq}$ a.e. $x$, where $q\in[1,\infty)$, $r\in(0,1)$ and $rq\leq N$. We prove for $u\in W^{1,q}(\mathbb{R}^N),1<q\leq N$, that
\begin{equation}
\lim\limits_{\varepsilon\to 0^+}\frac{1}{\varepsilon^N}\int_{B_{\varepsilon}(x)}
|u(z)-u_{B_{\varepsilon}(x)}|^qdz=0
\end{equation}
for $\mathcal{H}^{N-q}$ a.e. $x\in \R^N$.

In addition, we prove Lusin-type approximation for fractional Sobolev functions $u\in W^{r,q}(\mathbb{R}^N,\mathbb{R}^d)$ by H{\"o}lder continuous functions in $C^{0,r}(\mathbb{R}^N,\mathbb{R}^d)$.
\tableofcontents
\end{abstract}
\maketitle

\section{\textbf{Introduction}}
Throughout the article the letters $N,d$ denote natural numbers.

The so-called ``$BBM$ formula", which was represented by Bourgain, Brezis and Mironescu in \cite{BBM}, gives a characterization of Sobolev functions $W^{1,q}(\Omega)$ for
$1<q<\infty$ and of functions of bounded variation $BV(\Omega)$ using
double integrals and mollifiers, where $\Omega\subset\mathbb{R}^N$ is
open and bounded set with Lipschitz boundary. The characterization for
$BV(\Omega)$ functions in its full strength is due to D{\'a}vila
\cite{Davila}. For a general introduction to the theory of Sobolev
functions see \cite{EG,Maz'ya} and for the theory of functions of
bounded variation see \cite{AFP,EG}. The $BBM$ formula says, in
particular, that for an open and bounded set
$\Omega\subset\mathbb{R}^N$ with Lipschitz boundary, $1<q<\infty$ and
$u\in W^{1,q}(\Omega)$ we have
\begin{equation}
\lim_{\epsilon\to 0^+}\int_{\Omega}\left(\int_{\Omega\cap B_{\epsilon}(x)}\frac{1}{\epsilon^N}\frac{|u(x)-u(y)|^q}{|x-y|^q}dy\right)dx=C_{q,N}\|\nabla u\|^q_{L^q(\Omega)};
\end{equation}
and for $u\in BV(\Omega)$ we have
\begin{equation}
\label{eq:ineqality6}
\lim_{\epsilon\to 0^+}\int_{\Omega}\left(\int_{\Omega\cap B_{\epsilon}(x)}\frac{1}{\epsilon^N}\frac{|u(x)-u(y)|}{|x-y|}dy\right)dx=C_{1,N}\|D u\|(\Omega),
\end{equation}
where $C_{q,N},C_{1,N}$ are dimensional constants.

In \cite{P}, Poliakovsky investigated the limit
\begin{equation}
\label{eq:inequality7}
\lim_{\epsilon\to 0^+}\int_{\Omega}\left(\int_{\Omega\cap B_{\epsilon}(x)}\frac{1}{\epsilon^N}\frac{|u(x)-u(y)|^q}{|x-y|}dy\right)dx,
\end{equation}
for $1<q<\infty$,
$\Omega\subset\mathbb{R}^N$ an open set with bounded Lipschitz boundary and
$u\in BV(\Omega,\mathbb{R}^d)\cap L^\infty(\Omega,\mathbb{R}^d)$. The space $BV^q(\Omega,\mathbb{R}^d)$ was also considered in \cite{P}:
we say that $u\in BV^q(\Omega,\mathbb{R}^d)$ if and only if $u\in L^q(\Omega,\mathbb{R}^d)$ and
\begin{equation}
\limsup_{\epsilon\to 0^+}\int_{\Omega}\left(\int_{\Omega\cap B_{\epsilon}(x)}\frac{1}{\epsilon^N}\frac{|u(x)-u(y)|^q}{|x-y|}dy\right)dx<\infty.
\end{equation}
The limit
$\eqref{eq:inequality7}$ is obtained by replacing $\frac{|u(x)-u(y)|}{|x-y|}$ in $\eqref{eq:ineqality6}$ by $\frac{|u(x)-u(y)|^q}{|x-y|}$. In \cite{P} it was proved that the limit
$\eqref{eq:inequality7}$ is given in terms of the jump part of the distributional derivative of $u$ only, without the absolutely continuous and Cantor parts. Before describing it as a theorem recall the definition of one-sided approximate limits:

\begin{definition}(Approximate jump points, Definition 3.67 in \cite{AFP})
\label{def:approximate jump point} Let
$\Omega\subset\mathbb{R}^N$ be an open set, $u\in
L^1_{loc}(\Omega,\mathbb{R}^d)$ and $x\in \Omega$. We say that $x$
is an approximate jump point of $u$ if and only if there exist
$a,b\in \mathbb{R}^d$ and $\nu\in S^{N-1}$ such that $a\neq b$ and
\begin{equation}
\label{eq:one-sided approximate limitint} \lim_{\rho\to
0^+}\fint_{B^+_\rho(x,\nu)}|u(y)-a|dy=0,\quad \lim_{\rho\to
0^+}\fint_{B^-_\rho(x,\nu)}|u(y)-b|dy=0,
\end{equation}
where
\begin{equation}
B^+_\rho(x,\nu):=\left\{y\in B_\rho(x):(y-x)\cdot
\nu>0\right\},\quad B^-_\rho(x,\nu):=\left\{y\in
B_\rho(x):(y-x)\cdot \nu<0\right\}.
\end{equation}
The triple $(a,b,\nu)$, uniquely determined by $\eqref{eq:one-sided
approximate limitint}$ up to a permutation of $(a,b)$ and the change
of sign of $\nu$, is denoted by $(u^+(x),u^-(x),\nu_u(x))$. The set
of approximate jump points is denoted by $\mathcal{J}_u$.
\end{definition}

\begin{theorem}(Theorem 1.1 in \cite{P})
\\
\label{thm:Besov constant in terms of jump}
Let $\Omega\subset\R^N$ be an open set with bounded Lipschitz boundary and let $u\in BV(\Omega,\R^d)\cap L^\infty(\Omega,\R^d).$ Then for every $1<q<\infty$ we have $u\in BV^q(\Omega,\R^d)$ and
\begin{equation}
\label{eq:equality8}
C_N\int_{\mathcal{J}_u}|u^+(x)-u^-(x)|^q~d\Haus^{N-1}(x)=\lim_{\epsilon\to 0^+}\int_{\Omega}\left(\int_{\Omega\cap B_{\epsilon}(x)}\frac{1}{\epsilon^N}\frac{|u(x)-u(y)|^q}{|x-y|}dy\right)dx,
\end{equation}
where
\begin{equation}
\label{eq:definition of dimensional constant C_N}
C_N:=\frac{1}{N}\int_{S^{N-1}}|z_1|~d\Haus^{N-1}(z),\quad z:=(z_1,...,z_N).
\end{equation}
\end{theorem}

We are interested in the following question: is the formula
$\eqref{eq:equality8}$ valid for functions $u\in
BV^q(\Omega,\mathbb{R}^d)$, that do not belong to
$BV(\Omega,\mathbb{R}^d)\cap L^\infty(\Omega,\mathbb{R}^d)$? This
question has a meaning: we do not know whether or not the limit in the
right hand side of $\eqref{eq:equality8}$ exists for any function
$u\in BV^q(\Omega,\mathbb{R}^d)$, but the upper limit is finite for
such functions. In addition, the term on the left hand side of
$\eqref{eq:equality8}$ which includes the one-sided approximate
limits $u^+,u^-$ and the jump set $\mathcal{J}_u$ makes sense
because the jump set of functions $BV^q$ is countably
$(N-1)-$rectifiable as they are locally integrable functions. This
result is due to Giacomo Del Nin \cite{DelNin}:
\begin{theorem}
\label{thm:rectifiability of jump set}
Let $\Omega\subset\mathbb{R}^N$ be an open set and let $u\in L^1_{loc}(\Omega)$. Then the jump set $\mathcal{J}_u$ is countably
$(N-1)-$rectifiable.
\end{theorem}

Our answer to the above question is negative: the formula in
$\eqref{eq:equality8}$ is not true for every function $u\in
BV^q(\Omega,\mathbb{R}^d)$. Nevertheless, we do have the following theorem which is the main result of this paper:

\begin{theorem}
\label{thm:main result1}
Let $1\leq q<\infty$,
$\Omega\subset \mathbb{R}^N$ be an open set and $u\in BV^q(\Omega,\R^d)$. Then,
\begin{align}
\label{eq:inequality30}
C_N\int_{\mathcal{J}_u}|u^+(x)-u^-(x)|^qd\Haus^{N-1}(x)\leq \liminf_{\epsilon\to 0^+}\int_{\Omega}\left(\int_{\Omega\cap B_{\epsilon}(x)}\frac{1}{\epsilon^N}\frac{|u(x)-u(y)|^q}{|x-y|}dy\right)dx.
\end{align}
\end{theorem}

Theorem \ref{thm:main result1} will be proven in a more general context, as referred to in Theorem \ref{thm:main result,general formulation}. We will show in subsection \ref{subsec:Examples  of $BV^q$-functions for which the $q$-jump inequality is strict} that the inequality
$\eqref{eq:inequality30}$ can be strict.

\begin{remark}
\label{rem:Besov inequality in case of BV}
Using Theorem $\ref{thm:Besov constant in terms of jump}$ one can
prove by a truncation argument the validity of inequality \eqref{eq:inequality30} for functions
$u\in BV(\Omega,\R^d)\cap BV^q(\Omega,\R^d)$, where
$\Omega\subset\mathbb{R}^N$ is an open set with bounded Lipschitz boundary. More precisely, for $u=(u^1,...,u^d)\in BV(\Omega,\R^d)\cap BV^q(\Omega,\R^d)$ we have $u_l\in BV(\Omega,\R^d)\cap L^\infty(\Omega,\R^d)$, where $l\in \mathbb{N}$,  $u_l=(u^1_l,...,u^d_l)$ and
$u^i_l:=l\wedge(-l\vee u^i)$ for $1\leq i\leq d$; we get by Theorem $\ref{thm:Besov constant in terms of jump}$ for every $l\in\mathbb{N}$
\begin{equation}
C_N\int_{\mathcal{J}_{u_l}}|(u_l)^+(x)-(u_l)^-(x)|^q~d\Haus^{N-1}(x)=\lim_{\epsilon\to 0^+}\int_{\Omega}\left(\int_{\Omega\cap B_{\epsilon}(x)}\frac{1}{\epsilon^N}\frac{|u_l(x)-u_l(y)|^q}{|x-y|}dy\right)dx,
\end{equation}
and by taking the limit as $l\to \infty$, one can get $\eqref{eq:inequality30}$.
\end{remark}

Note that in Theorem
$\ref{thm:main result1}$, in comparison with Remark
$\ref{rem:Besov inequality in case of BV}$, we allow the case $q=1$,
we do not have any geometric assumptions on the boundary of $\Omega$,
or make the assumption that $u$ lies in $BV(\Omega,\mathbb{R}^d)$.

An importance of Theorem
$\ref{thm:main result1}$ is for Besov functions $B^{1/q}_{q,\infty}(\R^N,\R^d)$. Recall the definition of Besov space $B^s_{q,\infty}(\mathbb{R}^N,\mathbb{R}^d)$:
\begin{definition}
Let $1\leq q<\infty$ and $s\in (0,1)$. Define:
\begin{equation}
B^s_{q,\infty}(\R^N,\R^d):=\left\{u\in L^q(\R^N,\R^d):\sup_{\rho>0}\left(\sup_{|h|\leq \rho}\int_{\R^N}\frac{|u(x+h)-u(x)|^q}{\rho^{sq}}dx\right)<\infty\right\}.
\end{equation}
For an open set $\Omega\subset\R^N$, the local space $\left(B^{s}_{q,\infty}\right)_{loc}(\Omega,\R^d)$ is defined to be the set of all functions $u\in L^q_{loc}(\Omega,\R^d)$ such that for every compact $K\subset\Omega$ there exists a function $u_K\in B^s_{q,\infty}(\R^N,\R^d)$ such that $u_K(x)=u(x)$ for a.e. $x\in K$.
\end{definition}

The following proposition is due to Poliakovsky and it gives us a connection between Besov functions in $B^{1/q}_{q,\infty}$ and $BV^q-$functions.

\begin{proposition}(Proposition 1.1 in \cite{P})
For $1<q<\infty$ we have:
\begin{equation}
BV^q(\R^N,\R^d)=B^{1/q}_{q,\infty}(\R^N,\R^d).
\end{equation}
Moreover, for every open set $\Omega\subset\R^N$ we have
\begin{equation}
BV^q_{loc}(\Omega,\R^d)=\left(B^{1/q}_{q,\infty}\right)_{loc}(\Omega,\R^d),
\end{equation}
where the local space $BV^q_{loc}(\Omega,\R^d)$ is defined in a usual way (Definition \ref{def:definition of BV^q}).
\end{proposition}
For a general introduction to Besov spaces see for example \cite{Giovanni}.

\begin{remark}(Connection between $BV$ and $BV^1$ functions)
By the $BBM$ formula we have for any open set $\Omega\subset\R^N$ that $BV^1_{loc}(\Omega,\R^d)=BV_{loc}(\Omega,\R^d)$, and $BV^1(\Omega,\R^d)\subset BV(\Omega,\R^d)$. If $\Omega$ has bounded Lipschitz boundary, then $BV^1(\Omega,\R^d)=BV(\Omega,\R^d)$, and in particular $BV^1(\R^N,\R^d)=BV(\R^N,\R^d)$.
We do not know whether or not the equality $BV^1(\Omega,\R^d)=BV(\Omega,\R^d)$ holds for arbitrary open sets.
\end{remark}

\begin{remark}($q$-jump variation for Besov functions)
Let $u\in BV\left(\mathbb{R}^N,\mathbb{R}^d\right)$. Denote by $Du$ its distributional derivative and by
$D^ju$ the jump part of $Du$. The total variation $|D^ju|$ is a finite
Radon measure in $\mathbb{R}^N$ which can be expressed by the formula:
\begin{equation}
|D^ju|(B)=\int_{B\cap \mathcal{J}_u}|u^+(x)-u^-(x)|d\mathcal{H}^{N-1}(x)
\end{equation}
for Borel sets
$B\subset \mathbb{R}^N$ (see for example equation $(3.90)$ and Proposition 1.23 in \cite{AFP}).
Let us define a set function
\begin{equation}
|D^j_qu|(B):=\left(\int_{B\cap \mathcal{J}_u}|u^+(x)-u^-(x)|^qd\mathcal{H}^{N-1}(x)\right)^{\frac{1}{q}}
\end{equation}
for Borel sets
$B\subset \mathbb{R}^N$.
As a consequence of Theorem $\ref{thm:main result1}$ we obtain an analogous finite Radon measure $|D^j_qu|^q$ in $\mathbb{R}^N$ for Besov functions $u\in B_{q,\infty}^{1/q}\left(\mathbb{R}^N,\mathbb{R}^d\right)$. We call the measure $|D^j_qu|^q$ the \textbf{$q-$jump variation of $u$}.
\end{remark}

Our general conjecture is that $BV^q-$functions with $q>1$ inherit
at least part of the fine properties of $BV-$functions. This
conjecture was previously partially supported by the results of
De-Lellis, Otto \cite{DO} and Ghiraldin, Lamy \cite{GL}. Indeed, as
a combination of results from \cite{DO,GL}, one can get fine
properties of divergence free vector fields $u\in BV^3(\Omega,
S^1)$, where $\Omega\subset\R^2$ is a bounded open domain and
$S^1\subset\R^2$ is the unit circle. More precisely, one can get the
following simple combination of the main results in \cite{DO} and
\cite{GL}:
\begin{theorem}\label{ghjggjgj}
Let $\Omega\subset\R^2$ be a open set, and $u\in
BV^3_{loc}(\Omega,\R^2)$, satisfying
\begin{equation}
\big|u(x)\big|^2=1\quad\quad\text{for $\mathcal{L}^N$ a.e.}\quad
x\in\Omega
\end{equation}
and
\begin{equation}
\Div u=0\quad\quad\text{in the sense of distributions in}\quad
\Omega\,,
\end{equation}
which means that $\int_{\Omega}\nabla \phi(x)\cdot u(x)dx=0$ for every
$\phi\in C^\infty_c(\Omega)$. Then, the jump set $\mathcal{J}_u$ is
countably $1-$rectifiable set, oriented with the jump vector
$\nu_u(x)$, and moreover, we have
\begin{equation}
\lim\limits_{\e\to 0^+}\fint_{B_{\e}(x)} \left|u(
z)-u_{B_{\e}(x)}\right|dz=0\quad\quad\text{for
$\mathcal{H}^1$ a.e.}\quad x\in\Omega\setminus \mathcal{J}_u\,,
\end{equation}
where $u_{B_\e(x)}:=\fint_{B_\e(x)}u(w)dw$.
I.e. we have $\mathcal{H}^{1}\big(\mathcal{S}'_u\setminus
\mathcal{J}_u\big)=0$ (see Definition \ref{def:approximate limitsint} and Remark \ref{rem:equivalence for points in S",S'}).
\end{theorem}

The common belief is that Theorem \ref{ghjggjgj} could be improved towards obtaining $\mathcal{H}^{1}\big(\mathcal{S}_u\setminus
\mathcal{J}_u\big)=0$ instead of
$\mathcal{H}^{1}\big(\mathcal{S}'_u\setminus \mathcal{J}_u\big)=0$
(see Definition \ref{def:approximate limit} below), or at least
towards deducing that $\mathcal{S}_u$ is $\mathcal{H}^{1}$
$\sigma$-finite. Recently, in \cite{LaMr}, Lamy and Marconi
established that within the framework of Theorem \ref{ghjggjgj}, the set
$\mathcal{S}_u$ is of dimension at most one (this is slightly weaker
than being $\mathcal{H}^{1}$ $\sigma$-finite).

On the other hand, in the general settings of $BV^q_{loc}$
functions, we were able to prove the following theorem:
\begin{theorem}
\label{thm:fine property of Besov functions} Let $\Omega\subset\R^N$
be an open set and $q\in[1,\infty)$. Let $u\in
BV^q_{loc}(\Omega,\R^d)$. Then, there exists a $\mathcal{H}^{N-1}$
$\sigma$-finite Borel set $S$, such that we have
\begin{equation}
\liminf\limits_{\e\to 0^+}\fint_{B_{\e}(x)} \left|u(
z)-u_{B_{\e}(x)}\right|^qdz=0\quad\quad\text{for
$\mathcal{H}^{N-1}$ a.e.}\quad x\in\Omega\setminus S\,.
\end{equation}

I.e. $\mathcal{S}''_u$ is $\mathcal{H}^{N-1}$ $\sigma$-finite Borel
set (see Definition \ref{def:generalized approximate limit-oscillation points} and Remark \ref{rem:equivalence for points in S",S'}).
\end{theorem}

Theorem \ref{thm:fine property of Besov functions} will be proven in a more general setting; refer to Theorem \ref{thm:sigma finiteness of limiting average with respect to Hausdorff measure}.

Recall the definition of approximate limits:
\begin{definition}(Approximate limit)
\label{def:approximate limit} Let $\Omega\subset\mathbb{R}^N$ be an
open set and $u\in L^1_{loc}(\Omega,\mathbb{R}^d)$. We say that $u$
has approximate limit at $x\in \Omega$ if and only if there exists
$z\in \mathbb{R}^d$ such that
\begin{equation}
\label{eq:approximate limit} \lim_{\rho\to
0^+}\fint_{B_\rho(x)}|u(y)-z|dy=0.
\end{equation}
The set $\mathcal{S}_u$ of points where this property does not hold
is called the approximate discontinuity set. For any $x\in \Omega$
the vector $z$, uniquely determined by $\eqref{eq:approximate
limit}$, is called the approximate limit of $u$ at $x$ and denoted by
$\tilde{u}(x)$.
\end{definition}
Recall Federer-Vol'pert Theorem for
$BV$ functions (Theorem 3.78 in \cite{AFP}):
\begin{theorem}\label{ghjggjgji90}
Let $\Omega\subset\R^N$ be an open set, and $u\in
BV_{loc}(\Omega,\R^d)$. Then, the jump set $\mathcal{J}_u$ is
countably $(N-1)-$rectifiable set, oriented with the jump vector
$\nu_u(x)$, and moreover, we have
$\mathcal{H}^{N-1}\big(\mathcal{S}_u\setminus \mathcal{J}_u\big)=0$.
\end{theorem}
In Section \ref{sec:examples of pathological behavior of The Sobolev and $BV^q$ functions} we give an example (Example \ref{ex:2}) for a function $u\in W^{\frac{1}{2},2}(\Omega)\subset BV^2(\Omega)$ such that $\mathcal{H}^{N-1}\big(\mathcal{S}_u\setminus \mathcal{J}_u\big)>0$.

We prove also some properties of Sobolev functions. We prove that:

\begin{theorem}
\label{thm:fine property of fractional Sobolev functions(1)}
Let $\Omega\subset\R^N$ be an open set. Let $q\in[1,\infty)$, $r\in(0,1)$ be
such that $rq\leq N$. Let $u\in W^{r,q}_{loc}(\Omega,\R^d)$.
Then,
\begin{align}
&\lim\limits_{\rho\to 0^+}\fint_{B_{\rho}(x)} \left|u(
z)-u_{B_{\rho}(x)}\right|^qdz\nonumber
\\
&=\lim_{\rho\to
0^+}\fint_{B_{\rho}(x)}\Bigg\{\fint_{B_{\rho}(x)}
\Big|u\big(z\big)-u(y)\Big|^qdy\Bigg\}dz=0
\quad\quad\text{for}\quad\mathcal{H}^{N-rq}\quad \text{a.e.}\quad
x\in\Omega\,.
\end{align}

I.e. we have $\mathcal{H}^{N-rq}\big(\mathcal{S}'_u\big)=0$ (see Definition \ref{def:approximate limitsint} and Remark \ref{rem:equivalence for points in S",S'}).
\end{theorem}

The following theorem is a generalization of fine properties of Sobolev functions $W^{1,p}(\Omega)$, $p>1$. It is known that for $u\in W^{1,p}(\Omega)$, we have
$\lim_{\rho\to
0^+}\fint_{B_{\rho}(x)}|u(y)-u_{B_\rho(x)}|^pdy=0$ for $\cp_p$
a.e. $x\in \Omega$, where $\cp_p$ is the $p-$capacity (see proof in
\cite{Ziemer} Theorem 3.3.2). We show that the same result holds when  replacing the measure $\cp_p$ by the measure $\mathcal{H}^{N-p}$, in case $1<p\leq N$. Recall that for $1\leq p<N$ there exists a constant $C$, depending only on $p$ and $N$, such that  $\cp_p(E)\leq C\mathcal{H}^{N-p}(E)$ for every $E\subset\R^N$. Recall also that $\cp_1$ and $\mathcal{H}^{N-1}$ share negligible sets, but this is not the case for $p>1$, for example, if $1<p<N$, then any set $E$ of $\mathcal{H}^{N-p}$-finite measure has $\cp_p$ measure zero (see Remark 4.16 in \cite{Kinnunen}).

\begin{theorem}
\label{thm:fine properties of Sobolev function(1)}
Let $\Omega\subset \R^N$ be an open set, $1< p\leq N$ and $u\in W^{1,p}(\Omega)$. Then,
\begin{align}
\label{eq:fine property10}
\lim_{\rho\to
0^+}\fint_{B_{\rho}(x)}|u(y)-u_{B_\rho(x)}|^pdy=0
\quad\quad\text{for}\;\;\mathcal{H}^{N-p}\,\,\text{a.e.}\,\,x\in
\Omega\,.
\end{align}
I.e. we have $\mathcal{H}^{N-p}\big(\mathcal{S}'_u\big)=0$ (see Definition \ref{def:approximate limitsint} and Remark \ref{rem:equivalence for points in S",S'}).
\end{theorem}
Recall that Theorem \ref{thm:fine properties of Sobolev
function(1)} is true also if the averages
$u_{B_\rho(x)}$ are replaced by the precise representative
$u^*(x):=\lim_{\rho\to 0^+}u_{B_\rho(x)}$ and the Hausdorff measure $\mathcal{H}^{N-p}$  is replaced by the capacity $\cp_p$ \cite{AdH,EG,Ziemer}. However, in section \ref{sec:examples of pathological behavior of The Sobolev
and $BV^q$ functions} we give examples that show that Theorem \ref{thm:fine properties of Sobolev function(1)} is not true for  $u^*$ together with $\mathcal{H}^{N-p}$ (Examples \ref{ex:4},\ref{ex:3}).

In addition, we establish a Lusin approximation theorem for fractional Sobolev functions by H{\"o}lder continuous functions:
\begin{theorem}
\label{thm:Lusin approximation for fractional Sobolev functions(1)}
Let $\Omega\subset \R^N$ be an open set and $q\in[1,\infty)$, $r\in(0,1)$. Let $u\in W^{r,q}_{loc}(\Omega,\R^d)$
and let $K\subset \Omega$ be a compact set. Then for every $\e>0$ there exists a compact set
$B\subset K$ such that
$\Leb^N\left(K\setminus B\right)<\e$ and
$u\in C^{0,r}(B,\R^d)$.
\end{theorem}

Theorem \ref{thm:Lusin approximation for fractional Sobolev functions(1)} will be proven for a broader class of functions, as referred to in Theorem \ref{thm:Lusin approximation for A(r,q) functions}.

\section{Oscillation blow-ups}
\begin{definition}
\label{def:oscillation blow-up}
(Oscillation blow-up)
Let $\Omega\subset \R^N$ be an open set and $u\in L^1_{loc}(\Omega,\R^d)$. We say that $x\in \Omega$ is an oscillation blow-up point of $u$ if there exists $u_x\in L^1(B_1(0))$ such that
\begin{equation}\label{hhjjhjhjhjh}
\lim_{\rho\to
0^+}\Bigg(\inf\limits_{c\in\mathbb{R}^d}\fint_{B_1(0)}|u(x+\rho y)-u_x(y)-c|dy\Bigg)=0.
\end{equation}
We call the function $u_x$ an oscillation blow-up of $u$ at $x$.
\end{definition}
Oscillation blow-up is unique up to adding a constant:

\begin{proposition}
\label{prop:uniqueness of oscillation blow-up up to constant}
(Uniqueness of oscillation blow-up up to adding a constant)
Let $\Omega\subset \R^N$ be an open set, $u\in
L^1_{loc}(\Omega,\R^d)$. Assume $x\in \Omega$ is such that there exist a function $f\in L^1(B_1(0))$ and a sequence $\rho_i>0$ converging to zero as $i\to\infty$ such that
\begin{equation}
\label{eq:oscillation blow-up w.r.t a sequence}
\lim_{i\to \infty}\Bigg(\inf\limits_{c\in\mathbb{R}^d}\fint_{B_1(0)}|u(x+\rho_i
y)-f(y)-c|dy\Bigg)=0.
\end{equation}
Then, the property \eqref{eq:oscillation blow-up w.r.t a sequence} determines the function $f$ a.e. uniquely up to adding a constant, which means that, if $g\in L^1(B_1(0))$, then $g$ also satisfies \eqref{eq:oscillation blow-up w.r.t a sequence} if and only if there exists a constant $c_0\in \R^d$ such that $f(y)=g(y)+c_0$ for a.e. $y\in B_1(0)$. In particular, if $u_x$ is an oscillation blow-up of $u$ at $x$, then it is a.e. uniquely determined up to adding a constant.
\end{proposition}

\begin{proof}
Assume \eqref{eq:oscillation blow-up w.r.t a sequence},
and let $c_0\in\R^d$. Define $g\in
L^1(B_1(0))$ by $g(y):=f(y)+c_0$. Then by \eqref{eq:oscillation blow-up w.r.t a sequence} we
deduce
\begin{equation}\label{hhjjhjhjhjh1}
\lim_{i\to
\infty}\Bigg(\inf\limits_{c\in\mathbb{R}^d}\fint_{B_1(0)}|u(x+\rho_i
y)-g(y)-c|dy\Bigg)=\lim_{i\to
\infty}\Bigg(\inf\limits_{c\in\mathbb{R}^d}\fint_{B_1(0)}|u(x+\rho_i
y)-g(y)-c_0-c|dy\Bigg)=0.
\end{equation}

Conversely, assume $f,g\in L^1(B_1(0))$ are such that
\begin{equation}
\label{hhjjhjhjhjh2}
\lim_{i\to
\infty}\Bigg(\inf\limits_{c_1\in\mathbb{R}^d}\fint_{B_1(0)}|u(x+\rho_i
y)-f(y)-c_1|dy\Bigg)=\lim_{i\to
\infty}\Bigg(\inf\limits_{c_2\in\mathbb{R}^d}\fint_{B_1(0)}|u(x+\rho_i
y)-g(y)-c_2|dy\Bigg)=0.
\end{equation}
Then, by the triangle inequality, we deduce
\begin{multline}\label{hhjjhjhjhjh3}
0=\lim_{i\to
\infty}\Bigg(\inf\limits_{c_1\in\mathbb{R}^d}\fint_{B_1(0)}|u(x+\rho_i
y)-f(y)-c_1|dy\Bigg)+\lim_{i\to
\infty}\Bigg(\inf\limits_{c_2\in\mathbb{R}^d}\fint_{B_1(0)}|u(x+\rho_i
y)-g(y)-c_2|dy\Bigg)
\\
=\lim_{i\to
\infty}\Bigg\{\inf\limits_{c_1,c_2\in\mathbb{R}^d}\bigg(\fint_{B_1(0)}|u(x+\rho_i
y)-f(y)-c_1|dy+\fint_{B_1(0)}|u(x+\rho_i y)-g(y)-c_2|dy\bigg)\Bigg\}\\
\geq\inf\limits_{c_1,c_2\in\mathbb{R}^d}\fint_{B_1(0)}|f(y)-g(y)+c_1-c_2|dy
\\
=\inf\limits_{c\in\mathbb{R}^d}\fint_{B_1(0)}|f(y)-g(y)+c|dy\,,
\end{multline}
so that there exists $c_0\in\R^d$ such that $g(y)=f(y)+c_0$ for
a.e. $y\in B_1(0)$.
\end{proof}

\begin{definition}(Approximate limit-oscillation points)
\label{def:approximate limitsint} Let $\Omega\subset\mathbb{R}^N$ be
an open set and $u\in L^1_{loc}(\Omega,\mathbb{R}^d)$. We say that
$x\in \Omega$ is an approximate limit-oscillation point of $u$ if
\begin{equation}
\label{eq:approximate limitsint}
 \lim_{\rho\to
0^+}\Bigg(\inf\limits_{c\in\mathbb{R}^d}\fint_{B_\rho(x)}|u(y)-c|dy\Bigg)=0.
\end{equation}
The set $\mathcal{S}'_u$ of points where this property does not hold
is called the reduced approximate discontinuity set. Obviously
$\mathcal{S}'_u\subset\mathcal{S}_u$. In terms of oscillation blow-up: $x\notin \mathcal{S}'_u$ if and only if $u$ has a constant oscillation blow-up at $x$.
\end{definition}

\begin{definition}(Generalized approximate limit-oscillation points)
\label{def:generalized approximate limit-oscillation points} Let $\Omega\subset\mathbb{R}^N$
be an open set and $u\in L^1_{loc}(\Omega,\mathbb{R}^d)$. We say that
$x\in \Omega$ is a generalized approximate limit-oscillation point of $u$ if
\begin{equation}
\label{eq:approximate limitsintw} \liminf_{\rho\to
0^+}\Bigg(\inf\limits_{c\in\mathbb{R}^d}\fint_{B_\rho(x)}|u(y)-c|dy\Bigg)=0.
\end{equation}
The set $\mathcal{S}''_u$ of points where this property does not
hold is called the restricted approximate discontinuity set.
Obviously
$\mathcal{S}''_u\subset\mathcal{S}'_u\subset\mathcal{S}_u$.
\end{definition}

\begin{remark}
\label{rem:equivalence for points in S",S'}
Let $\Omega\subset\mathbb{R}^N$ be an open set and $u\in
L^1_{loc}(\Omega,\mathbb{R}^d)$. Then, for every $x\in
\Omega$ we have
\begin{multline}
\lim_{\rho\to
0^+}\Bigg(\inf\limits_{c\in\mathbb{R}^d}\fint_{B_\rho(x)}|u(y)-c|dy\Bigg)=0\quad\text{if
and only if}\quad 
\\
\lim_{\rho\to
0^+}\fint_{B_\rho(x)}|u(y)-u_{B_\rho(x)}|dy=0,
\end{multline}
and
\begin{multline}
\liminf_{\rho\to
0^+}\Bigg(\inf\limits_{c\in\mathbb{R}^d}\fint_{B_\rho(x)}|u(y)-c|dy\Bigg)=0\quad\text{if
and only if}\quad 
\\
\liminf_{\rho\to
0^+}\fint_{B_\rho(x)}|u(y)-u_{B_\rho(x)}|dy=0.
\end{multline}

\end{remark}

\begin{definition}(Generalized approximate jump points)
\label{def:approximate jump pointsgint} Let
$\Omega\subset\mathbb{R}^N$ be an open set, $u\in
L^1_{loc}(\Omega,\mathbb{R}^d)$ and $x\in \Omega$. We say that $x$
is a generalized approximate jump point of $u$ if and only if there
exist $h\in \mathbb{R}^d\setminus\{0\}$ and $\nu\in S^{N-1}$ such
that
\begin{equation}
\label{eq:one-sided approximate limitgint} \lim_{\rho\to
0^+}\Bigg(\inf\limits_{c\in\mathbb{R}^d}\bigg\{\fint_{B^+_\rho(x,\nu)}|u(y)-h-c|dy+\fint_{B^-_\rho(x,\nu)}|u(y)-c|dy\bigg\}\Bigg)=0,
\end{equation}
where
\begin{equation}
B^+_\rho(x,\nu):=\left\{y\in B_\rho(x):(y-x)\cdot
\nu>0\right\},\quad B^-_\rho(x,\nu):=\left\{y\in
B_\rho(x):(y-x)\cdot \nu<0\right\}.
\end{equation}
The couple $(h,\nu)$, uniquely determined by $\eqref{eq:one-sided
approximate limitgint}$ up to a change of sign, is denoted by
$(u^{j'}(x),\nu'_u(x))$ (see Proposition \ref{hjhjhjgjggj}). The set of generalized approximate jump points is denoted
by $\mathcal{J}'_u$. Notice that
$\mathcal{J}_u\subset
\mathcal{J}'_u$. Moreover, for $x\in \mathcal{J}_u$ we have
$\big(u^{j'}(x),\nu'_u(x)\big)=\pm\big((u^+(x)-u^-(x)),\nu_u(x)\big)$.
\end{definition}

\begin{proposition}\label{hjhjhjgjggj}
Let $\Omega\subset \R^N$ be an open set, $u\in
L^1_{loc}(\Omega,\R^d)$ and $x\in \Omega$. For $k=1,2$ let
$h_k\in\R^d\setminus\{0\}$ and $\vec \nu_k\in S^{N-1}$ and consider
\begin{equation}\label{hihhihi}
w_{k}(y)=
\begin{cases}
h_k,\quad &\vec\nu_k\cdot y>0\\
0,\quad &\vec\nu_k\cdot y<0\\
\end{cases}.
\end{equation}
Assume for $k=1,2$
\begin{equation}\label{hhjjhjhjhjhojojo11}
\lim_{\rho\to
0^+}\Bigg(\inf\limits_{c\in\mathbb{R}^d}\fint_{B_1(0)}|u(x+\rho
y)-w_{k}(y)-c|dy\Bigg)=0.
\end{equation}
Then $(h_1,\vec \nu_1)=\pm(h_2,\vec \nu_2)$.
\end{proposition}
\begin{proof}
By Proposition \ref{prop:uniqueness of oscillation blow-up up to constant} there exists $c_0\in\R^d$ such that for a.e.
$y\in B_1(0)$ we have
\begin{equation}\label{hhjjhjhjhjhojojo}
w_2(y)-w_1(y)=c_0.
\end{equation}
Assume by contradiction that $\vec \nu_1\neq\pm \vec \nu_2$.
Thus, $|\vec\nu_1\cdot\vec\nu_2|<1$, so that
$y_1:=\frac{1}{4}\big(\vec \nu_2-\vec \nu_1\big)\in B_1(0)$ and
$y_2:=-\frac{1}{4}\big(\vec \nu_2+\vec \nu_1\big)\in B_1(0)$ satisfy
\begin{align}\label{hhjjhjhjhjhojojouiouiui}
y_1\cdot\vec \nu_1<0\quad\text{and}\quad y_1\cdot\vec \nu_2>0\\
\label{hhjjhjhjhjhojojouiouiuiuiuiy} y_2\cdot\vec
\nu_1<0\quad\text{and}\quad y_2\cdot\vec \nu_2<0.
\end{align}
In other words, the open sets $A_1,A_2$ given by
\begin{align}\label{hhjjhjhjhjhojojouiouiuiuiuiuy}
A_1:=\Big\{y\in B_1(0)\;:\;
y\cdot\vec \nu_1<0\quad\text{and}\quad y\cdot\vec \nu_2>0\Big\}\\
\label{hhjjhjhjhjhojojouiouiuigujgj} A_2:=\Big\{y\in
B_1(0)\;:\;y\cdot\vec \nu_1<0\quad\text{and}\quad y\cdot\vec
\nu_2<0\Big\}
\end{align}
are not empty. However,
\begin{equation}\label{hihhihihh}
w_{2}(y)-w_{1}(y)=
\begin{cases}
h_2,\quad &y\in A_1\\
0,\quad &y\in A_2\\
\end{cases}.
\end{equation}
It contradicts \er{hhjjhjhjhjhojojo}, since $h_2\neq 0$.

So necessarily $\vec \nu_1=\pm \vec \nu_2$. In the case $\vec
\nu_1=\vec \nu_2$, by \er{hihhihi} we have
\begin{equation}\label{hihhihiiiy}
w_{2}(y)-w_1(y)=
\begin{cases}
h_2-h_1,\quad &\vec\nu_2\cdot y>0\\
0,\quad &\vec\nu_2\cdot y<0\\
\end{cases}
\end{equation}
and thus, using \er{hhjjhjhjhjhojojo} we deduce $h_2=h_1$, so that
$(h_1,\vec \nu_1)=(h_2,\vec \nu_2)$. Otherwise, we have $\vec
\nu_1=-\vec \nu_2$, and by \er{hihhihi} we have
\begin{equation}\label{hihhihiiiyIYYY}
w_{2}(y)-w_1(y)=
\begin{cases}
h_2,\quad &\vec\nu_2\cdot y>0\\
-h_1,\quad &\vec\nu_2\cdot y<0\\
\end{cases}
\end{equation}
and thus, using \er{hhjjhjhjhjhojojo} we deduce $h_2=-h_1$, so that
$(h_1,\vec \nu_1)=-(h_2,\vec \nu_2)$. It completes the proof.
\end{proof}

\begin{proposition}
\label{prop:inclusion of generalized jump set in the S''}
Let
$\Omega\subset\mathbb{R}^N$ be an open set and let $u\in
L^1_{loc}(\Omega,\mathbb{R}^d)$. Then
$$
\mathcal{J}_u\subset
\mathcal{J}'_u\subset\mathcal{S}''_u\subset \mathcal{S}'_u\subset
\mathcal{S}_u.
$$
\end{proposition}

\begin{proof}
The inclusions $\mathcal{J}_u\subset
\mathcal{J}'_u$, $\mathcal{S}''_u\subset \mathcal{S}'_u\subset
\mathcal{S}_u$ are immediate from definitions. We prove
$\mathcal{J}'_u\subset\mathcal{S}''_u$. If $x\in \mathcal{J}'_u$, then by Definition \ref{def:approximate jump pointsgint} and Definition \ref{def:oscillation blow-up} the function $u$ has a non-constant oscillation blow-up $v$ at $x$:
\begin{equation}
\label{eq:definition of the blow-up v}
v(y):=
\begin{cases}
u^{j'}(x),\quad &\nu'_u(x)\cdot y>0\\
0,\quad &\nu'_u(x)\cdot y<0\\
\end{cases},\quad u^{j'}(x)\neq 0,
\end{equation}
and
\begin{equation}
\label{eq:equation1 in Remark}
\lim_{\rho\to
0^+}\Bigg(\inf\limits_{z\in\mathbb{R}^d}\fint_{B_{\rho}(x)}|u(y)-v(y)-z|dy\Bigg)=0.
\end{equation}
If, by contradiction,
$x\notin \mathcal{S}''_u$, then there exists a sequence $\rho_i>0$ converging to zero such that
\begin{equation}
\label{eq:equation2 in Remark}
\lim_{i\to
\infty}\Bigg(\inf\limits_{z\in\mathbb{R}^d}\fint_{B_{\rho_i}(x)}|u(y)-z|dy\Bigg)=0.
\end{equation}
\eqref{eq:equation1 in Remark} and \eqref{eq:equation2 in Remark} contradict the uniqueness of oscillation blow-ups up to a constant (Proposition \ref{prop:uniqueness of oscillation blow-up up to constant}).
\end{proof}

\begin{remark}
\label{rem:oscillation blow-up in case of jump and generalized jump points}
Let $\Omega\subset\R^N$ be an open set and $u\in L^1_{loc}(\Omega,\R^d)$. In case $x\in \mathcal{J}_u$ we have that
\begin{equation}
u_x(y)=
\begin{cases}
u^+(x),\quad &\nu_u(x)\cdot y>0\\
u^-(x),\quad &\nu_u(x)\cdot y<0\\
\end{cases}
\end{equation}
is an oscillation blow-up of $u$ at $x$.

In case $x\in \mathcal{J}'_u $
we have that
\begin{equation}
\label{eq:jeneralized jump}
u_x(y)=
\begin{cases}
u^{j'}(x),\quad &\nu'_u(x)\cdot y>0\\
0,\quad &\nu'_u(x)\cdot y<0\\
\end{cases}
\end{equation}
is an oscillation blow-up of $u$ at $x$.
\end{remark}

\begin{proposition}
\label{prop:oscillation blow-up of sum equals to the sum of the oscillation blow-up}
(Oscillation blow-up of sum equals to the sum of the oscillation blow-up)
Let
$\Omega\subset\mathbb{R}^N$ be an open set, $u,B\in
L^1_{loc}(\Omega,\mathbb{R}^d)$. Assume $x\in \Omega$ is an oscillation blow-up point of $u,B$. Then, $x$ is an oscillation blow-up point of the sum $u+B$ and $(u+B)_x=u_x+B_x$.
\end{proposition}

\begin{proof}
It follows that
\begin{equation}
\lim_{\rho\to
0^+}\Bigg(\inf\limits_{c\in\mathbb{R}^d}\fint_{B_1(0)}|u(x+\rho y)-u_x(y)-c|dy\Bigg)=0,
\end{equation}

\begin{equation}
\lim_{\rho\to
0^+}\Bigg(\inf\limits_{m\in\mathbb{R}^d}\fint_{B_1(0)}|B(x+\rho y)-B_x(y)-m|dy\Bigg)=0.
\end{equation}

Therefore,
\begin{align}
0=&\lim_{\rho\to
0^+}\Bigg(\inf_{(c,m)\in\mathbb{R}^d\times\R^d}\Bigg\{\fint_{B_1(0)}|u(x+\rho y)-u_x(y)-c|dy\nonumber
\\
&+\fint_{B_1(0)}|B(x+\rho y)-B_x(y)-m|dy\Bigg\}\Bigg)\nonumber
\\
&=\lim_{\rho\to
0^+}\Bigg(\inf_{(c,m)\in\mathbb{R}^d\times\R^d}\Bigg\{\fint_{B_1(0)}|u(x+\rho y)-u_x(y)-c|dy\nonumber
\\
&+\fint_{B_1(0)}|B(x+\rho y)-B_x(y)-(m-c)|dy\Bigg\}\Bigg)\nonumber
\\
&\geq\lim_{\rho\to
0^+}\Bigg(\inf_{(c,m)\in\mathbb{R}^d\times\R^d}\Bigg\{\fint_{B_1(0)}\left|\left(u(x+\rho y)+B(x+\rho y)\right)-\left(u_x(y)+B_x(y)\right)-m\right|dy\Bigg\}\Bigg)\nonumber
\\
&=\lim_{\rho\to
0^+}\Bigg(\inf_{m\in\mathbb{R}^d}\Bigg\{\fint_{B_1(0)}\left|\left(u(x+\rho y)+B(x+\rho y)\right)-\left(u_x(y)+B_x(y)\right)-m\right|dy\Bigg\}\Bigg).
\end{align}
Thus, $u_x+B_x$ is an oscillation blow-up of $u+B$ at $x$.
\end{proof}

\begin{proposition}
\label{prop:convergence of the rescaled family gamma}
($L^1$-convergence of the rescaled family $u_{x,t}$ for $x\in \mathcal{J}'_u$)
Let
$\Omega\subset\mathbb{R}^N$ be an open set and let $u\in
L^1_{loc}(\Omega,\mathbb{R}^d)$. For each $x\in \mathcal{J}'_u$ let
us define a function:
\begin{equation}
w'_x:\mathbb{R}^N\to \mathbb{R}^d,\quad  w'_x(y):=
\begin{cases}
u^{j'}(x),\quad &\nu'_u(x)\cdot y>0\\
0,\quad &\nu'_u(x)\cdot y<0\\
\end{cases}.
\end{equation}
For each $x\in \Omega$ and $t>0$ let us define a function
$u_{x,t}(y):=u(x+t y),y\in \frac{\Omega-x}{t}$. Then:
\\
1. For every $x\in \mathcal{J}'_u$ and every compact
$K\subset\mathbb{R}^N$ we have
\begin{equation}
\lim_{t\to
0^+}\Bigg(\inf\limits_{z\in\mathbb{R}^d}\int_{K}\big|u_{x,t}(y)-w'_x(y)-z\big|dy\Bigg)=0.
\end{equation}
\\
2. For every $x\in \mathcal{J}'_u$, define the family of functions
\begin{equation}
\gamma_{x,t}(y,z):=\left|u_{x,t}(y)-u_{x,t}(z)\right|.
\end{equation}
Then, for every $\alpha,\beta>0$ the family of functions
$y\mapsto\gamma_{x,t}(\alpha y,-\beta y)$ converges in
$L^1_{loc}(\mathbb{R}^N)$ to the constant function $\left|u^{j'}(x)\right|$
as $t\to 0^+$. Moreover, for every $q>0$
\begin{align}
&\liminf\limits_{t\to 0^+}\fint_{B_{1}(0)\times
B_{1}(0)}\Big|\gamma_{x,t}(y,z)\Big|^qd(\mathcal{L}^N\times\mathcal{L}^N)(y,z)\geq
\frac{1}{2}\Big|u^{j'}(x)\Big|^q.
\end{align}
\end{proposition}

\begin{proof}
1. Let $K\subset\mathbb{R}^N$ be a nonempty compact set such that $K\neq\{0\}$, and $x\in
\mathcal{J}'_u$. Denote
\begin{equation}
R:=\sup_{y\in K}|y|>0.
\end{equation}
It follows that
\begin{multline}
\lim_{t\to
0^+}\Bigg(\inf\limits_{z\in\mathbb{R}^d}\int_{K}\big|u_{x,t}(y)-w'_x(y)-z\big|dy\Bigg)\leq
\lim_{t\to
0^+}\Bigg(\inf\limits_{z\in\mathbb{R}^d}\int_{B_R(0)}\big|u_{x,t}(y)-w'_x(y)-z\big|dy\Bigg)
\\
=\lim_{t\to
0^+}\Bigg(\inf\limits_{z\in\mathbb{R}^d}\int_{B_R(0)}\big|u_{x,Rt}(y)-w'_x(y)-z\big|dy\Bigg)=
\frac{1}{R^N}\lim_{t\to
0^+}\Bigg(\inf\limits_{z\in\mathbb{R}^d}\int_{B_1(0)}\big|u_{x,t}(w)-w'_x(w/R)-z\big|dw\Bigg)
\\
=\frac{1}{R^N}\lim_{t\to
0^+}\Bigg(\inf\limits_{z\in\mathbb{R}^d}\frac{1}{t^N}\int_{B_t(x)}\big|u(y)-w'_x\left(\frac{y-x}{tR}\right)-z\big|dy\Bigg)
\\
=\frac{\alpha(N)}{R^N}\,\lim_{t\to
0^+}\Bigg(\inf\limits_{z\in\mathbb{R}^d}\bigg\{\fint_{B^+_t(x,\nu'_u(x))}|u(y)-u^{j'}(x)-z|dy+\fint_{B^-_t(x,\nu'_u(x))}|u(y)-z|dy\bigg\}\Bigg)=0.
\end{multline}

2. Let $K\subset\mathbb{R}^N$ be a nonempty compact set such that $K\neq\{0\}$, and $x\in
\mathcal{J}'_u$. Denote
\begin{equation}
R:=\sup_{y\in K}|y|>0.
\end{equation}
Then by item 1, using triangle inequality, for every
$\alpha,\beta>0$ we have
\begin{multline}
0=\lim_{t\to
0^+}\Bigg(\inf\limits_{z\in\mathbb{R}^d}\bigg\{\Big(\frac{1}{\alpha^N}+\frac{1}{\beta^N}\Big)\int_{B_{(\alpha+\beta)R}(0)}\big|u_{x,t}(y)-w'_x(y)-z\big|dy\bigg\}\Bigg)
\\
\geq\lim_{t\to
0^+}\Bigg(\inf\limits_{z\in\mathbb{R}^d}\bigg\{\frac{1}{\alpha^N}\int_{B_{\alpha
R}(0)}\big|u_{x,t}(y)-w'_x(y)-z\big|dy
+\frac{1}{\beta^N}\int_{B_{\beta
R}(0)}\big|u_{x,t}(y)-w'_x(y)-z\big|dy\bigg\}\Bigg)\\
= \lim_{t\to
0^+}\Bigg(\inf\limits_{z\in\mathbb{R}^d}\bigg\{\int_{B_{R}(0)}\big|u_{x,t}(\alpha
v)-w'_x(\alpha v)-z\big|dv +\int_{B_{R}(0)}\big|u_{x,t}(-\beta
v)-w'_x(-\beta v)-z\big|dv\bigg\}\Bigg)
\\
\geq \lim_{t\to
0^+}\Bigg(\int_{B_{R}(0)}\Big|u_{x,t}(\alpha
v)-u_{x,t}(-\beta v)-\left(w'_x(\alpha v)-w'_x(-\beta
v)\right)\Big|dv\Bigg)
\\
\geq \lim_{t\to
0^+}\Bigg(\int_{B_{R}(0)}\Big|\gamma_{x,t}(\alpha v,-\beta
v)-\left|w'_x(\alpha v)-w'_x(-\beta
v)\right|\Big|dv\Bigg)
\\
=\lim_{t\to
0^+}\Bigg(\int_{B_{R}(0)}\Big|\gamma_{x,t}(\alpha v,-\beta
v)-|u^{j'}(x)|\Big|dv\Bigg)
\\
\geq \lim_{t\to
0^+}\Bigg(\int_{K}\Big|\gamma_{x,t}(\alpha v,-\beta
v)-|u^{j'}(x)|\Big|dv\Bigg),
\end{multline}

so that for every $\alpha,\beta>0$ the family of functions
$\gamma_{x,t}(\alpha y,-\beta y)$ converges in
$L^1_{loc}(\mathbb{R}^N)$ to the constant function $|u^{j'}(x)|$
as $t\to 0^+$.

Finally, by item 1, using triangle inequality, we have
\begin{multline}
0=\lim_{t\to
0^+}\Bigg(\inf\limits_{c\in\mathbb{R}^d}2\fint_{B_{1}(0)}\big|u_{x,t}(y)-w'_x(y)-c\big|dy\Bigg)
\\
=\lim_{t\to
0^+}\Bigg(\inf\limits_{c\in\mathbb{R}^d}\bigg\{\fint_{B_{1}(0)}\big|u_{x,t}(y)-w'_x(y)-c\big|dy+\fint_{B_{1}(0)}\big|u_{x,t}(z)-w'_x(z)-c\big|dz\bigg\}\Bigg)
\\
=\lim_{t\to
0^+}\Bigg(\inf\limits_{c\in\mathbb{R}^d}\Bigg\{\fint_{B_{1}(0)\times
B_{1}(0)}\big|u_{x,t}(y)-w'_x(y)-c\big|d(\mathcal{L}^N\times\mathcal{L}^N)(y,z)
\\
+\fint_{B_{1}(0)\times
B_{1}(0)}\big|u_{x,t}(z)-w'_x(z)-c\big|d(\mathcal{L}^N\times\mathcal{L}^N)(y,z)\Bigg\}\Bigg)
\\
\geq\lim_{t\to
0^+}\Bigg(\fint_{B_{1}(0)\times
B_{1}(0)}\Big|\big(u_{x,t}(y)-u_{x,t}(z)\big)-\big(w'_x(y)-w'_x(z)\big)\Big|d(\mathcal{L}^N\times\mathcal{L}^N)(y,z)\Bigg)
\\
\geq \lim_{t\to 0^+}\Bigg(\fint_{B_{1}(0)\times
B_{1}(0)}\Big|\gamma_{x,t}(y,z)-\left|w'_x(y)-w'_x(z)\right|\Big|d(\mathcal{L}^N\times\mathcal{L}^N)(y,z)\Bigg).
\end{multline}
Thus, by lower semi-continuity of $L^q$ norm with respect to
convergence in $L^1$ and Fubini's theorem we have
\begin{align}
&\liminf\limits_{t\to 0^+}\int_{B_{1}(0)\times
B_{1}(0)}\Big|\gamma_{x,t}(y,z)\Big|^qd(\mathcal{L}^N\times\mathcal{L}^N)(y,z)\geq
\int_{B_{1}(0)\times
B_{1}(0)}\left|w'_x(y)-w'_x(z)\right|^qd(\mathcal{L}^N\times\mathcal{L}^N)(y,z)\nonumber
\\
&=2\int_{B^-_{1}(0,\nu'_u(x))}\int_{B^+_{1}(0,\nu'_u(x))}\Big|w'_x(y)-w'_x(z)\Big|^qdydz
=\frac{\alpha(N)^2}{2}\Big|u^{j'}(x)\Big|^q.
\end{align}
It completes the proof.
\end{proof}

In section \ref{sec:examples of pathological behavior of The Sobolev and $BV^q$ functions} we give several examples of the
pathological behavior of functions in $W^{\frac{1}{q},q}$ and $BV^q$
with $q>1$. We are able to build an example of $u\in W^{\frac{1}{2},2}(\Omega)\subset
BV^2(\Omega)$, where $\Omega\subset\R^N$ is an open set, such that
$\mathcal{H}^{N-1}\big(\mathcal{S}'_u\big)=0$ and
$\mathcal{H}^{N-1}\big(\mathcal{S}_u\big)>0$. Moreover, we are able
to build an example of $u\in  BV^2(\Omega)$ such that
$\mathcal{H}^{N-1}\big(\mathcal{J}_u\big)=0$ and
$\mathcal{H}^{N-1}\big(\mathcal{J}'_u\big)>0$.

\begin{remark}
\label{rem:rectifiability of generalized jump set}
In \cite{DelNin}, the following set was represented for $u\in L^1_{loc}(\Omega,\R^d)$, where $\Omega\subset\R^N$ is an open set:
\begin{equation}
\Sigma_u:=\left\{x\in \Omega:u_{x,r}\to v\,\, \text{in}\,\, L^1(B_1(0))\,\text{as $r\to 0^+$,\,\,\text{for some $v$ not constant}}\right\}.
\end{equation}
Here $u_{x,r}(z):=u(x+rz)$. The set
$\Sigma_u$ is the set of all points in the domain of $u$ at which there exists a non-constant blow-up. It was proved that $\Sigma_u$ is countably $(N-1)-$rectifiable and $\mathcal{J}_u\subset\Sigma_u$.

As was explained in \cite{DelNin}, the set of all points $x\in \Omega$ for which for every $r>0$ there exists $c_r\in\R^d$ such that
\begin{equation}
u_{x,r}-c_r\to v\quad\text{in $L^1(B_1(0))$ as $r\to 0^+$ for $v$ not constant}
\end{equation}
is also countably $(N-1)-$rectifiable (see the explanation after Theorem 1.2 in \cite{DelNin}).  We denote this set by $\Sigma'_u$. Note that

$$
\inf_{c\in \R^d}\fint_{B_r(x)}|u(y)-c|dy=\fint_{B_r(x)}|u(y)-c_r|dy
$$
for some $c_r\in \R^d$.
It follows that
$\mathcal{J}'_u\subset\Sigma'_u$ since if $x\in \mathcal{J}'_u$, then by Definition \ref{def:approximate jump pointsgint}, there exists for each $r>0$ an element $c_r\in \R^d$ such that $u_{x,r}-c_r$ converges in $L^1(B_1(0))$ as $r\to 0^+$ to the non-constant function \eqref{eq:jeneralized jump}, and therefore $u$ has non-constant blow-up at $x$, i.e.  $x\in \Sigma'_u$.
Thus,
{\it$\mathcal{J}'_u$ is countably $(N-1)-$rectifiable}.
\end{remark}

\section{Properties of $BV^q$-functions}
In this section we prove a fine property of $BV^q$- functions (Theorem \ref{thm:sigma finiteness of limiting average with respect to Hausdorff measure}) as well as the summability of the jump function $|u^{j'}|\in L^q(\mathcal{J}'_u,\mathcal{H}^{N-1})$ for functions $u\in BV^q(\Omega,\R^d)$ (Corollary \ref{cor: Besov constant controls the q-variation}). Moreover, we show that the lower infinitesimal Besov $q$-constant $\underline{B}_{u,q}(\Omega)$ is always bigger or equal to the $q$-jump variation of $u\in BV^q(\Omega,\R^d)$ multiplied by an optimal constant $C_N$ (Theorem \ref{thm:main result,general formulation}).

\begin{definition}
\label{def:definitions of Besov constants}
Let $\Omega\subset\mathbb{R}^N$ be an open set, $1\leq q<\infty$ and let
$u:\Omega\to \R^d$ be an $\mathcal{L}^N-$measurable function.

We define the {\it Besov $q-$constant} of $u$ by
\begin{equation}
\hat{B}_{u,q}(\Omega):=\sup_{\epsilon\in (0,1)}\int_{\Omega}\left(\int_{\Omega\cap B_{\epsilon}(x)}\frac{1}{\epsilon^N}\frac{|u(x)-u(y)|^q}{|x-y|}dy\right)dx.
\end{equation}
We define the {\it upper infinitesimal Besov $q-$constant} of $u$ by
\begin{equation}
\overline{B}_{u,q}(\Omega):=\limsup_{\epsilon\to 0^+}\int_{\Omega}\left(\int_{\Omega\cap B_{\epsilon}(x)}\frac{1}{\epsilon^N}\frac{|u(x)-u(y)|^q}{|x-y|}dy\right)dx.
\end{equation}
We define the {\it lower infinitesimal Besov $q-$constant} of $u$ by
\begin{equation}
\underline{B}_{u,q}(\Omega):=\liminf_{\epsilon\to 0^+}\int_{\Omega}\left(\int_{\Omega\cap B_{\epsilon}(x)}\frac{1}{\epsilon^N}\frac{|u(x)-u(y)|^q}{|x-y|}dy\right)dx.
\end{equation}
If the following limit exists:
\begin{equation}
B_{u,q}(\Omega):=\lim_{\epsilon\to 0^+}\int_{\Omega}\left(\int_{\Omega\cap B_{\epsilon}(x)}\frac{1}{\epsilon^N}\frac{|u(x)-u(y)|^q}{|x-y|}dy\right)dx,
\end{equation}
then it is called the {\it infinitesimal Besov $q-$constant} of $u$.
\end{definition}

We redefine here the space $BV^q$ in terms of Definition \ref{def:definitions of Besov constants}:

\begin{definition}
\label{def:definition of BV^q}
Let $1\leq q<\infty$ and let $\Omega\subset \mathbb{R}^N$ be an open set. We define a set:
\begin{equation}
BV^q(\Omega,\R^d):=\left\{u\in L^q\left(\Omega,\R^d\right):\hat{B}_{u,q}(\Omega)<\infty\right\},
\end{equation}
where $\hat{B}_{u,q}(\Omega)$ is defined in Definition $\ref{def:definitions of Besov constants}$.
The local space $BV^q_{loc}(\Omega,\R^d)$ is defined as follows: $u\in BV^q_{loc}(\Omega,\R^d)$ if and only if $u\in L^q_{loc}(\Omega,\R^d)$ and $u\in BV^q(\Omega_0,\R^d)$ for every open $\Omega_0\subset\subset\Omega$.
\end{definition}

\begin{proposition}
\label{prop:equivalence of finiteness of upper and Besov constants}
Let $u\in L^q(\Omega,\mathbb{R}^d)$, where $q\in [1,\infty)$ and $\Omega\subset\R^N$ is an open set. Then, the finiteness of
$\hat{B}_{u,q}(\Omega)$ is equivalent to the finiteness of the upper infinitesimal Besov $q-$constant $\overline{B}_{u,q}(\Omega)$.
\end{proposition}

\begin{proof}
If $\hat{B}_{u,q}(\Omega)<\infty$, then $\overline{B}_{u,q}(\Omega)\leq \hat{B}_{u,q}(\Omega)<\infty$.
Assume
$\overline{B}_{u,q}(\Omega)<\infty.$
Then there exists a number
$0<\e_0<1$
such that
\begin{equation}
\sup_{\e\in (0,\e_0]}\int_{\Omega}\left(\int_{\Omega\cap B_\e(x)}\frac{1}{\e^N}\frac{|u(x)-u(y)|^q}{|x-y|}dy\right)dx<\infty.
\end{equation}
We have
\begin{align}
\sup_{\e\in [\e_0,1)}\int_{\Omega}\left(\int_{\Omega\cap B_\e(x)}\frac{1}{\e^N}\frac{|u(x)-u(y)|^q}{|x-y|}dy\right)dx\nonumber
\\
\leq \int_{\Omega}\left(\int_{\Omega\cap B_{\e_0}(x)}\frac{1}{\e_0^N}\frac{|u(x)-u(y)|^q}{|x-y|}dy\right)dx\nonumber
\\
+\sup_{\e\in [\e_0,1)}\int_{\Omega}\left(\int_{\Omega\cap \left(B_\e(x)\setminus B_{\e_0}(x)\right)}\frac{1}{\e^N}\frac{|u(x)-u(y)|^q}{|x-y|}dy\right)dx.
\end{align}
Since
$u\in L^q(\Omega,\mathbb{R}^d),$
then we have
\begin{align}
\sup_{\e\in [\e_0,1)}\int_{\Omega}\left(\int_{\Omega\cap \left(B_\e(x)\setminus B_{\e_0}(x)\right)}\frac{1}{\e^N}\frac{|u(x)-u(y)|^q}{|x-y|}dy\right)dx\nonumber
\\
\leq 2^{q-1}\frac{1}{\e_0^{N+1}}\sup_{\e\in [\e_0,1)}\int_{\Omega}\left(\int_{\Omega\cap \left(B_\e(x)\setminus B_{\e_0}(x)\right)}|u(x)|^qdy\right)dx\nonumber
\\
+2^{q-1}\frac{1}{\e_0^{N+1}}\sup_{\e\in [\e_0,1)}\int_{\Omega}\left(\int_{\Omega\cap \left(B_{\e}(x)\setminus B_{\e_0}(x)\right)}|u(y)|^qdy\right)dx\nonumber
\\
\leq2^q\frac{1}{\e_0^{N+1}}\Leb^N(B_1(0))\|u\|^q_{L^q(\Omega,\mathbb{R}^d)}<\infty.
\end{align}
Thus, $\hat{B}_{u,q}(\Omega)<\infty$.
\end{proof}

\begin{remark}
The set $BV^q\left(\Omega,\R^d\right)$ is a real Banach space continuously embedded in
$L^q\left(\Omega,\R^d\right)$ equipped with the norm:

\begin{equation}
\|u\|_{BV^q(\Omega,\mathbb{R}^d)}:=\left(\hat{B}_{u,q}(\Omega)\right)^{\frac{1}{q}}+\|u\|_{L^q(\Omega,\mathbb{R}^d)}.
\end{equation}
For more details about the space $BV^q(\Omega,\R^d)$ see \cite{P}.
\end{remark}

\subsection{Fine properties of $BV^q$-functions}
Recall Theorem \ref{thm:Hausdorff and Radon measures} in the Appendix. We give here a kind of generalization of this theorem for a sequence of Radon measures:
\begin{lemma}
\label{lem:the upper density lemma for a sequence of Radon measures}
Let $\Omega\subset\R^N$ be an open set. For every $n=1,2,\ldots$ let
$\mu_n$ be a positive Radon measure on $\Omega$. Moreover, assume
that for some $s\geq 0$, some $t>0$ and some Borel set
$A\subset\Omega$ we have
\begin{equation}\label{gtgyihhzzkkhhjgjhjggjgiyyu22jokuiu1ookjkjkkkjhhpkokkliokkoijjjjhliojojj109ojhyyuyyiuuyuuyujkhjhjijiyuuhjgjklkhuiuiuyhhiuy35klljkjjk11nhj}
\lim_{r\to 0^+}\Bigg\{\liminf\limits_{n\to
\infty}\Bigg(\sup\limits_{0<\rho\leq
r}\bigg(\frac{1}{\alpha(s)\rho^s}\mu_{n}\Big(\ov
B_{\rho}(x)\Big)\bigg)\Bigg)\Bigg\}\geq t
\quad\quad\text{for}\;\;\mathcal{H}^s\;\;{a.e.}\;\; x\in A,
\end{equation}
where
\begin{equation}
\alpha(s):=\frac{\pi^{s/2}}{\Gamma\left(\frac{s}{2}+1\right)},\quad \Gamma(s):=\int_{0}^\infty e^{-x}x^{s-1}dx,\quad \Gamma:[0,\infty)\to \R.
\end{equation}
Then we have
\begin{equation}\label{gtgyihhzzkkhhjgjhjggjgiyyu22jokuiu1ookjkjkkkjhhpkokkliokkoijjjjhliojojj109ojhyyuyyiuuyuuyujkhjhjijiyuuhjgjklkhuiuiuyhhiuy35klljkjjkjkjhuhh}
5^s\Big(\liminf\limits_{n\to \infty}\mu_n(G)\Big)\geq
t\mathcal{H}^s(A)\,,
\end{equation}
where $G\subset\Omega$ is an arbitrary open set, such that $A\subset
G$.
\end{lemma}
\begin{proof}
Assume first that a compact $K\subset A$ is such that
\begin{equation}\label{gthkhkhhkhkjh}
\mathcal{H}^s(K)<\infty\,.
\end{equation}
Let $U\subset\Omega$ be an arbitrary bounded open set, such that
$K\subset U\subset\ov U\subset\Omega$. Furthermore, let $\delta\in
(0,1)$ be an arbitrary real number, satisfying
\begin{equation}
\label{eq:delta is less than the distance between K and the complement of U}
\delta<|y-z|\quad\quad\forall\, y\in K\;\;\;\forall\, z\in
\R^N\setminus U\,.
\end{equation}
Then, by
\er{gtgyihhzzkkhhjgjhjggjgiyyu22jokuiu1ookjkjkkkjhhpkokkliokkoijjjjhliojojj109ojhyyuyyiuuyuuyujkhjhjijiyuuhjgjklkhuiuiuyhhiuy35klljkjjk11nhj},
we have
\begin{equation}\label{gtgyihhzzkkhhjgjhjggjgiyyu22jokuiu1ookjkjkkkjhhpkokkliokkoijjjjhliojojj109ojhyyuyyiuuyuuyujkhjhjijiyuuhjgjklkhuiuiuyhhiuy35klljkjjk11nhjpoopooo}
\liminf\limits_{n\to \infty}\Bigg(\sup\limits_{0<\rho<
\delta}\bigg(\frac{1}{\alpha(s)\rho^s}\mu_{n}\Big(\ov
B_{\rho}(x)\Big)\bigg)\Bigg)\geq t
\quad\quad\text{for}\;\;\mathcal{H}^s\;\;{a.e.}\;\; x\in K\,.
\end{equation}

Next, for every natural $n\geq 1$ we define
\begin{equation}\label{gtgyihhzzkkhhjgjhjggjgiyyu22jokuiu1ookjkjkkkjhhpkokkliokkoijjjjhliojojj109ojhyyuyyiuuyuuyujkhjhjijiyuuhjgjklkhuiuiuyhhiuy35klljkjjk11nhjpoopoookkkhhkh}
K'_n=\Bigg\{x\in K\;:\;\;\sup\limits_{0<\rho<
\delta}\bigg(\frac{1}{\alpha(s)\rho^s}\mu_{n}\Big(\ov
B_{\rho}(x)\Big)\bigg)> t(1-\delta)\Bigg\} \,,
\end{equation}

\begin{equation}
K_n=\bigcap_{m=n}^{\infty}K'_m\,.
\end{equation}
Then, $K_n\subset K_{n+1}$ and by
\eqref{gtgyihhzzkkhhjgjhjggjgiyyu22jokuiu1ookjkjkkkjhhpkokkliokkoijjjjhliojojj109ojhyyuyyiuuyuuyujkhjhjijiyuuhjgjklkhuiuiuyhhiuy35klljkjjk11nhjpoopooo}
we have
\begin{equation}
\mathcal{H}^s\Bigg(K\setminus\bigcup_{n=1}^{\infty}K_n\Bigg)=0\,.
\end{equation}
In particular, there exists a natural number $n_0\geq 1$, such that
\begin{equation}\label{gtgyihhzzkkhhjgjhjggjgiyyu22jokuiu1ookjkjkkkjhhpkokkliokkoijjjjhliojojj109ojhyyuyyiuuyuuyujkhjhjijiyuuhjgjklkhuiuiuyhhiuy35klljkjjk11nhjpoopoooplklklkkk}
\mathcal{H}^s\big(K\setminus K_{n_0}\big)<\delta\,,
\end{equation}
and moreover, for every $x\in K_{n_0}$ and every $n\geq n_0$ we have
\begin{equation}\label{gtgyihhzzkkhhjgjhjggjgiyyu22jokuiu1ookjkjkkkjhhpkokkliokkoijjjjhliojojj109ojhyyuyyiuuyuuyujkhjhjijiyuuhjgjklkhuiuiuyhhiuy35klljkjjk11nhjpoopoookkkhhkhukkukkj}
\sup\limits_{0<\rho<
\delta}\bigg(\frac{1}{\alpha(s)\rho^s}\mu_{n}\Big(\ov
B_{\rho}(x)\Big)\bigg)> t(1-\delta) \,.
\end{equation}
Therefore, for every $x\in K_{n_0}$ and every $n\geq n_0$ there
exists $\rho_n(x)\in(0,\delta)$ such that
\begin{equation}\label{gtgyihhzzkkhhjgjhjggjgiyyu22jokuiu1ookjkjkkkjhhpkokkliokkoijjjjhliojojj109ojhyyuyyiuuyuuyujkhjhjijiyuuhjgjklkhuiuiuyhhiuy35klljkjjk11nhjpoopoookkkhhkhukkujkkh}
t\alpha(s)(\rho_n(x))^s<\frac{1}{1-\delta}\,\mu_{n}\Big(\ov
B_{\rho_n(x)}(x)\Big)\,.
\end{equation}
Next, for every natural $n\geq n_0$ consider a collection of
closed balls
\begin{equation}\label{gtgyihhzzkkhhjgjhjggjgiyyu22jokuiu1ookjkjkkkjhhpkokkliokkoijjjjhliojojj109ojhyyuyyiuuyuuyujkhjhjijiyuuhjgjklkhuiuiuyhhiuy35klljkjjkjkhl;l;33yy}
\mathcal{F}^{(n)}:=\Big\{\ov B_{\rho_{n}(x)}(x)\subset\R^N\;:\;x\in
K_{n_0}\Big\}\,,
\end{equation}
such that we have by
$\eqref{eq:delta is less than the distance between K and the complement of U}$
\begin{equation}\label{gtgyihhzzkkhhjgjhjggjgiyyu22jokuiu1ookjkjkkkjhhpkokkliokkoijjjjhliojojj109ojhyyuyyiuuyuuyujkhjhjijiyuuhjgjklkhuiuiuyhhiuy35klljkjjkjkhkljlj33yy}
K_{n_0}\subset\bigcup\limits_{B_{\rho_{n}(x)}(x)\in
\mathcal{F}^{(n)}}\ov B_{\rho_{n}(x)}(x)\subset U\,.
\end{equation}
Then, by Theorem \ref{thm:Vitali's covering theorem} (Vitali's covering theorem), for every $n\geq n_{0}$ there
exists at most countable family of \underline{disjoint} balls
\begin{equation}\label{gtgyihhzzkkhhjgjhjggjgiyyu22jokuiu1ookjkjkkkjhhpkokkliokkoijjjjhliojojj109ojhyyuyyiuuyuuyujkhjhjijiyuuhjgjklkhuiuiuyhhiuy35klljkjjkjkhlukkkljjklkklklkk3355}
\mathcal{G}^{(n)}\subset\mathcal{F}^{(n)}\,,
\end{equation}
so that
\begin{equation}
\label{eq: K_n0 is a subset of D hat}
K_{n_0}\subset \hat D^{(n)}\,,
\end{equation}
where we denote
\begin{equation}\label{gtgyihhzzkkhhjgjhjggjgiyyu22jokuiu1ookjkjkkkjhhpkokkliokkoijjjjhliojojj109ojhyyuyyiuuyuuyujkhjhjijiyuuhjgjklkhuiuiuyhhiuy35klljkjjkjkhkljljjjjkkkjjhhhllllkljk3355}
D^{(n)}:=\bigcup\limits_{\ov B_{\rho_{n}(x)}(x)\in
\mathcal{G}^{(n)}}\ov B_{\rho_{n}(x)}(x)\subset U\quad{and}\quad
\hat D^{(n)}:=\bigcup\limits_{\ov B_{\rho_{n}(x)}(x)\in
\mathcal{G}^{(n)}}\ov B_{5\rho_{n}(x)}(x)\,.
\end{equation}
Thus, by
\er{gtgyihhzzkkhhjgjhjggjgiyyu22jokuiu1ookjkjkkkjhhpkokkliokkoijjjjhliojojj109ojhyyuyyiuuyuuyujkhjhjijiyuuhjgjklkhuiuiuyhhiuy35klljkjjk11nhjpoopoookkkhhkhukkujkkh},
for every $n\geq n_{0}$ we have
\begin{multline}
\label{eq: estimate for the sum of radii of balls is G^n}
\sum\limits_{\ov B_{\rho_{n}(x)}(x)\in
\mathcal{G}^{(n)}}t\alpha(s)(\rho_{n}(x))^s\leq
\frac{1}{1-\delta}\sum\limits_{\ov B_{\rho_{n}(x)}(x)\in
\mathcal{G}^{(n)}}\mu_{n}\Big(\ov
B_{\rho_{n}(x)}(x)\Big)\\=\frac{1}{1-\delta}\mu_{n}\big(D^{(n)}\big)
\leq\frac{1}{1-\delta}\mu_{n}\big(U\big)\,.
\end{multline}
Then, by $\eqref{eq: K_n0 is a subset of D hat}$ and $\eqref{eq: estimate for the sum of radii of balls is G^n}$ for every $n\geq n_{0}$ we have
\begin{equation}\label{gtgyihhzzkkhhjgjhjggjgiyyu22jokuiu1ookjkjkkkjhhpkokkliokkoijjjjhliojojj109ojhyyuyyiuuyuuyujkhjhjijiyuuhjgjklkhuiuiuyhhiuy35klljkjjkjkhkljljjjjkkkkjljjljkjljfyfyfghghghkkljhjkkj}
\frac{t}{5^s}\mathcal{H}^s_{10\delta}\big(K_{n_0}\big)\leq
\sum\limits_{\ov B_{\rho_{n}(x)}(x)\in
\mathcal{G}^{(n)}}\frac{t}{5^s}\alpha(s)(5\rho_{n}(x))^s=
\sum\limits_{\ov B_{\rho_{n}(x)}(x)\in
\mathcal{G}^{(n)}}t\alpha(s)(\rho_{n}(x))^s\leq\frac{1}{1-\delta}\mu_{n}\big(U\big)\,.
\end{equation}

So, for every $n\geq n_{0}$ we have
\begin{equation}\label{gtgyihhzzkkhhjgjhjggjgiyyu22jokuiu1ookjkjkkkjhhpkokkliokkoijjjjhliojojj109ojhyyuyyiuuyuuyujkhjhjijiyuuhjgjklkhuiuiuyhhiuy35klljkjjkjkhkljljjjjkkkkjljjljkjljfyfyfghghghkkljhjkkjll55}
t\mathcal{H}^s_{10\delta}\big(K_{n_0}\big)\leq
\frac{5^s}{1-\delta}\mu_{n}\big(U\big)\,.
\end{equation}
Thus, by
\er{gtgyihhzzkkhhjgjhjggjgiyyu22jokuiu1ookjkjkkkjhhpkokkliokkoijjjjhliojojj109ojhyyuyyiuuyuuyujkhjhjijiyuuhjgjklkhuiuiuyhhiuy35klljkjjk11nhjpoopoooplklklkkk}
and
\er{gtgyihhzzkkhhjgjhjggjgiyyu22jokuiu1ookjkjkkkjhhpkokkliokkoijjjjhliojojj109ojhyyuyyiuuyuuyujkhjhjijiyuuhjgjklkhuiuiuyhhiuy35klljkjjkjkhkljljjjjkkkkjljjljkjljfyfyfghghghkkljhjkkjll55}
together, we have
\begin{equation}\label{gtgyihhzzkkhhjgjhjggjgiyyu22jokuiu1ookjkjkkkjhhpkokkliokkoijjjjhliojojj109ojhyyuyyiuuyuuyujkhjhjijiyuuhjgjklkhuiuiuyhhiuy35klljkjjkjkhkljljjjjkkkkjljjljkjljfyfyfghghghkkljhjkkjll}
t\mathcal{H}^s_{10\delta}\big(K\big)\leq
\frac{5^s}{1-\delta}\mu_{n}\big(U\big)+t\delta\,.
\end{equation}

Then, taking the lower limit as $n\to\infty$ in
\er{gtgyihhzzkkhhjgjhjggjgiyyu22jokuiu1ookjkjkkkjhhpkokkliokkoijjjjhliojojj109ojhyyuyyiuuyuuyujkhjhjijiyuuhjgjklkhuiuiuyhhiuy35klljkjjkjkhkljljjjjkkkkjljjljkjljfyfyfghghghkkljhjkkjll}
implies
\begin{equation}
\label{eq: ineqality of H^d10delta}
t\mathcal{H}^s_{10\delta}\big(K\big)\leq
\frac{5^s}{1-\delta}\Big(\liminf_{n\to
\infty}\mu_{n}\big(U\big)\Big)+t\delta\,.
\end{equation}
Thus, letting $\delta\to 0^+$ in $\eqref{eq: ineqality of H^d10delta}$
we deduce
\begin{equation}
\label{eq:local inequality for H^d}
t\mathcal{H}^s(K)\leq
5^s\Big(\liminf\limits_{n\to\infty}\mu_{n}(U)\Big)\,.
\end{equation}

By a well known theorem about Hausdorff measures (see Corollary 2.10.48 in Federer's book \cite{F}) we have
\begin{equation}
\mathcal{H}^s(A)=\sup\limits_{K\subset
A,\,\mathcal{H}^s(K)<\infty,K\,is\,compact}\mathcal{H}^s(K)\,.
\end{equation}
Let $G\subset\Omega$ be an arbitrary open set such that $A\subset G$. For every compact $K\subset A$ let $U_K$ be an open and bounded set such that $K\subset U_K\subset \ov U_K\subset G$.
Therefore, by $\eqref{eq:local inequality for H^d}$ we deduce
\begin{align}
t\mathcal{H}^s(A)&=\sup\limits_{K\subset
A,\,\mathcal{H}^s(K)<\infty,K\,is\,compact}t\mathcal{H}^s(K)\nonumber
\\
&\leq \sup\limits_{K\subset
A,\,\mathcal{H}^s(K)<\infty,K\,is\,compact}5^s\bigg(\liminf\limits_{n\to\infty}\mu_n(U_K)\bigg)
\leq
5^s\bigg(\liminf\limits_{n\to\infty}\mu_n(G)\bigg).
\end{align}
\end{proof}

\begin{theorem}
\label{thm:lower function of upper density of sequence of Radon measures}
Let $\Omega\subset\R^N$ be an open set. For every $n=1,2,\ldots$ let
$\mu_n$ be a positive Radon measure on $\Omega$. Moreover, assume
that for some $s\geq 0$, we have
\begin{equation}
\label{eq: upper densitiy of a sequence of Radon measures is bounded from below by g}
\lim_{r\to 0^+}\Bigg\{\liminf\limits_{n\to
\infty}\Bigg(\sup\limits_{0<\rho\leq
r}\frac{1}{\alpha(s)\rho^s}\mu_{n}\Big(\ov
B_{\rho}(x)\Big)\Bigg)\Bigg\}\geq
g(x)\quad\quad\text{for}\;\;\mathcal{H}^s\;\;{a.e.}\;\; x\in
\Omega\,,
\end{equation}
where $g:\Omega\to[0,\infty]$ is some Borel function. Then,
\begin{equation}\label{gtgyihhzzkkhhjgjhjggjgiyyu22jokuiu1ookjkjkkkjhhpkokkliokkoijjjjhliojojj109ojhyyuyyiuuyuuyujkhjhjijiyuuhjgjklkhuiuiuyhhiuy35klljkjjkjkjkljlj}
5^s\Big(\liminf\limits_{n\to \infty}\mu_n(\Omega)\Big)\geq
\int_\Omega g(x)d\mathcal{H}^s(x).
\end{equation}
\end{theorem}

\begin{proof}
Assume that
\begin{equation}\label{gtgyihhzzkkhhjgjhjggjgiyyu22jokuiu1ookjkjkkkjhhpkokkliokkoijjjjhliojojj109ojhyyuyyiuuyuuyujkhjhjijiyuuhjgjklkhuiuiuyhhiuy35klljkjjkjkjkljljkljkljl}
\liminf\limits_{n\to \infty}\mu_n(\Omega)=L<\infty\,,
\end{equation}
otherwise
\er{gtgyihhzzkkhhjgjhjggjgiyyu22jokuiu1ookjkjkkkjhhpkokkliokkoijjjjhliojojj109ojhyyuyyiuuyuuyujkhjhjijiyuuhjgjklkhuiuiuyhhiuy35klljkjjkjkjkljlj}
is trivial. Then, up to a subsequence,
\begin{equation}\label{gtgyihhzzkkhhjgjhjggjgiyyu22jokuiu1ookjkjkkkjhhpkokkliokkoijjjjhliojojj109ojhyyuyyiuuyuuyujkhjhjijiyuuhjgjklkhuiuiuyhhiuy35klljkjjkjkjkljljkljkljlii}
\lim\limits_{n\to \infty}\mu_n(\Omega)=L<\infty\,.
\end{equation}
Thus, up to a further subsequence,
\begin{equation}\label{jmjkjhjhjhj}
\mu_n\rightharpoonup^*\mu\quad\quad\text{weakly}^*\quad\text{in}\quad
\Omega\,.
\end{equation}
Let $K\subset\Omega$ be a compact set such that for $\mathcal{H}^s$ a.e. $x\in K$ we have $g(x)\geq t$ for some number $t> 0$. Let $G\subset\Omega$ be an arbitrary open set such that $K\subset G$.
Let $U$ be a bounded open set such that $K\subset U\subset \ov U \subset G$.
Then, by Lemma \ref{lem:the upper density lemma for a sequence of Radon measures}, assumption
$\eqref{eq: upper densitiy of a sequence of Radon measures is bounded from below by g}$ and the weak* convergence of $\mu_n$ to $\mu$
\begin{align}\label{gtgyihhzzkkhhjgjhjggjgiyyu22jokuiu1ookjkjkkkjhhpkokkliokkoijjjjhliojojj109ojhyyuyyiuuyuuyujkhjhjijiyuuhjgjklkhuiuiuyhhiuy35klljkjjkjkhkljljjjjkkkkjljjiiljljjjkjkiuhhhkjhjjhjkhhjhhghhg}
t\mathcal{H}^s(K)&\leq
5^s\bigg(\liminf\limits_{n\to\infty}\mu_n(U)\bigg)\leq
5^s\bigg(\liminf\limits_{n\to\infty}\mu_n(\ov U)\bigg)\nonumber
\\
&\leq 5^s\bigg(\limsup\limits_{n\to\infty}\mu_n(\ov U)\bigg)\leq
5^s\mu(\ov U)\leq 5^s\mu(G).
\end{align}
By the outer regularity of the Radon measure $\mu$ we get $t\mathcal{H}^s(K)\leq 5^s\mu(K)$.

Let $B\subset\Omega$ be any Borel set such that $g(x)\geq t$  $\mathcal{H}^s$ a.e. $x\in B$ for some $t>0$.
By Corollary 2.10.48 in \cite{F} and the inner regularity of
$\mu$ we get
\begin{align}
\label{eq:inquality8}
t\mathcal{H}^s(B)&=\sup\limits_{K\subset
B,\,\mathcal{H}^s(K)<\infty,K\,is\,compact}t\mathcal{H}^s(K)\nonumber
\\
&\leq \sup\limits_{K\subset
B,\,\mathcal{H}^s(K)<\infty,K\,is\,compact}5^s\mu(K)\nonumber
\\
&\leq 5^s\sup\limits_{K\subset
B,K\,is\,compact}\mu(K)=5^s\mu(B).
\end{align}

Then, by \eqref{eq:inquality8}, the definition of Lebesgue integral (using simple lower functions for $g$) and the weak* convergence of $\mu_n$ to $\mu$ we have
\begin{equation}\label{gtgyihhzzkkhhjgjhjggjgiyyu22jokuiu1ookjkjkkkjhhpkokkliokkoijjjjhliojojj109ojhyyuyyiuuyuuyujkhjhjijiyuuhjgjklkhuiuiuyhhiuy35klljkjjkjkhkljljjjjkkkkjljjiiljljjjkjkiuhhhkjhjjhjkhhjhhghhghhuuhuhhhyyt}
\int_\Omega g(x)d\mathcal{H}^s(x)\leq 5^s\mu(\Omega)\leq
5^s\Big(\liminf\limits_{n\to \infty}\mu_n(\Omega)\Big)\,.
\end{equation}
\end{proof}

\begin{theorem}
\label{thm:Besov constant controls double average integral}
Let $\Omega\subset\R^N$ be an open set and $q,r\in (0,\infty)$ such that $rq\leq N$. Next consider
any sequence $\e_n>0$, such that $\lim_{n\to\infty}\e_n=0$ and
suppose $u:\Omega\to \R^d$ is $\mathcal{L}^N$ measurable function. Then, there exists a constant
$\beta(N;r,q)$, depending on $N,r,q$ only, such that
\begin{multline}
\liminf\limits_{n\to\infty}\int_\Omega\Bigg\{\fint_{B_1(0)}\chi_\Omega\big(x+\e_n
z\big) \frac{\Big|u\big(x+\e_n
z\big)-u(x)\Big|^q}{\e_n^{rq}}dz\Bigg\}dx
\\
\geq\beta(N;r,q)\int_\Omega \Bigg\{\liminf\limits_{n\to
\infty}\fint_{B_{1}(0)}\fint_{B_1(0)} \Big|u\big(x+(\e_n/2)
z\big)-u\big(x+(\e_n/2) y\big)\Big|^qdzdy\Bigg\}d\mathcal{H}^{N-rq}(x)\,.
\end{multline}
\end{theorem}

\begin{proof}
For every $x\in\Omega$ we have
\begin{multline}
\liminf\limits_{n\to \infty}\Bigg(\sup\limits_{0<\rho<
\delta}\frac{1}{\alpha(N-rq)\rho^{N-rq}}\int_{B_{\rho}(x)}\Bigg\{\fint_{B_1(0)}\chi_\Omega\big(y+\e_n
z\big) \frac{\Big|u\big(y+\e_n
z\big)-u(y)\Big|^q}{\e_n^{rq}}dz\Bigg\}dy\Bigg)
\\
=\liminf\limits_{n\to \infty}\Bigg(\sup\limits_{0<M<
\frac{\delta}{\e_n}}\frac{1}{\alpha(N-rq)M^{N-rq}\e^N_n}\int_{B_{M\e_n}(x)}\Bigg\{\fint_{B_1(0)}\chi_\Omega\big(y+\e_n
z\big) \Big|u\big(y+\e_n z\big)-u(y)\Big|^qdz\Bigg\}dy\Bigg)
\\
\geq \liminf\limits_{n\to \infty}\frac{1}{\alpha(N-rq)\e^N_n}\int_{B_{\e_n}(x)}\Bigg\{\fint_{B_1(0)}\chi_\Omega\big(y+\e_n
z\big) \Big|u\big(y+\e_n z\big)-u(y)\Big|^qdz\Bigg\}dy
\\
=\liminf\limits_{n\to
\infty}\frac{1}{\alpha(N-rq)}\int_{B_{1}(0)}\Bigg\{\fint_{B_1(0)}\chi_\Omega\big(x+\e_n
y+\e_n
z\big)\Big|u\big(x+\e_n
y+\e_n z\big)-u(x+\e_n y)\Big|^qdz\Bigg\}dy
\\
=\liminf\limits_{n\to
\infty}\frac{1}{\alpha(N-rq)}\int_{B_{1}(0)}\Bigg\{\fint_{B_1(0)}\Big|u\big(x+\e_n
y+\e_n z\big)-u(x+\e_n y)\Big|^qdz\Bigg\}dy.
\end{multline}
Thus,
\begin{multline}\label{gtgyihhzzkkhhjgjhjggjgiyyu22jokuiu1ookjkjkkkjhhpkokkliokkoijjjjhliojojj109ojhyyuyyiuuyuuyujkhjhjijiyuuhjgjklkhuiuiuyhhiuy35klljkjjkkhjh11kjkhuyy}
\lim\limits_{\delta\to 0^+}\liminf\limits_{n\to
\infty}\Bigg(\sup\limits_{0<\rho<
\delta}\frac{1}{\alpha(N-rq)\rho^{N-rq}}\int_{B_{\rho}(x)}\Bigg\{\fint_{B_1(0)}\chi_\Omega\big(y+\e_n
z\big) \frac{\Big|u\big(y+\e_n
z\big)-u(y)\Big|^q}{\e_n^{rq}}dz\Bigg\}dy\Bigg)
\\
\geq \liminf\limits_{n\to
\infty}\frac{1}{\alpha(N-rq)}\int_{B_{1}(0)}\Bigg\{\fint_{B_1(0)}\Big|u\big(x+\e_n
y+\e_n z\big)-u(x+\e_n y)\Big|^qdz\Bigg\}dy \quad\quad\forall\,x\in
\Omega\,.
\end{multline}

Let us denote for each $n\in \N$ and $E\subset \Omega$
\begin{equation}
\mu_n(E):=\int_{E}\Bigg\{\fint_{B_1(0)}\chi_\Omega\big(y+\e_n
z\big) \frac{\Big|u\big(y+\e_n
z\big)-u(y)\Big|^q}{\e_n^{rq}}dz\Bigg\}dy,
\end{equation}
and
\begin{equation}
g(x):=\liminf\limits_{n\to
\infty}\frac{1}{\alpha(N-rq)}\int_{B_{1}(0)}\Bigg\{\fint_{B_1(0)}\Big|u\big(x+\e_n
y+\e_n z\big)-u(x+\e_n y)\Big|^qdz\Bigg\}dy.
\end{equation}
Therefore, by Theorem \ref{thm:lower function of upper density of sequence of Radon measures} we deduce
\begin{multline}
\label{eq:ineqality1}
5^{N-rq}\liminf\limits_{n\to\infty}\int_\Omega\Bigg\{\fint_{B_1(0)}\chi_\Omega\big(x+\e_n
z\big) \frac{\Big|u\big(x+\e_n z\big)-u(x)\Big|^q}{\e_n^{rq}}dz\Bigg\}dx
\\
\geq \int_\Omega \Bigg\{ \liminf\limits_{n\to
\infty} \frac{1}{\alpha(N-rq)}\int_{B_{1}(0)}\fint_{B_1(0)}\Big|u\big(x+\e_n y+\e_n
z\big)-u(x+\e_n y)\Big|^qdzdy\Bigg\}d\mathcal{H}^{N-rq}(x)\,.
\end{multline}
However, for every $x\in\Omega$ we have
\begin{multline}
\label{eq:ineqality2}
\liminf\limits_{n\to \infty}\int_{B_{1}(0)}\int_{B_1(0)}
\Big|u\big(x+(\e_n/2) z\big)-u\big(x+(\e_n/2) y\big)\Big|^qdzdy
\\
=\liminf\limits_{n\to
\infty}\,\frac{1}{(\e_n/2)^{N}}\int_{B_{\e_n/2}(x)}\Bigg\{\frac{1}{(\e_n/2)^{N}}\int_{B_{\e_n/2}(x)}
\Big|u\big(z\big)-u(y)\Big|^qdz\Bigg\}dy
\\
\leq\liminf\limits_{n\to
\infty}\,\frac{1}{(\e_n/2)^{N}}\int_{B_{\e_n/2}(x)}\Bigg\{\frac{1}{(\e_n/2)^{N}}\int_{B_{\e_n}(y)}
\Big|u\big(z\big)-u(y)\Big|^qdz\Bigg\}dy
\\
\leq 2^{2N}\,\liminf\limits_{n\to
\infty}\,\frac{1}{\e_n^{N}}\int_{B_{\e_n}(x)}\Bigg\{\frac{1}{\e_n^{N}}\int_{B_{\e_n}(y)}
\Big|u\big(z\big)-u(y)\Big|^qdz\Bigg\}dy
\\
=2^{2N}\,\liminf\limits_{n\to
\infty}\int_{B_{1}(0)}\Bigg\{\int_{B_1(0)}\Big|u\big(x+\e_n
y+\e_n z\big)-u(x+\e_n y)\Big|^qdz\Bigg\}dy
\\
=2^{2N}\alpha(N)\alpha(N-rq)\,\liminf\limits_{n\to
\infty}\frac{1}{\alpha(N-rq)}\int_{B_{1}(0)}\fint_{B_1(0)}\Big|u\big(x+\e_n
y+\e_n z\big)-u(x+\e_n y)\Big|^qdzdy.
\end{multline}
Thus, inserting \eqref{eq:ineqality2} into \eqref{eq:ineqality1} we get
\begin{multline}
\liminf\limits_{n\to\infty}\int_\Omega\Bigg\{\fint_{B_1(0)}\chi_\Omega\big(x+\e_n
z\big) \frac{\Big|u\big(x+\e_n z\big)-u(x)\Big|^q}{\e_n^{rq}}dz\Bigg\}dx
\\
\geq\beta(N;r,q)\int_\Omega \Bigg\{\liminf\limits_{n\to
\infty}\fint_{B_{1}(0)}\fint_{B_1(0)} \Big|u\big(x+(\e_n/2)
z\big)-u\big(x+(\e_n/2)
y\big)\Big|^qdzdy\Bigg\}d\mathcal{H}^{N-rq}(x)\,,
\end{multline}
where $\beta(N;r,q):=\frac{\alpha(N)}{5^{N-rq}2^{2N}\alpha(N-rq)}$.
\end{proof}
Then we have the following first result:

\begin{theorem}
\label{thm:sigma finiteness of limiting average with respect to Hausdorff measure}
Let $\Omega\subset\R^N$ be an open set, $q\in [1,\infty)$ and $r\in(0,\infty)$ are such that $rq\leq N$. Suppose
$u\in L^q_{loc}(\Omega,\R^d)$ and
\begin{equation}
\label{eq:finiteness of upper infinitesimal Besov constan}
\limsup\limits_{\e\to
0^+}\int_{\Omega_0}\Bigg\{\fint_{B_1(0)}\chi_{\Omega_0}\big(x+\e z\big)
\frac{\Big|u\big(x+\e z\big)-u(x)\Big|^q}{\e^{rq}}dz\Bigg\}dx
<\infty,
\end{equation}
for every open
$\Omega_0\subset\subset\Omega$.
Then, for any sequence $\e_n>0$, such that
$\lim_{n\to\infty}\e_n=0$ there exists a $\mathcal{H}^{N-rq}$
$\sigma$-finite set $S\subset\Omega$, such that $\forall x\in\Omega\setminus S$
\begin{equation}
\label{eq:fine property of BV^q funcions}
\liminf_{n\to \infty}\fint_{B_{\e_n}(x)}\left|u(z)-u_{B_{\e_n}(x)}\right|^qdz=0,\quad \quad u_{B_{\e_n}(x)}:=\fint_{B_{\e_n}(x)}u(z)dz.
\end{equation}
In particular, if we choose $r=\frac{1}{q}$, then \eqref{eq:fine property of BV^q funcions} holds for $u\in BV^q_{loc}(\Omega,\R^d)$.
\end{theorem}

\begin{proof}
Let $\Omega_0$ be an open set such that
$\Omega_0\subset\subset\Omega$. Let us denote
\begin{equation}
S_0:=\left\{x\in \Omega_0: f(x)>0 \right\}=\bigcup_{l=1}^\infty B_l,\quad B_l:=\left\{x\in \Omega_0: f(x)>\frac{1}{l} \right\},
\end{equation}
where
\begin{equation}
f(x):=\liminf\limits_{n\to \infty}\fint_{B_{1}(0)}\fint_{B_1(0)}
\Big|u\big(x+\e_n z\big)-u\big(x+\e_n
y\big)\Big|^qdzdy.
\end{equation}
By Chebyshev's inequality, Theorem $\ref{thm:Besov constant controls double average integral}$ and assumption
$\eqref{eq:finiteness of upper infinitesimal Besov constan}$, we get
\begin{equation}
\mathcal{H}^{N-rq}(B_l)\leq l\int_{\Omega_0}f(x)d\mathcal{H}^{N-rq}(x)<\infty.
\end{equation}

By Jensen's inequality
\begin{multline}
\label{eq:fine property in local case}
\liminf\limits_{n\to \infty}\fint_{B_{\e_n}(x)} \Bigg|u(
z)-\fint_{B_{\e_n}(x)}u(y)dy\Bigg|^qdz=\liminf\limits_{n\to
\infty}\fint_{B_1(0)} \Bigg|u\big(x+\e_n
z\big)-\fint_{B_1(0)}u\big(x+\e_n
y\big)dy\Bigg|^qdz
\\
\leq\liminf\limits_{n\to \infty}\fint_{B_{1}(0)}\fint_{B_1(0)}
\Big|u\big(x+\e_n z\big)-u\big(x+\e_n
y\big)\Big|^qdydz=0\quad\quad\forall x\in\Omega_0\setminus S_0\,.
\end{multline}
Let $\Omega=\bigcup_{j=1}^\infty\Omega_j$ for open sets $\Omega_j\subset\subset\Omega$. For every $j$ there exists a $\mathcal{H}^{N-rq}$
$\sigma$-finite set $S_j$ such that \eqref{eq:fine property in local case} holds for $x\in \Omega_j\setminus S_j$. Denote $S:=\bigcup_{j=1}^\infty S_j$. Then $S$ is $\mathcal{H}^{N-rq}$
$\sigma$-finite set and \eqref{eq:fine property in local case} holds for $x\in \Omega\setminus S$.

Now, if $u\in BV^q_{loc}(\Omega,\R^d)$, then $u\in L^q_{loc}(\Omega,\R^d)$ and for every $\Omega_0\subset\subset\Omega$
\begin{multline}
\infty>\overline{B}_{u,q}(\Omega_0)=\limsup_{\e\to 0^+}\int_{\Omega_0}\left(\int_{\Omega_0\cap B_{\e}(x)}\frac{1}{\e^N}\frac{|u(x)-u(y)|^q}{|x-y|}dy\right)dx
\\
=\limsup\limits_{\e\to
0^+}\int_{\Omega_0}\Bigg\{\int_{B_1(0)}\chi_{\Omega_0}\big(x+\e z\big)
\frac{\Big|u\big(x+\e z\big)-u(x)\Big|^q}{\e|z|}dz\Bigg\}dx
\\
\geq \alpha(N)\limsup\limits_{\e\to
0^+}\int_{\Omega_0}\Bigg\{\fint_{B_1(0)}\chi_{\Omega_0}\big(x+\e z\big)
\frac{\Big|u\big(x+\e z\big)-u(x)\Big|^q}{\e}dz\Bigg\}dx.
\end{multline}
Thus, functions $u\in BV^q_{loc}(\Omega,\R^d)$ satisfy the conditions of the theorem with $r=\frac{1}{q}$.
\end{proof}

\begin{corollary}
\label{cor: Besov constant controls the q-variation}
Let $\Omega\subset\R^N$ be an open set and $q>0$ and suppose
$u\in L^1_{loc}(\Omega,\R^d)$. Then, there exists a constant $\gamma(N)$,
depending on dimension $N$ only, such that
\begin{equation}
\underline{B}_{u,q}(\Omega)\geq\gamma(N)
\int_{\mathcal{J}'_u} \Big|u^{j'}(x)\Big|^qd\mathcal{H}^{N-1}(x),
\end{equation}
where
$\underline{B}_{u,q}(\Omega)$
is the lower infinitesimal Besov $q-$constant defined in Definition $\ref{def:definitions of Besov constants}$.
\end{corollary}

\begin{proof}
Let us denote
\begin{equation}
\gamma_{x,\varepsilon}(y,z):=u(x+\varepsilon y)-u(x+\varepsilon z).
\end{equation}
By Proposition
$\ref{prop:convergence of the rescaled family gamma}$ we get for $x\in \mathcal{J}'_u$

\begin{align}
&\liminf\limits_{\e\to 0^+}\fint_{B_{1}(0)\times B_{1}(0)}\Big|\gamma_{x,\varepsilon}(y,z)\Big|^qd(\mathcal{L}^N\times\mathcal{L}^N)(y,z)\geq \frac{1}{2}\Big|u^{j'}(x)\Big|^q.
\end{align}

Thus, by Fubini's theorem
\begin{align}
\liminf\limits_{\e\to 0^+}\fint_{B_{1}(0)}\fint_{B_1(0)}
\Big|u\big(x+\e z\big)-u\big(x+\e
y\big)\Big|^qdydz\geq\frac{1}{2}\Big|u^{j'}(x)\Big|^q.
\end{align}

Let $\varepsilon_n>0$ be a sequence such that $\lim_{n\to \infty}\varepsilon_n=0$ and
\begin{align}
\underline{B}_{u,q}(\Omega)=\liminf_{\epsilon\to 0^+}\int_{\Omega}\left(\int_{\Omega\cap B_{\epsilon}(x)}\frac{1}{\epsilon^N}\frac{|u(x)-u(y)|^q}{|x-y|}dy\right)dx\nonumber
\\
=\liminf_{n\to \infty}\int_{\Omega}\left(\int_{\Omega\cap B_{\varepsilon_n}(x)}\frac{1}{\varepsilon_n^N}\frac{|u(x)-u(y)|^q}{|x-y|}dy\right)dx.
\end{align}

By Theorem
$\ref{thm:Besov constant controls double average integral}$ with $r=\frac{1}{q}$
\begin{multline}
\underline{B}_{u,q}(\Omega)=\liminf_{n\to \infty}\int_{\Omega}\left(\int_{\Omega\cap B_{\varepsilon_n}(x)}\frac{1}{\varepsilon_n^N}\frac{|u(x)-u(y)|^q}{|x-y|}dy\right)dx
\\
=\liminf\limits_{n\to \infty}\int_\Omega\Bigg\{\int_{B_1(0)}\chi_\Omega\big(x+\e_n
z\big) \frac{\Big|u\big(x+\e_n
z\big)-u(x)\Big|^q}{\e_n|z|}dz\Bigg\}dx
\\
\geq\alpha(N)\liminf\limits_{n\to \infty}\int_\Omega\Bigg\{\fint_{B_1(0)}\chi_\Omega\big(x+\e_n
z\big) \frac{\Big|u\big(x+\e_n
z\big)-u(x)\Big|^q}{\e_n}dz\Bigg\}dx
\\
\geq \alpha(N)\beta(N;1/q,q)\int_\Omega \Bigg\{\liminf\limits_{n\to
\infty}\fint_{B_{1}(0)}\fint_{B_1(0)} \Big|u\big(x+(\e_n/2)
z\big)-u\big(x+(\e_n/2) y\big)\Big|^qdzdy\Bigg\}d\mathcal{H}^{N-1}(x)
\\
\geq \alpha(N)\beta(N;1/q,q)\int_\Omega \Bigg\{\liminf\limits_{\varepsilon\to 0^+}\fint_{B_{1}(0)}\fint_{B_1(0)} \Big|u\big(x+\e
z\big)-u\big(x+\e y\big)\Big|^qdzdy\Bigg\}d\mathcal{H}^{N-1}(x)
\\
\geq \alpha(N)\beta(N;1/q,q)\int_{\mathcal{J}'_u} \Bigg\{\liminf\limits_{\varepsilon\to 0^+}\fint_{B_{1}(0)}\fint_{B_1(0)} \Big|u\big(x+\e
z\big)-u\big(x+\e y\big)\Big|^qdzdy\Bigg\}d\mathcal{H}^{N-1}(x)
\\
\geq \frac{\alpha(N)\beta(N;1/q,q)}{2}\int_{\mathcal{J}'_u}\Big|u^{j'}(x)\Big|^qd\mathcal{H}^{N-1}(x).
\end{multline}
Denote $\gamma(N):=\frac{\alpha(N)\beta(N;1/q,q)}{2}=\frac{\alpha(N)^2}{5^{N-1}2^{2N+1}\alpha(N-1)}$.
\end{proof}

\begin{remark}
\label{rem:Gamma(N)<C_N}
The constant $\gamma(N)$ satisfies $\gamma(N)<C_N$, where $C_N$ is defined in \eqref{eq:definition of dimensional constant C_N}:
by the area formula
\begin{align}
C_N=\frac{1}{N}\int_{S^{N-1}}|z_1|d\mathcal{H}^{N-1}(z)&=\frac{2}{N}\int_{B^{N-1}_1(0)}\sqrt{1-\sum_{j=2}^{N}z^2_j}\,\sqrt{1+\frac{1}{1-\sum_{j=2}^{N}z^2_j}\sum_{j=2}^{N}z^2_j}\,d\mathcal{L}^{N-1}(z)\nonumber
\\
&=\frac{2}{N}\Big(\mathcal{L}^{N-1}\big(B^{N-1}_1(0)\big)\Big)=\frac{2}{N}\alpha(N-1).
\end{align}
Thus
\begin{equation}
\label{eq:equivalence to gamma(N)<C_N}
\gamma(N)<C_N\quad \Longleftrightarrow \quad \left[\frac{\alpha(N)}{\alpha(N-1)}\right]^2N<5^{N-1}2^{2N+2}.
\end{equation}
One can use the properties of gamma function $\Gamma(s+1)=s\Gamma(s)$ for $s>0$,
$\Gamma(m)=(m-1)!$ for $m\in \N$ and $\Gamma\left(\frac{1}{2}\right)=\sqrt{\pi}$ to get an estimate
$\frac{\Gamma\left(\frac{N-1}{2}+1\right)}{\Gamma\left(\frac{N}{2}+1\right)}\leq \pi^{1/2}$, and thus

\begin{equation}
\frac{\alpha(N)}{\alpha(N-1)}=\frac{\frac{\pi^{N/2}}{\Gamma\left(\frac{N}{2}+1\right)}}{\frac{\pi^{(N-1)/2}}{\Gamma\left(\frac{N-1}{2}+1\right)}}=\pi^{1/2}\frac{\Gamma\left(\frac{N-1}{2}+1\right)}{\Gamma\left(\frac{N}{2}+1\right)}\leq \pi.
\end{equation}

So the inequality on the right hand side of \eqref{eq:equivalence to gamma(N)<C_N} can be proved by proving $\pi^2N<16N\leq 5^{N-1}2^{2N+2}$ by induction on $N$.
\end{remark}

\subsection{Jumps of $BV^q$-functions}
By Theorem \ref{thm:Besov constant in terms of jump} and Remark \ref{rem:Gamma(N)<C_N} we see that the constant $\gamma(N)$ from Corollary \ref{cor: Besov constant controls the q-variation} is not optimal. The optimal constant is $C_N$ as we will show in the following theorem. Theorem \ref{thm:main result1} is an immediate consequence of Theorem \ref{thm:main result,general formulation} and Definition \ref{def:approximate jump pointsgint}.

\begin{theorem}
\label{thm:main result,general formulation}
Let $\Omega\subset \mathbb{R}^N$  be an open set and $1\leq q<\infty$. Assume that  $u\in L^1_{loc}(\Omega,\R^d)$. Then,
\begin{align}
\label{eq:ineqality3}
C_N\int_{\mathcal{J}'_u}|u^{j'}(x)|^qd\Haus^{N-1}(x)\leq \underline{B}_{u,q}\left(\Omega\right).
\end{align}
In particular, \eqref{eq:ineqality3} holds for functions $u\in BV^q(\Omega,\R^d)$.
We call inequality \eqref{eq:ineqality3} the \textbf{$q$-jump inequality } for the function $u$ in the open set $\Omega$.
\end{theorem}
\begin{proof}
The proof is divided into three steps.
Our idea is to use the rectifiability structure of the generalized jump set $\mathcal{J}'_u$.
We begin our analysis by decomposing the generalized jump set $\mathcal{J}'_u$ into countably many $(N-1)$-submanifolds $\Gamma^i$ of class $1$.
Then, we look at a small neighbourhood $\Gamma^i_{x_0}\subset \Gamma^i$ of a point $x_0\in \Gamma^i$, and translate it in a non tangent direction $n$ getting a set of translations $O_{t,n}\left(\Gamma_{x_0}^i\right)$ (see $\eqref{eq:union of translations of neighbourhood of x_0}$).
With this set of translations, the co-area formula and Besicovitch covering theorem we derive a local version of the desired inequality (see $\eqref{eq:local Becov inequality where C_N appears}$).
We use the local inequality and the assumption that $\underline{B}_{u,q}\left(\Omega\right)<\infty$ to get an auxiliary measure $\mu$.
We will use Radon-Nikodim derivative of the measure $\lambda(E):=C_N\int_{E\cap \mathcal{J}'_u}|u^{j'}(x)|^qd\mathcal{H}^{N-1}(x)$ with respect to $\mu$ to get the desired global inequality. By "global" here we mean on the whole generalized jump set $\mathcal{J}'_u$, not only on a part of it.
\\
\textbf{Step 1:}
\\
Since $u\in L^1_{loc}(\Omega,\mathbb{R}^d)$, we get by Remark $\ref{rem:rectifiability of generalized jump set}$ that $\mathcal{J}'_u$ is countably $(N-1)-$rectifiable. Using Theorem
$\ref{thm:rectifiable set is a union of smooth manifold}$, we have a countable family of an $(N-1)-$dimensional submanifold of $\R^N$ of class 1, $\Gamma^i$, such that for each $i\in \mathbb{N}$,
$\Haus^{N-1}(\Gamma^i)<\infty$ and
\begin{equation}
\label{eq:rectifiability of J_u}
\mathcal{H}^{N-1}\left(\mathcal{J}'_u\setminus \cup_{i=1}^\infty\Gamma^i\right)=0.
\end{equation}
Assume without loss of generality that $\Gamma^i\subset \Omega$, otherwise replace $\Gamma^i$ by $\Gamma^i\cap \Omega$.
Let us fix $x_0\in \cup_{i=1}^\infty\Gamma^i$. Then, there exists $i\in \N$ such that $x_0\in \Gamma^i$. Let
$\nu:\Gamma^i\to S^{N-1}$ be a (continuous)unit normal field along $\Gamma^i$. By Lemma
$\ref{lem:parametric surface relative to vector $n$}$, there exist an open set $\Gamma^i_{x_0} \subset\Gamma^i$ (in the induced topology from $\mathbb{R}^N$ to $\Gamma^i$), $x_0\in \Gamma^i_{x_0}$, such that for every $n_0\notin \nu(x_0)^\perp$ there exist $t_0>0$ and $\delta_0>0$ such that for every $0<t<t_0$ and $n\in \overline{B}_{\delta_0}(n_0)\subset\mathbb{R}^N\setminus \nu(x_0)^\perp$ there exist an open set
\begin{equation}
\label{eq:union of translations of neighbourhood of x_0}
O_{t,n}(\Gamma^i_{x_0}):=\cup_{s\in (-t,t)}(\Gamma^i_{x_0}-sn)
\end{equation}
and a $C^1$ and Lipschitz function
\begin{equation}
f:O_{t,n}(\Gamma^i_{x_0})\to (-t,t),
\end{equation}
such that for every
$z:=x-sn\in O_{t,n}(\Gamma^i_{x_0})$ we have $f(z)=s$ and $|\nabla f(z)|=\frac{1}{|\nu(z+f(z)n)\cdot n|}$.

Let us extend the Lipschitz function
$f$ from $O_{t,n}\left(\Gamma_{x_0}^i\right)$ to a Lipschitz function defined on all of $\R^N$ again denoting it by $f$; and let us think of $u$ as a function on all of $\mathbb{R}^N$ defining
$u(x)=0,x\in \mathbb{R}^N\setminus \Omega$.

Let us denote for every $0<t<t_0$, $n\in \overline{B}_{\delta_0}(n_0)$ and arbitrary Borel set $B\subset\Omega$
\begin{align}
F_{t,n}(x,s):=\chi_{\Omega}(x+(1-s)tn)\chi_B(x-stn)|u(x+(1-s)tn)-u(x-stn)|^q|\nu(x)\cdot n|,
\end{align}
where $s\in (0,1)$ and $x\in\Gamma^i$.
By Theorem $\ref{thm:the co-area formula}$ (the co-area  formula) we have
\begin{align}
\label{eq:inequality9}
&\int_0^1\left[\int_{\Gamma^i_{x_0}\cap \mathcal{J}'_u}F_{t,n}(x,s)d\Haus^{N-1}(x)\right]ds
\leq\int_0^1\left[\int_{\Gamma^i_{x_0}}F_{t,n}(x,s)d\Haus^{N-1}(x)\right]ds\nonumber
\\
&=\int_0^1\left[\int_{\Gamma^i_{x_0}-stn}F_{t,n}(y+stn,s)d\Haus^{N-1}(y)\right]ds\nonumber
=\frac{1}{t}\int_0^t\left[\int_{\Gamma^i_{x_0}-\tau n}F_{t,n}(y+\tau n,\tau/t)d\Haus^{N-1}(y)\right]d\tau\nonumber
\\
&\leq \frac{1}{t}\int_{\R}\left[\int_{f^{-1}(\{\tau\})\cap O_{t,n}\left(\Gamma^i_{x_0}\right)}F_{t,n}(y+\tau n,\tau/t)d\Haus^{N-1}(y)\right]d\tau\nonumber
\\
&=\frac{1}{t}\int_{\R}\left[\int_{f^{-1}(\{\tau\})\cap O_{t,n}\left(\Gamma^i_{x_0}\right)}F_{t,n}(y+ f(y)n,f(y)/t)d\Haus^{N-1}(y)\right]d\tau\nonumber
\\
&=\frac{1}{t}\int_{O_{t,n}\left(\Gamma^i_{x_0}\right)}F_{t,n}(x+ f(x)n,f(x)/t)|\nabla f(x)|dx\nonumber
\\
&=\frac{1}{t}\int_{O_{t,n}\left(\Gamma^i_{x_0}\right)}\chi_{\Omega}(x+tn)\chi_B(x)|u(x+tn)-u(x)|^qdx\nonumber
\\
&=\frac{1}{t}\int_{O_{t,n}\left(\Gamma^i_{x_0}\right)\cap B}\chi_{\Omega}(x+tn)|u(x+tn)-u(x)|^qdx.
\end{align}

Multiplying inequality \eqref{eq:inequality9} by $1/|n|$, taking the integral over
$\overline{B}_{\delta_0}(n_0)$ and using Fubini's theorem we obtain for every $0<t<t_0$
\begin{align}
\label{eq:intermediate inequlity1}
&\int_0^1\int_{\Gamma^i_{x_0}\cap \mathcal{J}'_u}\left[\int_{\overline{B}_{\delta_0}(n_0)}\frac{F_{t,n}(x,s)}{|n|}dn\right]d\Haus^{N-1}(x)ds\nonumber
\\
&\leq \int_{\overline{B}_{\delta_0}(n_0)}\int_{O_{t,n}\left(\Gamma^i_{x_0}\right)\cap B}\chi_{\Omega}(x+tn)\frac{|u(x+tn)-u(x)|^q}{t|n|}dxdn.
\end{align}

Let us denote for $s\in (0,1)$ and
$x\in \Gamma^i_{x_0}\cap \mathcal{J}'_u$
\begin{align}
&g_t(n):=\left(\frac{F_{t,n}(x,s)}{|n|}\right)^{1/q}\nonumber
\\
&=\chi_{\Omega}(x+(1-s)tn)\chi_B(x-stn)|u(x+(1-s)tn)-u(x-stn)|\left|\nu(x)\cdot \frac{n}{|n|}\right|^{1/q},\nonumber
\\
&g(n):=\chi_B(x)|u^{j'}(x)|\left|\nu(x)\cdot \frac{n}{|n|}\right|^{1/q}.
\end{align}

Assume that $\mathcal{H}^{N-1}\left(\partial B\cap \Gamma^i_{x_0}\right)=0$. Thus, for $\mathcal{H}^{N-1}$ a.e. $x\in \Gamma^i_{x_0}$, we have $\lim_{t\to 0^+}\chi_{B}(x-stn)=\chi_{B}(x)$ for every $s\in \R$ and $n\in \R^N$. Thus, for every $s\in (0,1)$ and $\mathcal{H}^{N-1}$ a.e. $x\in \Gamma^i_{x_0}\cap \mathcal{J}'_u$ we get by Proposition $\ref{prop:convergence of the rescaled family gamma}$ that $g_t$ converges to $g$ in $L^1_{loc}(\R^N)$, hence by lower semi-continuity of $L^q$ norm with respect to convergence in $L^1_{loc}(\R^N)$ we get
\begin{equation}
\label{eq:lower semi continuity of g_t}
\|g\|^q_{L^q(\overline{B}_{\delta_0}(n_0))}\leq \liminf_{t\to 0^+}\|g_t\|^q_{L^q(\overline{B}_{\delta_0}(n_0))}.
\end{equation}
Therefore, taking the lower limit as $t\to 0^+$ on both sides of $\eqref{eq:intermediate inequlity1}$ and using Fatou's lemma, Fubini's theorem and \eqref{eq:lower semi continuity of g_t} we obtain
\begin{align}
\label{eq:intermediate inequlity2}
\int_{\overline{B}_{\delta_0}(n_0)}\int_{\Gamma^i_{x_0}\cap \mathcal{J}'_u\cap B}|u^{j'}(x)|^q\left|\nu(x)\cdot \frac{n}{|n|}\right|d\Haus^{N-1}(x)dn\nonumber
\\
\leq \liminf_{t\to 0^+}\int_{\overline{B}_{\delta_0}(n_0)}\int_{O_{t,n}\left(\Gamma^i_{x_0}\right)\cap B}\chi_{\Omega}(x+tn)\frac{|u(x+tn)-u(x)|^q}{t|n|}dxdn.
\end{align}

Note that inequality
$\eqref{eq:intermediate inequlity2}$ is valid not only for $\delta_0$, but also for every
$0<\delta\leq \delta_0$.
Let us choose for each $w\in \R^N\setminus \nu(x_0)^\perp$ an appropriate $\delta_w>0$ for which the inequality
$\eqref{eq:intermediate inequlity2}$ holds for balls $\overline{B}_{\delta}(w),0<\delta\leq \delta_w$ (in place of $\overline{B}_{\delta_0}(n_0)$). Let us define a family of closed balls:
\begin{equation}
\mathcal{F}:=\left\{\overline{B}_{\delta}(w)\subset B_1(0):w\in B_1(0)\setminus \nu(x_0)^\perp,0<\delta\leq \delta_w\right\}.
\end{equation}
By Theorem $\ref{thm:Besicovitch covering theorem}$ (Besicovitch covering theorem), there exists a countable and pairwise disjoint subfamily $\{B_j\}_{j=1}^\infty\subset \mathcal{F}$ such that
\begin{equation}
\mathcal{L}^N\left(B_1(0)\setminus \bigcup_{j=1}^\infty B_j\right)=\mathcal{L}^N\left(\left(B_1(0)\setminus \nu(x_0)^\perp\right)\setminus \bigcup_{j=1}^\infty B_j\right)=0,
\end{equation}
where $\mathcal{L}^N\left(\nu(x_0)^\perp\right)=0$.
Therefore,
\begin{align}
\label{eq:local Besov inequality}
&\int_{B_1(0)}\int_{\Gamma^i_{x_0}\cap \mathcal{J}'_u\cap B}|u^{j'}(x)|^q\left|\nu(x)\cdot \frac{n}{|n|}\right|d\Haus^{N-1}(x)dn\nonumber
\\
&=\sum_{j=1}^\infty\int_{B_j}\int_{\Gamma^i_{x_0}\cap \mathcal{J}'_u\cap B}|u^{j'}(x)|^q\left|\nu(x)\cdot \frac{n}{|n|}\right|d\Haus^{N-1}(x)dn\nonumber
\\
&\leq \sum_{j=1}^\infty\liminf_{t\to 0^+}\int_{B_j}\int_{O_{t,n}\left(\Gamma^i_{x_0}\right)\cap B}\chi_{\Omega}(x+tn)\frac{|u(x+tn)-u(x)|^q}{t|n|}dxdn\nonumber
\\
&\leq \liminf_{t\to 0^+}\sum_{j=1}^\infty\int_{B_j}\int_{O_{t,n}\left(\Gamma^i_{x_0}\right)\cap B}\chi_{\Omega}(x+tn)\frac{|u(x+tn)-u(x)|^q}{t|n|}dxdn\nonumber
\\
&=\liminf_{t\to 0^+}\int_{B_1(0)}\int_{O_{t,n}\left(\Gamma^i_{x_0}\right)\cap B}\chi_{\Omega}(x+tn)\frac{|u(x+tn)-u(x)|^q}{t|n|}dxdn.
\end{align}
By Fubini's theorem and Remark
$\ref{rem:another form for C_N}$ after the proof
\begin{align}
\label{eq:equation where C_N is appeared}
&\int_{B_1(0)}\int_{\Gamma^i_{x_0}\cap \mathcal{J}'_u\cap B}|u^{j'}(x)|^q\left|\nu(x)\cdot \frac{n}{|n|}\right|d\Haus^{N-1}(x)dn\nonumber
\\
&=\int_{\Gamma^i_{x_0}\cap \mathcal{J}'_u\cap B}|u^{j'}(x)|^q\left(\int_{B_1(0)}\left|\nu(x)\cdot \frac{n}{|n|}\right|dn\right)d\Haus^{N-1}(x)\nonumber
\\
&=\left(\int_{B_1(0)}\left|e_1\cdot \frac{n}{|n|}\right|dn\right)\int_{\Gamma^i_{x_0}\cap \mathcal{J}'_u\cap B}|u^{j'}(x)|^qd\Haus^{N-1}(x)\nonumber
\\
&=C_N\int_{\Gamma^i_{x_0}\cap \mathcal{J}'_u\cap B}|u^{j'}(x)|^qd\Haus^{N-1}(x),
\end{align}
where $C_N$ is defined in
$\eqref{eq:definition of dimensional constant C_N}$ and $e_1=(1,0,0,...,0)\in \mathbb{R}^N$. Thus, $\eqref{eq:local Besov inequality}$ and $\eqref{eq:equation where C_N is appeared}$ give
\begin{align}
\label{eq:local Becov inequality where C_N appears}
&C_N\int_{\Gamma^i_{x_0}\cap \mathcal{J}'_u\cap B}|u^{j'}(x)|^qd\Haus^{N-1}(x)\nonumber
\\
&\leq \liminf_{t\to 0^+}\int_{B_1(0)}\int_{O_{t,n}\left(\Gamma^i_{x_0}\right)\cap B}\chi_{\Omega}(x+tn)\frac{|u(x+tn)-u(x)|^q}{t|n|}dxdn\nonumber
\\
&\leq \liminf_{t\to 0^+}\int_{B_1(0)}\int_{B}\chi_{\Omega}(x+tn)\frac{|u(x+tn)-u(x)|^q}{t|n|}dxdn.
\end{align}
We emphasize that inequality
$\eqref{eq:local Becov inequality where C_N appears}$ holds for every $i\in \mathbb{N}$, every $x_0\in\Gamma^i$, every small enough neighbourhood $\Gamma^i_{x_0}\subset \Gamma^i$ of $x_0$ and every Borel set $B\subset \Omega$ with $\mathcal{H}^{N-1}\left(\partial B\cap \Gamma^i_{x_0}\right)=0$.

Notice that if we choose
$B=B_\rho(x_0)$, then since $\Gamma^i_{x_0}$ is open set in $\Gamma^i$ we have $\Gamma^i_{x_0}\cap B_\rho(x_0)=\Gamma^i\cap B_\rho(x_0)$ for every small enough $\rho>0$, and since $\mathcal{H}^{N-1}(\Gamma^i_{x_0})<\infty$, we get by Lemma $\ref{lem: F is at most countable}$ that
$\mathcal{H}^{N-1}\left(\partial B_\rho(x_0)\cap \Gamma^i_{x_0}\right)=\mathcal{H}^{N-1}\llcorner \Gamma^i_{x_0}\left(\partial B_\rho(x_0)\right)=0$ for every $\rho>0$ except for a countable set.
\\
\textbf{Step 2:}
\\
If $\underline{B}_{u,q}(\Omega)=\infty$, then inequality \eqref{eq:ineqality3} trivially holds. Assume $\underline{B}_{u,q}(\Omega)<\infty$.
Let us denote for every $t>0$
\begin{equation}
\label{eq:definition of f rho}
f^{t}(x):=\int_{B_1(0)}\chi_{\Omega}(x+tn)\frac{\left|u(x+tn)-u(x)\right|^q}{t|n|}dn,\quad x\in\Omega.
\end{equation}
Thus
\begin{align}
\liminf_{t\to 0^+}f^t\mathcal{L}^N(\Omega)=\liminf_{t\to 0^+}\int_{\Omega}\int_{B_1(0)}\chi_{\Omega}(x+t n)\frac{\left|u(x+t n)-u(x)\right|^q}{t |n|}dndx\nonumber
\\
=\liminf_{t\to 0^+}\int_{\Omega}\frac{1}{t^N}\int_{\Omega\cap B_t(x)}\frac{\left|u(y)-u(x)\right|^q}{|y-x|}dydx=\underline{B}_{u,q}(\Omega)<\infty.
\end{align}

Thus, by weak* compactness there exists a subsequence  $f^{t_l}\Leb^N$ weakly* converging to some finite Radon measure
$\mu$ in $\Omega$ as $l\to \infty$ and such that
\begin{equation}
\label{eq:lower Besov term is obtained by a sequence}
\lim_{l\to \infty}f^{t_l}\Leb^N(\Omega)=\liminf_{t\to 0^+}f^{t}\Leb^N(\Omega).
\end{equation}

Notice that for every $x_0\in \Omega$, since $\mu$ is finite Radon measure in $\Omega$ we get by Lemma $\ref{lem: F is at most countable}$ that $\mu(\partial B_\rho(x_0))=0$ for every $\rho>0$ except for a countable set and such that $B_\rho(x_0)\subset \Omega$. In particular, we have for every such $\rho$ that
$\lim_{l\to \infty}f^{t_l}\Leb^N(B_{\rho}(x_0))
=\mu(B_{\rho}(x_0))$.
\\
\textbf{Step 3:}
\\
Let us denote for $E\subset\Omega$
\begin{equation}
\lambda(E):=C_N\int_{E\cap \mathcal{J}'_u}|u^{j'}(x)|^qd\Haus^{N-1}(x).
\end{equation}

For every $i\in \N$ and every $x_0\in \Gamma^i$ let $\rho_m>0$ be a converging to zero sequence such that
$\Gamma^i_{x_0}\cap B_{\rho_m}(x_0)=\Gamma^i\cap B_{\rho_m}(x_0)$, $\mathcal{H}^{N-1}\left(\partial B_{\rho_m}(x_0)\cap \Gamma^i_{x_0}\right)=0$ and
$\mu(\partial B_{\rho_m}(x_0))=0$ for every $m\in\N$.

By inequality $\eqref{eq:local Becov inequality where C_N appears}$ with $B=B_{\rho_m}(x_0)$ we get for every $m\in \N$
\begin{align}
\lambda(\Gamma^i\cap B_{\rho_m}(x_0))&\leq \liminf_{l\to \infty}\int_{B_1(0)}\int_{B_{\rho_m}(x_0)}\chi_{\Omega}(x+t_l n)\frac{|u(x+t_l n)-u(x)|^q}{t_l|n|}dxdn\nonumber
\\
&=\liminf_{l\to \infty}f^{t_l}\Leb^N(B_{\rho_m}(x_0))=\lim_{l\to \infty}f^{t_l}\Leb^N(B_{\rho_m}(x_0))
=\mu(B_{\rho_m}(x_0)).
\end{align}

By Lemma $\ref{lem:local measure inequality for rectifiable sets gives global ineqality}$
\begin{align}
&C_N\int_{\mathcal{J}'_u}|u^{j'}(x)|^qd\Haus^{N-1}(x)=\lambda(\mathcal{J}'_u)\leq \mu(\Omega)\leq \liminf_{l\to \infty}f^{t_l}\Leb^N(\Omega)\nonumber
\\
&=\lim_{l\to \infty}f^{t_l}\Leb^N(\Omega)=\liminf_{t\to 0^+}f^{t}\Leb^N(\Omega)\nonumber
&\\
&=\liminf_{t\to 0^+}\int_{\Omega}\int_{B_1(0)}\chi_{\Omega}(x+t n)\frac{|u(x+t n)-u(x)|^q}{t|n|}dndx\nonumber
&\\
&=\liminf_{t\to 0^+}\int_{\Omega}\frac{1}{t^N}\int_{\Omega\cap B_t(x)}\frac{|u(y)-u(x)|^q}{|y-x|}dydx=\underline{B}_{u,q}(\Omega).
\end{align}
\end{proof}

\begin{remark}
\label{rem:another form for C_N}
Note that the dimensional constant $C_N$ as defined in
$\eqref{eq:definition of dimensional constant C_N}$ can be also represented as
\begin{equation}
C_N=\int_{B_1(0)}\left|e_1\cdot \frac{n}{|n|}\right|dn,
\end{equation}
where $e_1:=(1,0,0,...,0)$. Indeed, by polar coordinates we have
\begin{align}
\int_{B_1(0)}\left|e_1\cdot \frac{n}{|n|}\right|dn=\int_0^1\frac{1}{r}\left(\int_{\partial B_r(0)}\left|e_1\cdot n\right|d\mathcal{H}^{N-1}(n)\right)dr\nonumber
\\
=\int_0^1r^{N-1}\left(\int_{S^{N-1}}\left|e_1\cdot w\right|d\mathcal{H}^{N-1}(w)\right)dr
=C_N.
\end{align}
Note also that for every $v_1,v_2\in S^{N-1}$
\begin{equation}
\int_{B_1(0)}\left|v_1\cdot \frac{n}{|n|}\right|dn=\int_{B_1(0)}\left|v_2\cdot \frac{n}{|n|}\right|dn.
\end{equation}
Indeed, take an isometry $A:\mathbb{R}^N\to \mathbb{R}^N$ such that $A(v_2)=v_1$.
Then, by the change of variable formula
\begin{equation}
\int_{B_1(0)}\left|v_1\cdot \frac{n}{|n|}\right|dn=\int_{A^{-1}\left(B_1(0)\right)}\left|A(v_2)\cdot \frac{A(w)}{|A(w)|}\right|dw=\int_{B_1(0)}\left|v_2\cdot \frac{w}{|w|}\right|dw.
\end{equation}
\end{remark}

\subsection{Examples  of $BV^q$-functions for which the $q$-jump inequality is strict}
\label{subsec:Examples  of $BV^q$-functions for which the $q$-jump inequality is strict}
Proposition
$\ref{prop:example when Besov inequality is strict}$ and Remarks $\ref{rem:remark about Assouad's theorem},\ref{rem:remark about functions on the whole space where Besov inequality is strict}$ below supply us with examples where the inequality
$\eqref{eq:ineqality3}$ is strict.

\begin{proposition}
\label{prop:example when Besov inequality is strict}
Let $1\leq q<\infty$. Assume the existence of a function $u:\mathbb{R}^N\to \mathbb{R}^d$ such that there exist constants
$0<A_1,A_2<\infty$ such that
\begin{equation}
\label{eq:u is bi-Holder}
A_1|x-y|^{1/q}\leq |u(x)-u(y)|\leq A_2|x-y|^{1/q},\quad \forall x,y\in \mathbb{R}^N.
\end{equation}
Then, for every
$\Omega\subset \mathbb{R}^N$ open, bounded and nonempty we have
\begin{equation}
0<\alpha(N)A_1^q\mathcal{L}^N(\Omega)\leq \underline{B}_{u,q}(\Omega)\leq \overline{B}_{u,q}(\Omega)\leq \alpha(N)A_2^q\mathcal{L}^N(\Omega)<\infty,
\end{equation}
where $\alpha(N)=\mathcal{L}^N(B_1(0))$ stands for the volume of the unit ball in $\R^N$. In particular, since $u$ is continuous in
$\mathbb{R}^N$, we have $u\in L^q(\Omega,\mathbb{R}^d)$ and by Proposition
$\ref{prop:equivalence of finiteness of upper and Besov constants}$ we get
$u\in BV^q(\Omega,\mathbb{R}^d)$; since $u$ is continuous we get $\mathcal{J}'_u=\emptyset$, so the inequality
$\eqref{eq:ineqality3}$ is strict.
\end{proposition}

\begin{proof}
It follows by $\eqref{eq:u is bi-Holder}$
\begin{align}
\label{eq:ineqality4}
\int_{\Omega}\left(\int_{\Omega\cap B_{\epsilon}(x)}\frac{1}{\epsilon^N}\frac{|u(x)-u(y)|^q}{|x-y|}dy\right)dx\geq A_1^q\int_{\Omega}\left(\int_{\Omega\cap B_{\epsilon}(x)}\frac{1}{\epsilon^N}dy\right)dx\nonumber
\\
=\alpha(N)A_1^q\int_{\Omega}\frac{\mathcal{L}^N\left(\Omega\cap B_{\epsilon}(x)\right)}{\mathcal{L}^N(B_\epsilon(x))}dx.
\end{align}

Let us denote
$g_\epsilon(x):=\frac{\mathcal{L}^N\left(\Omega\cap B_{\epsilon}(x)\right)}{\mathcal{L}^N(B_\epsilon(x))}$ for $x\in \Omega$. It follows that $|g_\epsilon(x)|\leq 1$ for every $x\in \Omega$ and every $\epsilon>0$, and since $\Omega$ is open we get $\lim_{\epsilon\to 0^+}g_\epsilon(x)=1$ for every $x\in \Omega$. Taking the lower limit as $\epsilon\to 0^+$ on both sides of inequality $\eqref{eq:ineqality4}$ and using dominated convergence theorem we get
$\alpha(N)A_1^q\mathcal{L}^N(\Omega)\leq \underline{B}_{u,q}(\Omega)$. The inequality
$\overline{B}_{u,q}(\Omega)\leq \alpha(N)A_2^q\mathcal{L}^N(\Omega)$ follows by similar considerations.
\end{proof}

\begin{remark}
\label{rem:remark about Assouad's theorem}
By Assouad's theorem, for
$\mathbb{R}^N$ as a metric space endowed with the metric $d(x,y):=|x-y|^{\alpha}$, $\alpha\in (0,1)$, there exist $0<A_1,A_2<\infty$ and
$d\in \mathbb{N}$ and bi-Lipschitz function $u:\mathbb{R}^N\to \mathbb{R}^d$:
\begin{equation}
A_1|x-y|^{\alpha}\leq |u(x)-u(y)|\leq A_2|x-y|^{\alpha},\quad \forall x,y\in \mathbb{R}^N.
\end{equation}
A proof of this theorem can be found in
\cite{Assouad} or see Theorem 12.2. in \cite{JuhaHeinonen}.
\end{remark}

\begin{remark}
\label{rem:remark about functions on the whole space where Besov inequality is strict}
There exist functions
$f\in BV^q(\mathbb{R}^N,\mathbb{R}^d)\cap L^\infty(\mathbb{R}^N,\mathbb{R}^d)$ for which the inequality $\eqref{eq:ineqality3}$ is strict: by Assouad's theorem for $\mathbb{R}^N$ together with the metric $d(x,y):=|x-y|^{1/q}$, $1<q<\infty$, there exist $d\in \mathbb{N}$, $A_1,A_2\in(0,\infty)$ and a bi-Lipschitz function $u:\mathbb{R}^N\to \mathbb{R}^d$ such that inequality $\eqref{eq:u is bi-Holder}$ holds. Let $\zeta\in C^{0,1/q}_c(\mathbb{R}^N)$ (an H{\"o}lder continuous function with exponent ${1/q}$ and compact support in $\R^N$) which is identically $1$ on an open, bounded and nonempty set $\Omega\subset\mathbb{R}^N$.   Define $f:=u\zeta$. Hence,
$f\in L^\infty(\mathbb{R}^N,\mathbb{R}^d)$ and
\begin{align}
\label{eq:ineqality5}
\int_{\mathbb{R}^N}\left(\int_{B_{\epsilon}(x)}\frac{1}{\epsilon^N}\frac{|f(x)-f(y)|^q}{|x-y|}dy\right)dx\geq \int_{\Omega}\left(\int_{\Omega\cap B_{\epsilon}(x)}\frac{1}{\epsilon^N}\frac{|u(x)-u(y)|^q}{|x-y|}dy\right)dx\nonumber
\\
\geq \alpha(N)A_1^q\int_{\Omega}\frac{\mathcal{L}^N\left(\Omega\cap B_{\epsilon}(x)\right)}{\mathcal{L}^N(B_\epsilon(x))}dx.
\end{align}
Taking the lower limit on both sides of $\eqref{eq:ineqality5}$ and using   dominated convergence theorem we get
$\underline{B}_{f,q}\left(\mathbb{R}^N\right)>0$. We prove now that $f\in BV^q(\R^N,\R^d)$. For any open and bounded set
$\Omega_1\subset\mathbb{R}^N$ such that $A:=\Supp(\zeta)\subset \Omega_1$ and every $0<\epsilon<\dist\left(A,\mathbb{R}^N\setminus \Omega_1\right)$
\begin{align}
\int_{\mathbb{R}^N}\left(\int_{B_{\epsilon}(x)}\frac{1}{\epsilon^N}\frac{|f(x)-f(y)|^q}{|x-y|}dy\right)dx=\int_{\Omega_1}\left(\int_{B_{\epsilon}(x)}\frac{1}{\epsilon^N}\frac{|f(x)-f(y)|^q}{|x-y|}dy\right)dx\nonumber
\\
+\int_{\mathbb{R}^N\setminus \Omega_1}\left(\int_{B_{\epsilon}(x)}\frac{1}{\epsilon^N}\frac{|f(x)-f(y)|^q}{|x-y|}dy\right)dx=\int_{\Omega_1}\left(\int_{B_{\epsilon}(x)}\frac{1}{\epsilon^N}\frac{|f(x)-f(y)|^q}{|x-y|}dy\right)dx.
\end{align}
Let $\Omega_2\subset\mathbb{R}^N$ be an open and bounded set such that $\overline{\Omega}_1\subset\Omega_2$. Since $f$ is locally H{\"o}lder continuous in $\mathbb{R}^N$ with exponent $1/q$, there exists a constant $M$ such that
$|f(x)-f(y)|\leq M|x-y|^{1/q}$ for every $x,y\in \Omega_2$. It follows for every
$0<\epsilon<\min\{\dist\left(A,\mathbb{R}^N\setminus \Omega_1\right),\dist\left(\overline{\Omega}_1,\mathbb{R}^N\setminus\Omega_2\right)\}$
\begin{equation}
\int_{\mathbb{R}^N}\left(\int_{B_{\epsilon}(x)}\frac{1}{\epsilon^N}\frac{|f(x)-f(y)|^q}{|x-y|}dy\right)dx\leq\alpha(N) M^q\mathcal{L}^N\left(\Omega_1\right),
\end{equation}
and so $\overline{B}_{f,q}(\mathbb{R}^N)<\infty$. Thus, since $f\in L^q\left(\mathbb{R}^N,\mathbb{R}^d\right)$, we get by Proposition
$\ref{prop:equivalence of finiteness of upper and Besov constants}$ that $f\in BV^q\left(\mathbb{R}^N,\mathbb{R}^d\right)$. Since $f$ is continuous, we have $\mathcal{J}'_f=\emptyset$, so the inequality
$\eqref{eq:ineqality3}$ is strict.

Note also that, since $\underline{B}_{f,q}\left(\mathbb{R}^N\right)>0$, we also have $\overline{B}_{f,q}\left(\mathbb{R}^N\right)>0$, and so there exists $1\leq i\leq d$ such that $\overline{B}_{f_i,q}\left(\mathbb{R}^N\right)>0$, where $f=(f_1,...,f_d)$, because
\begin{align}
\overline{B}_{f,q}(\mathbb{R}^N)=\limsup_{\epsilon\to 0^+}\int_{\mathbb{R}^N}\left(\int_{B_{\epsilon}(x)}\frac{1}{\epsilon^N}\frac{|f(x)-f(y)|^q}{|x-y|}dy\right)dx\nonumber
\\
\leq C\sum_{i=1}^d\limsup_{\epsilon\to 0^+}\int_{\mathbb{R}^N}\left(\int_{B_{\epsilon}(x)}\frac{1}{\epsilon^N}\frac{|f_i(x)-f_i(y)|^q}{|x-y|}dy\right)dx= C\sum_{i=1}^d\overline{B}_{f_i,q}(\mathbb{R}^N),
\end{align}
where $C$ is a constant dependent on $N,q$ only. Consequently, there exist scalar functions $u\in BV^q\left(\mathbb{R}^N,\mathbb{R}\right)$ for which the amount
$C_N\int_{\mathcal{J}'_u}|u^{j'}(x)|^qd\Haus^{N-1}(x)$ is strictly less than the upper infinitesimal Besov $q-$constant
$\overline{B}_{u,q}\left(\mathbb{R}^N\right)$.
\end{remark}

\section{Properties of fractional Sobolev functions}
In this section we analyse fractional Sobolev functions. In subsection \ref{subs:Fractional Sobolev functions} we prove a fine property of fractional Sobolev functions (Corollary \ref{thm:fine property of fractional Sobolev functions}). In subsection \ref{subs:Lusin approximation for fractional Sobolev functions} we prove Lusin approximation of fractional Sobolev functions by H{\"o}lder continuous functions (Corollary \ref{cor:Lusin approximation for fractional Sobolev functions}).

\subsection{Fine properties of fractional Sobolev functions}
\label{subs:Fractional Sobolev functions}

\begin{lemma}
\label{hhjgggg779988}
Let $\Omega\subset\R^N$ be an open set,  $u:\Omega\to \R^d$ is $\mathcal{L}^N$ measurable function and $q,p,r\in (0,\infty)$ such that $N-rqp\geq 0$.  Assume that for every open set $\Omega_0\subset\subset\Omega$ there exists $M\in(0,\infty)$ such that
\begin{equation}
\label{eq:inequality11}
\intop_{\Omega_0}\sup\limits_{\delta\in(0,M)}\Bigg\{\int_{B_1(0)}\chi_{\Omega_0}\big(x+\delta
z\big) \frac{\Big|u\big(x+\delta
z\big)-u(x)\Big|^q}{\delta^{rq}}dz\Bigg\}^pdx
<\infty\,.
\end{equation}
Then,
\begin{equation}
\label{eq:fine property}
\lim_{\rho\to
0^+}\frac{1}{\rho^{N}}\int_{B_{\rho}(x)}\Bigg\{\frac{1}{\rho^{N}}\int_{B_{\rho}(x)}
\Big|u\big(z\big)-u(y)\Big|^qdz\Bigg\}^pdy=0
\quad\quad\text{for}\;\;\mathcal{H}^{N-rqp}\,\,\text{a.e.}\,\,x\in
\Omega\,.
\end{equation}

\end{lemma}
\begin{proof}
By the local nature of the desired result \eqref{eq:fine property} we can assume that \eqref{eq:inequality11} is true for $\Omega$ (use the same considerations as in the proof of Theorem \ref{thm:sigma finiteness of limiting average with respect to Hausdorff measure}).
Then, for every $\rho\in(0,M)$ we have
\begin{multline}\label{gtgyihhzzkkhhjgjhjggjgiyyu22jokuiu1ookjkjkkkjhhpkokkliokkoijjjjhliojojjoiukjjkh1iuhhjjh35779988}
\frac{1}{\rho^{N}}\int_{B_\rho(x)\cap\Omega}\Bigg\{\frac{1}{\rho^{N}}\int_{B_\rho(v)\cap\Omega}
\Big|u\big(y\big)-u(v)\Big|^qdy\Bigg\}^pdv
\\
=\frac{1}{\rho^{N-rqp}}\int_{B_\rho(x)\cap\Omega}\Bigg\{\int_{B_1(0)}
\chi_\Omega\big(v+\rho z\big) \frac{\Big|u\big(v+\rho
z\big)-u(v)\Big|^q}{\rho^{rq}}dz\Bigg\}^pdv
\\
\leq \frac{1}{\rho^{N-rqp}}
\intop_{B_\rho(x)\cap\Omega}\sup\limits_{\delta\in(0,M)}\Bigg\{\int_{B_1(0)}
\chi_\Omega\big(v+\delta z\big) \frac{\Big|u\big(v+\delta
z\big)-u(v)\Big|^q}{\delta^{rq}}dz\Bigg\}^pdv \quad\quad\forall \,x\in
\Omega \,.
\end{multline}
In particular, for every $P\in(0,M)$ we have
\begin{multline}\label{gtgyihhzzkkhhjgjhjggjgiyyu22jokuiu1ookjkjkkkjhhpkokkliokkoijjjjhliojojjoiukjjkh1iuhhjjh35779988khhjjh}
\sup_{\rho\in(0,P]}\frac{1}{\rho^{N}}\int_{B_\rho(x)\cap\Omega}\Bigg\{\frac{1}{\rho^{N}}\int_{B_\rho(v)\cap\Omega}
\Big|u\big(y\big)-u(v)\Big|^qdy\Bigg\}^pdv\\
\leq \sup_{\rho\in(0,P]}\Bigg(\frac{1}{\rho^{N-rqp}}
\intop_{B_\rho(x)\cap\Omega}\sup\limits_{\delta\in(0,M)}\Bigg\{\int_{B_1(0)}
\chi_\Omega\big(v+\delta z\big) \frac{\Big|u\big(v+\delta
z\big)-u(v)\Big|^q}{\delta^{rq}}dz\Bigg\}^pdv\Bigg) \quad\quad\forall
\,x\in \Omega \,.
\end{multline}
Thus
\begin{multline}\label{gtgyihhzzkkhhjgjhjggjgiyyu22jokuiu1ookjkjkkkjhhpkokkliokkoijjjjhliojojjoiukjjkh1iuhhjjh35779988iggg}
\limsup_{\rho\to
0^+}\frac{1}{\rho^{N}}\int_{B_\rho(x)}\Bigg\{\frac{1}{\rho^{N}}\int_{B_\rho(v)}
\Big|u\big(y\big)-u(v)\Big|^qdy\Bigg\}^pdv
\\
\leq\limsup_{\rho\to 0^+}\Bigg(\frac{1}{\rho^{N-rqp}}
\intop_{B_\rho(x)}\sup\limits_{\delta\in(0,M)}\Bigg\{\int_{B_1(0)}
\chi_\Omega\big(v+\delta z\big) \frac{\Big|u\big(v+\delta
z\big)-u(v)\Big|^q}{\delta^{rq}}dz\Bigg\}^pdv\Bigg) \quad\quad\forall
\,x\in \Omega \,.
\end{multline}
In particular, using
\er{gtgyihhzzkkhhjgjhjggjgiyyu22jokuiu1ookjkjkkkjhhpkokkliokkoijjjjhliojojjoiukjjkh1iuhhjjh35779988iggg}
we deduce:
\begin{multline}
\label{eq:inequality10}
\limsup_{\rho\to
0^+}\,\frac{1}{\rho^{N}}\int_{B_{\rho}(x)}\Bigg\{\frac{1}{\rho^{N}}\int_{B_{\rho}(x)}
\Big|u\big(z\big)-u(y)\Big|^qdz\Bigg\}^pdy
\\
\leq\limsup_{\rho\to
0^+}\,\frac{1}{\rho^{N}}\int_{B_{\rho}(x)}\Bigg\{\frac{1}{\rho^{N}}\int_{B_{2\rho}(y)}
\Big|u\big(z\big)-u(y)\Big|^qdz\Bigg\}^pdy
\\
\leq 2^{N(1+p)}\,\limsup_{\rho\to
0^+}\,\frac{1}{(2\rho)^{N}}\int_{B_{2\rho}(x)}\Bigg\{\frac{1}{(2\rho)^{N}}\int_{B_{2\rho}(y)}
\Big|u\big(z\big)-u(y)\Big|^qdz\Bigg\}^pdy
\\
\leq 2^{N(1+p)}\,\limsup_{\rho\to 0^+}\Bigg(\frac{1}{\rho^{N-rqp}}
\intop_{B_\rho(x)}\sup\limits_{\delta\in(0,M)}\Bigg\{\int_{B_1(0)}
\chi_\Omega\big(v+\delta z\big) \frac{\Big|u\big(v+\delta
z\big)-u(v)\Big|^q}{\delta^{rq}}dz\Bigg\}^pdv\Bigg)
\\
\quad\quad\forall \,x\in \Omega\,.
\end{multline}

Let us denote for Lebesgue measurable sets $E\subset\Omega$
\begin{equation}
\mu(E):=\intop_E\sup\limits_{\delta\in(0,M)}\Bigg\{\int_{B_1(0)}\chi_\Omega\big(x+\delta
z\big) \frac{\Big|u\big(x+\delta
z\big)-u(x)\Big|^q}{\delta^{rq}}dz\Bigg\}^pdx.
\end{equation}
Rewriting inequality \eqref{eq:inequality10} in terms of upper density we have
\begin{equation}
\limsup_{\rho\to
0^+}\,\frac{1}{\rho^{N}}\int_{B_{\rho}(x)}\Bigg\{\frac{1}{\rho^{N}}\int_{B_{\rho}(x)}
\Big|u\big(z\big)-u(y)\Big|^qdz\Bigg\}^pdy\leq\alpha(N-rqp) 2^{N(1+p)}\Theta^*_{N-rqp}(\mu,x),\quad \forall x\in \Omega.
\end{equation}
The measure $\mu$ is a finite positive Radon measure in $\Omega$ which is absolutely continuous with respect to $\mathcal{L}^N=\mathcal{H}^N$.
Thus, by Lemma \ref{lem:Hausdorff measure of the set of points of positive upper density with lower dimension has measure zero}, there exists a Borel set $D\subset\Omega$ such that
\begin{equation}\label{gtgyihhzzkkhhjgjhjggjgiyyu22jokuiu1ookjkjkkkjhhpkokkliokkoijjjjhliojojjoiukjjkh1iuhhjjh357799jlljkjjk88}
\limsup_{\rho\to
0^+}\frac{1}{\rho^{N}}\int_{B_\rho(x)}\Bigg\{\frac{1}{\rho^{N}}\int_{B_\rho(x)}
\Big|u\big(y\big)-u(v)\Big|^qdy\Bigg\}^pdv=0 \quad\quad\forall \,x\in
D \,,
\end{equation}
and $\mathcal{H}^{N-rqp}(\Omega\setminus D)=0$.
\end{proof}

The following theorem is a generalization of a fine property of Sobolev functions
$W^{1,p}(\Omega)$ to Hausdorff measures:

\begin{theorem}
\label{thm:fine properties of Sobolev functions}
Let $\Omega\subset \R^N$ be an open set, $1< p\leq N$ and $u\in W^{1,p}(\Omega)$. Then,
\begin{multline}
\label{eq:fine property1}
\lim_{\rho\to
0^+}\fint_{B_{\rho}(x)}|u(y)-u_{B_\rho(x)}|^pdy
\\
\leq \lim_{\rho\to
0^+}\fint_{B_{\rho}(x)}\Bigg\{\fint_{B_{\rho}(x)}
\Big|u\big(z\big)-u(y)\Big|dz\Bigg\}^pdy=0
\quad\quad\text{for}\;\;\mathcal{H}^{N-p}\,\,\text{a.e.}\,\,x\in
\Omega\,.
\end{multline}
\end{theorem}

\begin{proof}
Let us define
\begin{equation}
f(x):=
\begin{cases}
\nabla u(x)\quad &x\in \Omega\\
0\quad &x\in \R^N\setminus\Omega
\end{cases},\quad f\in L^p(\R^N,\R^N).
\end{equation}
Let $\Omega_0\subset\subset \Omega$ be an open set.
Let $M>0$ be such that
$M<\dist(\overline{\Omega}_0,\R^N\setminus\Omega)$.

Assume for a moment
$u\in W^{1,p}(\Omega)\cap C^1(\Omega)$. Our goal in the following calculation is to show that we can apply Lemma \ref{hhjgggg779988}.
By the Hardy-Littlewood maximal inequality (Theorem \ref{thm:Hardy-Littlewood maximal inequality})
\begin{align}
\intop_{\Omega_0}\sup\limits_{\delta\in(0,M)}\Bigg\{\int_{B_1(0)}\chi_{\Omega_0}\big(x+\delta
z\big) \frac{\Big|u\big(x+\delta
z\big)-u(x)\Big|}{\delta}dz\Bigg\}^pdx\nonumber
\\
\leq \intop_{\Omega_0}\sup\limits_{\delta\in(0,M)}\Bigg\{\int_{B_1(0)}\frac{\Big|u\big(x+\delta
z\big)-u(x)\Big|}{|\delta z|}dz\Bigg\}^pdx\nonumber
\\
=\intop_{\Omega_0}\sup\limits_{\delta\in(0,M)}\Bigg\{\int_{B_1(0)}\frac{\left|\intop_0^1\frac{d}{ds}u(x+s\delta z)ds\right|}{|\delta z|}dz\Bigg\}^pdx\nonumber
\\
=\intop_{\Omega_0}\sup\limits_{\delta\in(0,M)}\Bigg\{\int_{B_1(0)}\frac{\left|\intop_0^1\nabla u(x+s\delta z)\cdot(\delta z)ds\right|}{|\delta z|}dz\Bigg\}^pdx\nonumber
\\
\leq \intop_{\Omega_0}\sup\limits_{\delta\in(0,M)}\Bigg\{\int_{B_1(0)}\left(\intop_0^1\left|\nabla u(x+s\delta z)\right|ds\right)dz\Bigg\}^pdx\nonumber
\\
=\intop_{\Omega_0}\sup\limits_{\delta\in(0,M)}\Bigg\{\intop_0^1\left(\int_{B_1(0)}\left|\nabla u(x+s\delta z)\right|dz\right)ds\Bigg\}^pdx\nonumber
\\
=\intop_{\Omega_0}\sup\limits_{\delta\in(0,M)}\Bigg\{\intop_0^1\left(\frac{1}{(s\delta )^N}\int_{B_{s\delta}(x)}\left|\nabla u(y)\right|dy\right)ds\Bigg\}^pdx\nonumber
\\
\leq \intop_{\Omega_0}\Bigg\{\sup\limits_{\delta\in(0,M)}\sup_{s\in (0,1)}\left(\frac{1}{(s\delta )^N}\int_{B_{s\delta }(x)}\left|\nabla u(y)\right|dy\right)\Bigg\}^pdx\nonumber
\\
\leq \int_{\R^N}\left\{\sup\limits_{r>0}\left(\frac{1}{r^N}\int_{B_{r}(x)}\left|f(y)\right|dy\right)\right\}^pdx\nonumber
\\
=\alpha(N)^p\|M(f)\|^p_{L^p(\R^N)}\leq C \|f\|^p_{L^p(\R^N)}=C\|\nabla u\|^p_{L^p(\Omega)},
\end{align}
where $C$ is a constant dependent of $q,N$ only.
Here,
\begin{align}
M(f)(x):=\sup\limits_{r>0}\left(\fint_{B_{r}(x)}\left|f(y)\right|dy\right)
\end{align}
is the Hardy-littlewood maximal function.

Let $\{u_n\}_{n\in \N}\subset W^{1,p}(\Omega)\cap C^1(\Omega)$ be such that
\begin{equation}
\lim_{n\to \infty}\|\nabla u_n\|^p_{L^p(\Omega)}=\|\nabla u\|^p_{L^p(\Omega)},
\end{equation}
and
\begin{equation}
\lim_{n\to \infty}u_n(x)=u(x)\quad \mathcal{L}^N\quad a.e. \quad x\in \Omega.
\end{equation}

Thus, for every $n\in \N$
\begin{align}
\label{eq:inequality14}
\intop_{\Omega_0}\sup\limits_{\delta\in(0,M)}\Bigg\{\int_{B_1(0)}\chi_{\Omega_0}\big(x+\delta
z\big) \frac{\Big|u_n\big(x+\delta
z\big)-u_n(x)\Big|}{\delta}dz\Bigg\}^pdx
\leq C\|\nabla u_n\|^p_{L^p(\Omega)}.
\end{align}

Taking the limit on both sides of \eqref{eq:inequality14} and using Fatou's Lemma we get
\begin{align}
\intop_{\Omega_0}\sup\limits_{\delta\in(0,M)}\Bigg\{\int_{B_1(0)}\chi_{\Omega_0}\big(x+\delta
z\big) \frac{\Big|u\big(x+\delta
z\big)-u(x)\Big|}{\delta}dz\Bigg\}^pdx
\leq C\|\nabla u\|^p_{L^p(\Omega)}<\infty.
\end{align}
Now, \eqref{eq:fine property1} follows by Lemma \ref{hhjgggg779988} choosing $q=r=1$.
\end{proof}

\begin{lemma}
\label{lem:lemma about fractional Sobolev functions}
Let $\Omega\subset\R^N$ be an open set, $q\in[1,\infty)$, $r\in(0,1)$ and $u\in W^{r,q}_{loc}(\Omega,\R^d)$.
Then
\begin{equation}
\label{eq:finiteness of a term in which sup is inside the integral}
\int_{\Omega_0}\Bigg\{\sup\limits_{\delta\in(0,\infty)}\int_{B_1(0)}\chi_{\Omega_0}\big(x+\delta
z\big) \frac{\Big|u\big(x+\delta
z\big)-u(x)\Big|^q}{\delta^{rq}}dz\Bigg\}dx<\infty\,,
\end{equation}
for every open $\Omega_0\subset\subset\Omega$.
\end{lemma}

\begin{proof}
Let $\Omega_0\subset\subset\Omega$ be open.
Since $u\in W^{r,q}(\Omega_0,\R^d)$, then
\begin{equation}\label{gtgyihhzzkkhhjgjhjggjgiyyu22jokuiu1ookjkjkkkjhhpkokkliokkoijjjjhliojojj1hihih35779988}
\int_{\Omega_0}\Bigg(\int_{\R^N} \chi_{\Omega_0}(x+z)
\frac{\big|u(x+z)-u(x)\big|^q}{|z|^{N+rq}}dz\Bigg)dx=\int_{\Omega_0}\Bigg(\int_{\Omega_0}
\frac{\big|u(y)-u(x)\big|^q}{|y-x|^{N+rq}}dy\Bigg)dx<\infty\,.
\end{equation}

We have
\begin{multline}\label{gtgyihhzzkkhhjgjhjggjgiyyu22jokuiu1ookjkjkkkjhhpkokkliokkoijjjjhliojojj1hihihhhjhhyuyuu35779988}
\int_{\Omega_0}\Bigg(\int_{\R^N} \chi_{\Omega_0}(x+z)
\frac{\big|u(x+z)-u(x)\big|^q}{|z|^{N+rq}}dz\Bigg)dx
\\
\geq\int_{\Omega_0}\Bigg(\sup\limits_{\e\in (0,\infty)}\int_{B_\e(0)}
\chi_{\Omega_0}(x+z)
\frac{\big|u(x+z)-u(x)\big|^q}{|z|^{N+rq}}dz\Bigg)dx
\\
=\int_{\Omega_0}\Bigg(\sup\limits_{\e\in (0,\infty)}\int_{B_1(0)}
\chi_{\Omega_0}(x+\e y) \frac{\big|u(x+\e
y)-u(x)\big|^q}{|y|^{N+{rq}}\e^{rq}}dy\Bigg)dx\\
\geq\int_{\Omega_0}\Bigg(\sup\limits_{\e\in (0,\infty)}\int_{B_1(0)}
\chi_{\Omega_0}(x+\e y) \frac{\big|u(x+\e
y)-u(x)\big|^q}{\e^{rq}}dy\Bigg)dx\,,
\end{multline}
and \eqref{eq:finiteness of a term in which sup is inside the integral} follows.
\end{proof}

\begin{corollary}
\label{thm:fine property of fractional Sobolev functions}
Let $\Omega\subset\R^N$ be an open set, $q\in [1,\infty)$, $r\in(0,1)$ are
such that $rq\leq N$ and $u\in W^{r,q}_{loc}(\Omega,\R^d)$.
Then,
\begin{equation}\label{gtgyihhzzkkhhjgjhjggjgiyyu22jokuiu1ookjkjkkkjhhpkokkliokkoijjjjhliojojjoiukjjkhioiou1jkjkjkhkhjhjhighooioiuiiuuu35i9i9iuyuiiiihhj779988}
\lim_{\rho\to
0^+}\frac{1}{\rho^{N}}\int_{B_{\rho}(x)}\Bigg\{\frac{1}{\rho^{N}}\int_{B_{\rho}(x)}
\Big|u\big(z\big)-u(y)\Big|^qdz\Bigg\}dy=0
\quad\quad\text{for}\quad\mathcal{H}^{N-rq}\quad \text{a.e.}\quad
x\in\Omega\,.
\end{equation}
\end{corollary}

\begin{proof}
By Lemma \ref{lem:lemma about fractional Sobolev functions} we have
\begin{equation}\label{gtgyihhzzkkhhjgjhjggjgiyyu22jokuiu1ookjkjkkkjhhpkokkliokkoijjjjhliojojj1hihihhhjhhyuyuu35ljjjuututtttiojiijkhjkhioiuiuiyuuuy779988}
\int_{\Omega_0}\Bigg\{\sup\limits_{\delta\in(0,1)}\int_{B_1(0)}\chi_{\Omega_0}\big(x+\delta
z\big) \frac{\Big|u\big(x+\delta
z\big)-u(x)\Big|^q}{\delta^{rq}}dz\Bigg\}dx<\infty\,,
\end{equation}
for every open $\Omega_0\subset\subset\Omega$.
Thus, by Lemma \ref{hhjgggg779988} with $p=1$ we deduce the desired result.
\end{proof}

\subsection{Lusin approximation for fractional Sobolev functions by H{\"o}lder continuous functions}
\label{subs:Lusin approximation for fractional Sobolev functions}
We begin by representing a set of functions $\mathcal{A}^{r,q}$ which is bigger than the fractional Sobolev space $W^{r,q}$. We will show Lusin approximation for these functions.
The main result of this subsection is Corollary \ref{cor:Lusin approximation for fractional Sobolev functions}.

\begin{definition}
Let $\Omega\subset\R^N$ be an open set and $r,q\in(0,\infty)$. We say that $u\in \mathcal{A}^{r,q}(\Omega,\R^d)$ if and only if
$u\in L^q(\Omega,\R^d)$ and
\begin{equation}
\int_{\Omega}\left(\limsup_{\delta\to 0^+}\int_{B_\delta(x)}\frac{|u(y)-u(x)|^q}{\delta^{N+rq}}dy\right)dx<\infty.
\end{equation}

We say that $u\in \mathcal{A}^{r,q}_{loc}(\Omega,\R^d)$ if and only if $u\in L_{loc}^q(\Omega,\R^d)$ and $u\in  \mathcal{A}^{r,q}(\Omega_0,\R^d)$ for every open $\Omega_0\subset\subset\Omega$.
\end{definition}

We now give several propositions which tell us about some families of functions that are included in $\mathcal{A}^{r,q}$.

\begin{proposition}
Let $\Omega\subset\R^N$ be an open set, $q\in (0,\infty)$ and let $0<r\leq \alpha\leq 1$. Then
\begin{equation}
C^{0,\alpha}_{c}(\Omega,\R^d)\subset \mathcal{A}^{r,q}(\Omega,\R^d),
\end{equation}
where $C^{0,\alpha}_{c}(\Omega,\R^d)$ is the space of $\R^d$-valued $\alpha$-H{\"o}lder continuous functions with compact support in $\Omega$.
\end{proposition}

\begin{proof}
Let us denote by $K$ the compact support of a function $u\in C^{0,\alpha}_{c}(\Omega,\R^d)$. It follows that
\begin{align}
\int_{\Omega}\left(\limsup_{\delta\to 0^+}\int_{B_\delta(x)}\frac{|u(y)-u(x)|^q}{\delta^{N+rq}}dy\right)dx=\int_{K}\left(\limsup_{\delta\to 0^+}\int_{B_\delta(x)}\frac{|u(y)-u(x)|^q}{\delta^{N+rq}}dy\right)dx\nonumber
\\
+\int_{\Omega\setminus K}\left(\limsup_{\delta\to 0^+}\int_{B_\delta(x)}\frac{|u(y)-u(x)|^q}{\delta^{N+rq}}dy\right)dx.
\end{align}
Since $K$ is the support of $u$, then
\begin{equation}
\int_{\Omega\setminus K}\left(\limsup_{\delta\to 0^+}\int_{B_\delta(x)}\frac{|u(y)-u(x)|^q}{\delta^{N+rq}}dy\right)dx=0.
\end{equation}

Since $u\in C^{0,\alpha}_{c}(\Omega,\R^d)$
\begin{align}
\int_{K}\left(\limsup_{\delta\to 0^+}\int_{B_\delta(x)}\frac{|u(y)-u(x)|^q}{\delta^{N+rq}}dy\right)dx
\leq (C(u))^q\int_{K}\left(\limsup_{\delta\to 0^+}\int_{B_\delta(x)}\frac{|y-x|^{\alpha q}}{\delta^{N+rq}}dy\right)dx\nonumber
\\
\leq (C(u))^q \alpha(N)\mathcal{L}^N(K)\limsup_{\delta\to 0^+}\delta^{q(\alpha-r)},
\end{align}
where
\begin{equation}
C(u):=\sup_{x\neq y}\frac{|u(x)-u(y)|}{|x-y|^{\alpha}}.
\end{equation}
We see that since $\alpha\geq r$, then $u\in \mathcal{A}^{r,q}(\Omega,\R^d)$.
\end{proof}

\begin{proposition}
Let $\Omega\subset\R^N$ be an open set, and let $q\in [1,\infty),r\in (0,\infty)$ such that $rq\leq1$. Then
\begin{equation}
BV(\Omega,\R^d)\cap L^\infty(\Omega,\R^d)\subset \mathcal{A}^{r,q}(\Omega,\R^d).
\end{equation}
In case $q=1$ we have for every $r\in (0,1]$
\begin{equation}
BV(\Omega,\R^d)\subset \mathcal{A}^{r,1}(\Omega,\R^d).
\end{equation}
\end{proposition}

\begin{proof}
Recall Definitions \ref{def:approximate limit}, \ref{def:approximate differentiability points}.
For $u\in BV(\Omega,\R^d)$, by Theorem \ref{thm:Calderon-Zygmund} (Calder{\'o}n-Zygmund theorem)
$\Leb^N\left(\Omega\setminus \mathcal{D}_u\right)=0$,  where $\mathcal{D}_u\subset\Omega$ is the set of approximate differentiability points of $u$.

Assume that $u\in BV(\Omega,\R^d)\cap L^\infty(\Omega,\R^d)$. Let us denote $\|u\|_{\infty}:=\|u\|_{L^\infty(\Omega,\R^d)}$. For every $x\in \mathcal{D}_u$
\begin{align}
\label{eq:calculation1}
\int_{B_\delta(x)}\frac{|u(y)-\tilde{u}(x)|^q}{\delta^{N+rq}}dy
\leq (2\|u\|_{\infty})^{q-1}\int_{B_\delta(x)}\frac{|u(y)-\tilde{u}(x)|}{\delta^{N+rq}}dy\nonumber
\\
\leq (2\|u\|_{\infty})^{q-1}\int_{B_\delta(x)}\frac{|u(y)-\tilde{u}(x)-\nabla u(x)(y-x)|}{\delta^{N+rq}}dy
+(2\|u\|_{\infty})^{q-1}\int_{B_\delta(x)}\frac{|\nabla u(x)(y-x)|}{\delta^{N+rq}}dy\nonumber
\\
\leq\frac{(2\|u\|_{\infty})^{q-1}}{\delta^{rq-1}}\left(\frac{1}{\delta^{N+1}}\int_{B_\delta(x)}|u(y)-\tilde{u}(x)-\nabla u(x)(y-x)|dy\right)
+\frac{(2\|u\|_{\infty})^{q-1}}{\delta^{rq-1}}\alpha(N)|\nabla u(x)|.
\end{align}
Thus,
\begin{equation}
\limsup_{\delta\to 0^+}\int_{B_\delta(x)}\frac{|u(y)-\tilde{u}(x)|^q}{\delta^{N+rq}}dy\leq (2\|u\|_{\infty})^{q-1}\alpha(N)|\nabla u(x)|.
\end{equation}
Therefore,
\begin{equation}
\int_{\Omega}\left(\limsup_{\delta\to 0^+}\int_{B_\delta(x)}\frac{|u(y)-u(x)|^q}{\delta^{N+rq}}dy\right)dx\leq (2\|u\|_{\infty})^{q-1}\alpha(N)\int_{\Omega}|\nabla u(x)|dx<\infty.
\end{equation}
In case $q=1$ we get for $u\in BV(\Omega,\R^d)$ by a similar calculation as above
\begin{equation}
\int_{\Omega}\left(\limsup_{\delta\to 0^+}\int_{B_\delta(x)}\frac{|u(y)-u(x)|}{\delta^{N+r}}dy\right)dx\leq \alpha(N)\int_{\Omega}|\nabla u(x)|dx<\infty,
\end{equation}
for every $r\in (0,1]$.
\end{proof}

\begin{remark}
Recall the notion of $L^q-$approximate differentiability:
\\
Let $\Omega\subset\mathbb R^N$ be an open set, $0<q<\infty$ and $u\in L^q_{loc}(\Omega,\R^d)$.
We say that $u$ is {\it $L^q$-approximately differentiable} at $x\in \Omega$ if there exists a linear mapping
$L_x:\mathbb R^N\to \mathbb R^d$ such that
\begin{equation}
\lim_{\delta\to 0^+}\fint_{B_\delta(x)}\frac{|u(z)-\tilde{u}(x)-L_x(z-x)|^q}{\delta^q}dz=0.
\end{equation}
For $q\in (0,\infty)$ and  $r\in (0,1]$, if a function $u\in L^q_{loc}(\Omega,\R^d)$ is $L^q-$approximately differentiable
$\mathcal{L}^N$ almost everywhere in $\Omega$ and the approximate differential, $L$, satisfies $|L|\in L^q(\Omega)$, then we have for every approximate differentiability point $x\in \Omega$ and sufficiently small $\delta>0$

\begin{align}
\label{eq:calculation2}
&\int_{B_\delta(x)}\frac{|u(y)-\tilde{u}(x)|^q}{\delta^{N+rq}}dy\nonumber
\\
&\leq C\int_{B_\delta(x)}\frac{|u(y)-\tilde{u}(x)-L_x(y-x)|^q}{\delta^{N+rq}}dy
+C\int_{B_\delta(x)}\frac{|L_x(y-x)|^q}{\delta^{N+rq}}dy\nonumber
\\
&\leq\frac{C}{\delta^{q(r-1)}}\left(\frac{1}{\delta^{N+q}}\int_{B_\delta(x)}|u(y)-\tilde{u}(x)-L_x(y-x)|^qdy\right)
+\frac{C}{\delta^{q(r-1)}}\alpha(N)|L_x|^q,
\end{align}
where $C=1$ if $0<q\leq 1$, and $C=2^{q-1}$ if $1\leq q<\infty$.
Therefore,
\begin{equation}
\int_{\Omega}\left(\limsup_{\delta\to 0^+}\int_{B_\delta(x)}\frac{|u(y)-u(x)|^q}{\delta^{N+rq}}dy\right)dx\leq C\alpha(N)\int_{\Omega}|L_x|^qdx<\infty.
\end{equation}
Thus, $u\in \mathcal{A}^{r,q}(\Omega,\R^d)$. Notice that, by C{\'a}ldron-Zygmund theorem, functions of bounded variation are $L^1-$approximately differentiable $\mathcal{L}^N$ almost everywhere and the approximate differential is an $L^1$ mapping.
\end{remark}

\begin{proposition}
\label{prop:W(r,q) lies inside A(r,q)}
Let $\Omega\subset\R^N$ be an open set, and let $q\in [1,\infty),r\in (0,1)$. Then
\begin{equation}
W^{r,q}_{loc}(\Omega,\R^d)\subset \mathcal{A}_{loc}^{r,q}(\Omega,\R^d).
\end{equation}
\end{proposition}

\begin{proof}
Let $u\in W^{r,q}_{loc}(\Omega,\R^d)$. By Lemma \ref{lem:lemma about fractional Sobolev functions}, for an open set $\Omega_0\subset\subset\Omega$
\begin{align}
\infty>\int_{\Omega_0}\Bigg\{\sup\limits_{\delta\in(0,\infty)}\int_{B_1(0)}\chi_{\Omega_0}\big(x+\delta
z\big) \frac{\Big|u\big(x+\delta
z\big)-u(x)\Big|^q}{\delta^{rq}}dz\Bigg\}dx\nonumber
\\
\geq \int_{\Omega_0}\Bigg\{\limsup\limits_{\delta\to 0^+}\int_{B_1(0)}\chi_{\Omega_0}\big(x+\delta
z\big) \frac{\Big|u\big(x+\delta
z\big)-u(x)\Big|^q}{\delta^{rq}}dz\Bigg\}dx\nonumber
\\
=\int_{\Omega_0}\Bigg\{\limsup\limits_{\delta\to 0^+}\int_{B_\delta(x)} \frac{|u(y)-u(x)|^q}{\delta^{N+rq}}dy\Bigg\}dx.
\end{align}
Thus, $u\in \mathcal{A}_{loc}^{r,q}(\Omega,\R^d)$.

\end{proof}

\begin{theorem}
\label{thm:Lusin approximation for A(r,q) functions}
Let $\Omega\subset \R^N$ be an open set and $q\in [1,\infty) $, $0<r\leq 1$. Let $u\in L^q_{loc}(\Omega,\R^d)$ be such that for $\mathcal{L}^N$-almost every point $x\in \Omega$ we have
\begin{equation}
\label{eq:assumption about finiteness of limsup}
\limsup_{\delta\to 0^+}\int_{B_\delta(x)}\frac{|u(y)-u_{B_\delta(x)}|^q}{\delta^{N+rq}}dy<\infty, \quad u_{B_\delta(x)}:=\fint_{B_\delta(x)}u(y)dy.
\end{equation}
Let $K\subset \Omega$ be a compact set. Then for every $\e>0$ there exists a compact set
$B\subset K$ such that
$\Leb^N\left(K\setminus B\right)<\e$ and
$u\in C^{0,r}(B,\R^d)$.
\end{theorem}

\begin{proof}
Let $\Omega_0\subset \R^N$ be an open set such that $\Omega_0\subset \subset \Omega$ and $K\subset \Omega_0$.
\\
\textbf{Step 1}
\\
For $\alpha>0$ we define
\begin{equation}
A_\alpha:=\left\{x\in \Omega_0:\limsup_{\delta\to 0^+}\int_{B_\delta(x)}\frac{|u(y)-u_{B_\delta(x)}|^q}{\delta^{N+rq}}dy\leq \alpha\right\}.
\end{equation}

For each $n\in \N$ we define
\begin{equation}
B_n:=\left\{x\in \Omega_0:\int_{B_\delta(x)}\frac{|u(y)-u_{B_\delta(x)}|^q}{\delta^{N+rq}}dy\leq n\quad \forall 0<\delta<I(n):=\min\left\{\frac{1}{n},\dist(\R^N\setminus\Omega,\overline{\Omega}_0)\right\}\right\}.
\end{equation}
The sequence $\{B_n\}_{n=1}^\infty$ has the following four properties:
\\
1. $B_n\subset B_{n+1},\forall n\in \N$;
\\
2. $B_n$ is closed for every $n\in \N$;
\\
3. $\Leb^N\left(\Omega_0\setminus \bigcup_{n=1}^\infty B_n\right)=0$;
\\
4. $\tilde{u}\in C_{loc}^{0,r}(B_n\setminus \mathcal{S}_u,\R^d)$ for every $n\in \N$.
\\
\\
In item 4 above, $\tilde{u}$ is the approximate limit as defined in Definition \ref{def:approximate limit}. We mean by $\tilde{u}\in C_{loc}^{0,r}(B_n\setminus \mathcal{S}_u,\R^d)$, that for every compact set $K\subset B_n\setminus \mathcal{S}_u$ there is a constant $C$ such that $|\tilde{u}(x)-\tilde{u}(y)|\leq C|x-y|^r$ for every $x,y\in K$.
\\
The monotonicity property $B_n\subset B_{n+1}$ follows directly from the definition of $B_n$. The closedness property of $B_n$ is achieved by the continuity of the function $x\longmapsto \int_{B_\delta(x)}\frac{|u(y)-u_{B_\delta(x)}|^q}{\delta^{N+rq}}dy$, $x\in \Omega_0$, for every choice of
$0<\delta<\dist(\R^N\setminus\Omega,\overline{\Omega}_0)$. For the third property, notice that for each
$\alpha>0$ we have $A_\alpha\subset \bigcup_{n=1}^\infty B_n$, and by assumption \eqref{eq:assumption about finiteness of limsup} we get that $\mathcal{L}^N$-almost every $x\in \Omega_0$ lies in the union of the $A_\alpha$. Therefore,
\begin{align}
\Leb^N\left(\Omega_0\setminus \bigcup_{n=1}^\infty B_n\right)\leq \Leb^N\left(\Omega_0\setminus \bigcup_{\alpha>0} A_\alpha\right)=0.
\end{align}

Let us prove the fourth property. Let us fix $n_0\in \N$.
For $x\in B_{n_0}$ we get by Jensen inequality for every $0<\delta<I(n_0)$
\begin{equation}
\left(\fint_{B_\delta(x)}|u(y)-u_{B_\delta(x)}|dy\right)^q\leq \fint_{B_\delta(x)}|u(y)-u_{B_\delta(x)}|^qdy\leq \frac{n_0}{\alpha(N)}\delta^{rq}.
\end{equation}
Hence
\begin{equation}
\fint_{B_\delta(x)}|u(y)-u_{B_\delta(x)}|dy\leq \left(\frac{n_0}{\alpha(N)}\right)^{1/q}\delta^{r},\quad \forall 0<\delta<I(n_0).
\end{equation}

In what following we do estimates in order to get bounds in terms of
$\delta^{r}$
on the terms
\begin{equation}
|\tilde{u}(x)-u_{B_\delta(x)}|,|u_{B_\delta(x)}-u_{B_\delta(y)}|,
\end{equation}
where $x,y\in B_{n_0}\setminus \mathcal{S}_u$ and $0<\delta<I(n_0)$.

\begin{align}
|u_{B_{\delta/2^{k+1}}(x)}-u_{B_{\delta/2^{k}}(x)}|
&\leq \fint_{B_{\delta/2^{k+1}}(x)}|u(y)-u_{B_{\delta/2^{k}}(x)}|dy\nonumber
\\
&=\frac{1}{\Leb^N(B_{\delta/2^{k+1}}(x))}\int_{B_{\delta/2^{k+1}}(x)}|u(y)-u_{B_{\delta/2^{k}}(x)}|dy\nonumber
\\
&=\frac{2^N}{\Leb^N(B_{\delta/2^k}(x))}\int_{B_{\delta/2^{k+1}}(x)}|u(y)-u_{B_{\delta/2^{k}}(x)}|dy\nonumber
\\
&\leq 2^N\fint_{B_{\delta/2^{k}}(x)}|u(y)-u_{B_{\delta/2^{k}}(x)}|dy
\leq 2^N\left(\frac{n_0}{\alpha(N)}\right)^{1/q}\left(\frac{\delta}{2^k}\right)^{r}.
\end{align}

Therefore, for every $x\in B_{n_0}\setminus \mathcal{S}_u$ and $0<\delta<I(n_0)$, since $\lim_{\delta\to 0^+}u_{B_{\delta}(x)}=\tilde{u}(x)$ it follows that

\begin{align}
|\tilde{u}(x)-u_{B_{\delta}(x)}|&=\left|\sum_{k=0}^\infty \left[u_{B_{\delta/2^{k+1}}(x)}-u_{B_{\delta/2^{k}}(x)}\right]\right|\nonumber
\\
&\leq \sum_{k=0}^\infty \left|u_{B_{\delta/2^{k+1}}(x)}-u_{B_{\delta/2^{k}}(x)}\right|\nonumber
\\
&\leq \left(2^N\left(\frac{n_0}{\alpha(N)}\right)^{1/q}\sum_{k=0}^\infty\frac{1}{2^{rk}}\right)\delta^{r}.
\end{align}

Let $x,y\in B_{n_0}\setminus \mathcal{S}_u$ such that $x\neq y$ and $\delta=|x-y|<I(n_0)$. For every $z\in B_\delta(x)\cap B_\delta(y)\setminus \mathcal{S}_u$
we have by the triangle inequality
\begin{equation}
|u_{B_{\delta}(x)}-u_{B_{\delta}(y)}|\leq
|u_{B_{\delta}(x)}-\tilde{u}(z)|+|\tilde{u}(z)-u_{B_{\delta}(y)}|
\end{equation}
and therefore
\begin{align}
|u_{B_{\delta}(x)}-u_{B_{\delta}(y)}|&
\leq \fint_{B_\delta(x)\cap B_\delta(y)} \left(|u_{B_{\delta}(x)}-u(z)|+|u(z)-u_{B_{\delta}(y)}|\right)dz\nonumber
\\
&\leq 2^N\left(\fint_{B_{\delta}(x)}|u_{B_{\delta}(x)}-u(z)|dz+
\fint_{B_{\delta}(y)}|u(z)-u_{B_{\delta}(y)}|dz \right)\nonumber
\\
&\leq 2^{N+1}\left(\frac{n_0}{\alpha(N)}\right)^{1/q}\delta^{r}.
\end{align}
In the second inequality in the above calculation we used
\begin{equation}
B_{\delta/2}\left(\frac{x+y}{2}\right)\subset B_\delta(x)\cap B_\delta(y)\quad \Longrightarrow \quad \frac{1}{\Leb^N(B_\delta(x)\cap B_\delta(y))}\leq \frac{2^N}{\Leb^N(B_\delta(x))}=\frac{2^N}{\Leb^N(B_\delta(y))}.
\end{equation}
By the above results and the triangle inequality we obtain for $x,y\in B_{n_0}\setminus \mathcal{S}_u$ such that $x\neq y$ and
$\delta=|x-y|<I(n_0)$

\begin{align}
|\tilde{u}(x)-\tilde{u}(y)|&\leq
|\tilde{u}(x)-u_{B_{\delta}(x)}|+|u_{B_{\delta}(x)}-u_{B_{\delta}(y)}|+|u_{B_{\delta}(y)}-\tilde{u}(y)|\leq C|x-y|^{r},
\end{align}
where $C$ is a constant dependent of $N,n_0,r,q$ only.
It proves that $\tilde{u}\in C_{loc}^{0,r}(B_n\setminus \mathcal{S}_u,\R^d)$.
\\

\textbf{Step 2}
\\
Let $\e>0$. By the previous step (property 3) there exists $n_0\in \N$ such that $\Leb^N\left(K\setminus B_{n_0}\right)<\e$. Assume for a moment that $u$ is continuous in the compact set $K$. Let us define $B:=B_{n_0}\cap K$. $B$ is compact as an intersection of the compact set
$K$ with the closed set $B_{n_0}$. Then, since $u$ is continuous in $K$ we get $K\subset\Omega_0\setminus \mathcal{S}_u$, $u(x)=\tilde{u}(x)$ for $x\in K$, and by the previous step
$u\in C^{0,r}_{loc}(B,\R^d)$. Since $B$ is compact we have $u\in C^{0,r}(B,\R^d)$. In addition, by the choice of $B_{n_0}$ we have
$\Leb^N\left(K\setminus B\right)=\Leb^N\left(K\setminus B_{n_0}\right)<\e$.

\textbf{Step 3}
\\
If $u$ is not continuous in $K$, then by Lusin theorem we can choose a compact set $K_0\subset K$ such that $u$ is continuous in
$K_0$ and
$\Leb^N\left(K\setminus K_0\right)<\frac{\e}{2}$.
By the previous step, there exists a compact set $B\subset K_0$ such that $u\in C^{0,r}(B,\R^d)$ and $\Leb^N\left(K_0\setminus B\right)<\frac{\e}{2}$. We get
\begin{equation}
\Leb^N\left(K\setminus B\right)=\Leb^N\left(K\setminus K_0\right)+\Leb^N\left(K_0\setminus B\right)<\e.
\end{equation}

It completes the proof.
\end{proof}

\begin{corollary}
\label{thm:Lusin approximation for A(r,q) functions, corollary}
Let $\Omega\subset \R^N$ be an open set and $q\in [1,\infty) $, $0<r\leq 1$. Let $u\in \mathcal{A}_{loc}^{r,q}(\Omega,\R^d)$.
Let $K\subset \Omega$ be a compact set. Then for every $\e>0$ there exists a compact set
$B\subset K$ such that
$\Leb^N\left(K\setminus B\right)<\e$ and
$u\in C^{0,r}(B,\R^d)$.
\end{corollary}

\begin{proof}
Let $\Omega_0\subset \R^N$ be an open set such that $\Omega_0\subset \subset \Omega$ and $K\subset \Omega_0$.
By Jensen inequality and the convexity of $s\longmapsto s^q,s>0$, we have for
$x,y\in\Omega_0\setminus \mathcal{S}_u$ and sufficiently small $\delta>0$
\begin{multline}
|\tilde{u}(y)-u_{B_\delta(x)}|^q\leq \fint_{B_\delta(x)}|\tilde{u}(y)-u(z)|^qdz
\\
\leq 2^{q-1}\fint_{B_\delta(x)}|u(z)-\tilde{u}(x)|^qdz+2^{q-1}|\tilde{u}(y)-\tilde{u}(x)|^q.
\end{multline}
By taking the average of both sides of the last inequality on a ball $B_\delta(x)$ with respect to $dy$ we obtain for sufficiently small $\delta>0$
\begin{align}
\fint_{B_\delta(x)}|u(y)-u_{B_\delta(x)}|^qdy\leq 2^q\fint_{B_\delta(x)}|u(y)-\tilde{u}(x)|^qdy.
\end{align}
Multiplying both sides by $\frac{\alpha(N)}{\delta^{rq}}$ we get
\begin{equation}
\int_{B_\delta(x)}\frac{|u(y)-u_{B_\delta(x)}|^q}{\delta^{N+rq}}dy\leq
2^q\int_{B_\delta(x)}\frac{|u(y)-\tilde{u}(x)|^q}{\delta^{N+rq}}dy.
\end{equation}
Therefore, since $u\in \mathcal{A}^{r,q}(\Omega_0,\R^d)$, then we get
\begin{equation}
\int_{\Omega_0}\left(\limsup_{\delta\to 0^+}\int_{B_\delta(x)}\frac{|u(y)-u_{B_\delta(x)}|^q}{\delta^{N+rq}}dy\right)dx\leq
2^q\int_{\Omega_0}\left(\limsup_{\delta\to 0^+}\int_{B_\delta(x)}\frac{|u(y)-u(x)|^q}{\delta^{N+rq}}dy\right)dx<\infty.
\end{equation}
Thus, $\limsup_{\delta\to 0^+}\int_{B_\delta(x)}\frac{|u(y)-u_{B_\delta(x)}|^q}{\delta^{N+rq}}dy<\infty$ for $\mathcal{L}^N$-almost everywhere in $\Omega$, and the conditions of Theorem \ref{thm:Lusin approximation for A(r,q) functions} hold.
\end{proof}

\begin{corollary}
\label{cor:Lusin approximation for fractional Sobolev functions}
Let $\Omega\subset \R^N$ be an open set and $q\in [1,\infty)$, $0<r<1$. Let $u\in W^{r,q}_{loc}(\Omega,\R^d)$
and let $K\subset \Omega$ be a compact set. Then for every $\e>0$ there exists a compact set
$B\subset K$ such that
$\Leb^N\left(K\setminus B\right)<\e$ and
$u\in C^{0,r}(B,\R^d)$.
\end{corollary}

\begin{proof}
The proof follows by Proposition \ref{prop:W(r,q) lies inside A(r,q)} and Corollary \ref{thm:Lusin approximation for A(r,q) functions, corollary}.
\end{proof}

\begin{remark}
One can derive from Corollary \ref{cor:Lusin approximation for fractional Sobolev functions} that:
Let $\Omega\subset \R^N$ be an open set and $q\in [1,\infty)$, $0<r<1$. Let $u\in W^{r,q}_{loc}(\Omega,\R^d)$
and let $K\subset \Omega$ be a compact set. Then for every $\e>0$ there exists an H{\"o}lder function $f\in C^{0,r}(\R^N,\R^d)$ such that $\Leb^N\left(\left\{x\in K: f(x)\neq u(x)\right\}\right)<\e$.

Indeed, by Corollary \ref{cor:Lusin approximation for fractional Sobolev functions} we get for arbitrary $\e>0$ a set $B\subset K$ such that $\Leb^N\left(K\setminus B\right)<\e$ and
$u\in C^{0,r}(B,\R^d)$. One can extend $u=(u_1,...,u_d)$ to a function $f\in C^{0,r}(\R^N,\R^d)$ defining for $1\leq j\leq d$
\begin{align}
&f_j(\xi):=\inf\left\{u_j(x)+H(u_j)|x-\xi|^r:x\in B\right\},\quad H(u_j):=\sup_{x,\xi\in B,x\neq \xi}\frac{|u_j(x)-u_j(\xi)|}{|x-\xi|^r}
\end{align}
and
\begin{align}
f:=(f_1,...,f_d)
,\quad  H(u):=\sup_{x,\xi\in B,x\neq \xi}\frac{|u(x)-u(\xi)|}{|x-\xi|^r},\quad H(f):=\sup_{x,\xi\in \R^N,x\neq \xi}\frac{|f(x)-f(\xi)|}{|x-\xi|^r}\leq dH(u).
\end{align}
Thus, $\Leb^N\left(\left\{x\in K: f(x)\neq u(x)\right\}\right)\leq \Leb^N\left(K\setminus B\right)<\e$.
\end{remark}

\section{Examples of pathological behavior of Sobolev and $BV^q$ functions}
\label{sec:examples of pathological behavior of The Sobolev and $BV^q$ functions}
In the examples below we use the function $u(x):=\log\left|\log|x|\right|$, $u:B^k_{1/2}(0)\to \R$, where $k\in \N$, and we denote by $B^k_R(w)$ the open ball in $\R^k$ centred at $w$ with radius $R>0$. Notice that for $k\geq 2$, $\left|\nabla u(x)\right|\leq\frac{1}{\left|x\right|\left|log\left|x\right|\right|}$ for every $x\in B^k_{1/2}(0),x\neq 0$, so $u\in W^{1,k}(B^k_{1/2}(0))$. Notice also that for every $0<r<1$ we get the estimate $u(x)\geq \log\left|\log(r)\right|$ for every $x\in B^k_r(0)\setminus\{0\}$, and $\lim_{r\to 0^+}\log\left|\log(r)\right|=\infty$.

Recall that a Sobolev function $\eta \in W^{1,p}(\Omega)$ where $\Omega\subset\R^N$ has bounded Lipschitz boundary can be extended to a Sobolev function
$\bar{\eta}\in W^{1,p}(\R^N)$, and it has a trace $T\bar{\eta}\in W^{1-\frac{1}{p},p}(\R^{N-1})$.

Recall the following inclusions.
For open set $\Omega\subset \R^N$ with bounded Lipschitz boundary, $r\in (0,1)$ and $p\in [1,\infty)$ we have $W^{1,p}(\Omega)\subset W^{r,p}(\Omega)$ (Proposition 2.2 in \cite{FSF}). For general open set $\Omega\subset \R^N$ and $p\in [1,\infty)$ we have $W^{1/p,p}(\Omega)\subset BV^{p}(\Omega)$ (Theorem 1.2 in \cite{P}), and $BV(\Omega)\cap L^\infty(\Omega)\subset BV^p(\Omega)$ (Theorem 1.1 in \cite{P}).

\begin{example}
\label{ex:4}
In this example we show the existence of a function $u\in W^{1,p}(\Omega)$, $1<p=N$, such that $\mathcal{H}^{0}(\mathcal{S}'_u)=0$ and $\mathcal{H}^{0}(\mathcal{S}_u)>0$.

For the function $u$, we get $u\in W^{1,N}(B^N_{1/2}(0))$, and
\begin{equation}
\lim_{r\to 0^+}\fint_{B^N_r(0)}u(z)dz\geq \lim_{r\to 0^+}\log\left|\log(r)\right|=\infty.
\end{equation}
Thus, $\mathcal{H}^0(\mathcal{S}_u)=\mathcal{H}^0(\{0\})=1>0$. By Theorem \ref{thm:fine properties of Sobolev function(1)} we get
$\mathcal{H}^0(\mathcal{S}'_u)=0$.
\end{example}

\begin{example}
\label{ex:3}
In this example we show the existence of a function $f\in W^{1,p}(\Omega)$, $1<p<N,p\in \N$, such that $\mathcal{H}^{N-p}(\mathcal{S}'_f)=0$ and $\mathcal{H}^{N-p}(\mathcal{S}_f)>0$.

Let $p\in \N$ such that $1<p<N$. Then $u\in W^{1,p}(B^p_{1/2}(0))$. Define $f(x,z):=u(x)$, $f:\Omega\to \R$, where $\Omega:=B^p_{1/2}(0)\times B^{N-p}_{1/2}(0)$ ($f$ is not defined at points of the set $E$ defined below). The function $f$ has weak derivatives $\frac{\partial f}{\partial x_i}=\frac{\partial u}{\partial x_i},1\leq i\leq p$, $\frac{\partial f}{\partial z_i}=0,1\leq i\leq N-p$ in $L^p(\Omega)$. Thus, $f\in  W^{1,p}(\Omega)$. Set
$$
E:=\left\{(0,z):0\in \R^p,z\in B^{N-p}_{1/2}(0)\right\}.
$$
It follows for every $(0,z_0)\in E$
\begin{equation}
\lim_{r\to 0^+}\fint_{B^N_r\left((0,z_0)\right)}f(x,z)dxdz\geq \lim_{r\to 0^+}\log\left|\log(r)\right|=\infty.
\end{equation}
Thus, $\mathcal{H}^{N-p}(\mathcal{S}_f)\geq \mathcal{H}^{N-p}(E)=\mathcal{H}^{N-p}(B^{N-p}_{1/2}(0))>0$. By Theorem \ref{thm:fine properties of Sobolev function(1)} we get
$\mathcal{H}^{N-p}(\mathcal{S}'_f)=0$.
\end{example}

\begin{example}
\label{ex:1}
In this example we show the existence of a function $g\in W^{1/2,2}(\Omega)$, $\Omega\subset\R$, such that $\mathcal{H}^{0}(\mathcal{S}'_g)=0$ and $\mathcal{H}^{0}(\mathcal{S}_g)>0$.

Since $u\in W^{1,2}(B^{2}_{1/2}(0))$, then one can extend it to a function $\bar{u}\in W^{1,2}(\R^2)$. Let $g:=T\bar{u}\in W^{1/2,2}(\R)$ be the trace of $\bar{u}$ to $\R$. It follows that $g(x)=\log\left|\log|x|\right|, x\in B^1_{1/2}(0)\setminus \{0\}$, $g\in W^{1/2,2}(B^1_{1/2}(0))$. We have
\begin{equation}
\lim_{r\to 0^+}\fint_{B^1_r(0)}g(x)dx\geq \lim_{r\to 0^+}\log\left|\log(r)\right|=\infty.
\end{equation}
Thus, $\mathcal{H}^0(\mathcal{S}_g)=\mathcal{H}^0(\{0\})=1>0$. By Theorem \ref{thm:fine property of fractional Sobolev functions(1)} we get $\mathcal{H}^{1-\frac{1}{2}2}(\mathcal{S}'_g)=0$, in particular $\mathcal{S}'_g=\emptyset$.
\end{example}

\begin{example}
\label{ex:2}
In this example we show the existence of a function $\phi\in W^{1/2,2}(\Omega)$, $\Omega\subset\R^N$, such that $\mathcal{H}^{N-1}(\mathcal{S}'_\phi)=0$ and $\mathcal{H}^{N-1}(\mathcal{S}_\phi)>0$.

Since $u\in W^{1,2}(B^{2}_{1/2}(0))$, then one can extend it to a function $\bar{u}\in W^{1,2}(\R^2)$. Let $F(z,w):=\bar{u}(z)$, where $z\in \R^2,w\in \R^{N-1}$. Let us multiply the function $F$ by a smooth function which is constant $1$ on $B^{N+1}_{1/2}(0)$ and has compact support in $B^{N+1}_1(0)$. Let us denote the resulting function by $F_0$. Thus, $F_0\in W^{1,2}(\R^{N+1})$ (the multiplication by a function with compact support is needed to insure that $F_0$ and its weak derivatives are in $L^2(\R^{N+1})$). Let $G:=TF_0\in W^{1/2,2}(H)$ be the trace of $F_0$ to the set
$H:=\left\{(z_1,0,w):z_1\in\R,w\in \R^{N-1} \right\}$.
Note that for every $|(z_1,w)|<\frac{1}{2}$ we get
$$
G(z_1,0,w)=F_0(z_1,0,w)
=F(z_1,0,w)=\bar{u}(z_1,0)=u(z_1,0)=\log\left|\log|z_1|\right|.
$$
Define 
\begin{equation}
\label{eq:definition of phi}
\phi:\Omega\to \R,\quad \phi(z_1,w):=G(z_1,0,w)=\log\left|\log|z_1|\right|,
\end{equation}
where
\begin{equation}
\label{eq:definition of the domain of phi}
\Omega:=\left\{(z_1,w)\in\R\times\R^{N-1}:|(z_1,w)|<\frac{1}{2} \right\}.
\end{equation}
We have $\phi\in W^{1/2,2}(\Omega)$. Set $E:=\left\{(0,w):w\in \R^{N-1},|w|<1/2\right\}\subset \Omega$. It follows for every $(0,w_0)\in E$
\begin{equation}
\lim_{r\to 0^+}\fint_{B^{N}_r((0,w_0))}\phi(z_1,w)dz_1dw\geq \lim_{r\to 0^+}\log\left|\log(r)\right|=\infty.
\end{equation}
Thus, $\mathcal{H}^{N-1}(\mathcal{S}_\phi)\geq \mathcal{H}^{N-1}(E)>0$. By Theorem \ref{thm:fine property of fractional Sobolev functions(1)} we get $\mathcal{H}^{N-1}(\mathcal{S}'_\phi)=0$.
\end{example}

In the following two examples we use the definition of oscillation blow-up (Definition \ref{def:oscillation blow-up}).
\begin{example}
In this example we show the existence of a function $h\in BV^2(\Omega)$, $\Omega\subset \R$, such that $\mathcal{H}^{0}(\mathcal{J}_h)=0$ and $\mathcal{H}^{0}(\mathcal{J}'_h)>0$.

Let $g\in W^{1/2,2}(B^1_{1/2}(0))$ as in Example \ref{ex:1}. Thus $g\in BV^{2}(B^1_{1/2}(0))$. Define a function $h(x)=g(x)+\chi_{(0,1/2)}(x)$, where $\chi_{(0,1/2)}(x)$ is the characteristic function of $(0,1/2)$ at the point $x$.  Since $\chi_{(0,1/2)}\in BV(B^1_{1/2}(0))$ (and bounded), then $\chi_{(0,1/2)}\in BV^2(B^1_{1/2}(0))$. Since $BV^2(B^1_{1/2}(0))$ is a vector space, we have $h\in BV^2(B^1_{1/2}(0))$. Since $h(x)\geq g(x)$ for $x\in B^1_{1/2}(0)\setminus\{0\}$, then
\begin{equation}
\lim_{r\to 0^+}\frac{1}{r}\int^r_0h(x)dx\geq \lim_{r\to 0^+}\log\left|\log(r)\right|=\infty
\end{equation}
and
\begin{equation}
\lim_{r\to 0^+}\frac{1}{r}\int^0_{-r}h(x)dx\geq \lim_{r\to 0^+}\log\left|\log(r)\right|=\infty.
\end{equation}

Therefore, the one-sided approximate limits $h^+(0),h^-(0)$ do not exist. Therefore, since $h$ is continuous in $ B^1_{1/2}(0)\setminus\{0\}$, then
$\mathcal{J}_h=\emptyset$ and $\mathcal{H}^0(\mathcal{J}_h)=0$.
Let us show that $0\in \mathcal{J}'_h$. By Proposition \ref{prop:oscillation blow-up of sum equals to the sum of the oscillation blow-up}, $h_0=g_0+\left(\chi_{(0,1/2)}\right)_0$ if the oscillation blow-ups $g_0,\left(\chi_{(0,1/2)}\right)_0$ exist at $0$. Since $\mathcal{S}'_g=\emptyset$, in particular $0\notin \mathcal{S}'_g $, which means that

\begin{equation}
\lim_{\rho\to
0^+}\Bigg(\inf\limits_{z\in\mathbb{R}}\fint_{B^1_\rho(0)}|g(y)-z|dy\Bigg)=0.
\end{equation}

By definition of oscillation blow-up (Definition \ref{def:oscillation blow-up}) we get that $g_0=0$, which means that, $0$ is an oscillation blow-up of $g$ at $0$. Since $0\in \mathcal{J}_{\chi_{(0,1/2)}}$, then we get by Remark \ref{rem:oscillation blow-up in case of jump and generalized jump points} that
\begin{equation}
\left(\chi_{(0,1/2)}\right)_0(y)=
\begin{cases}
\left(\chi_{(0,1/2)}\right)^+(0)=1\quad &y>0\\
\left(\chi_{(0,1/2)}\right)^-(0)=0\quad &y<0\\
\end{cases},\quad \nu_{\chi_{(0,1/2)}}(0)=1,
\end{equation}
is an oscillation blow-up of $\chi_{(0,1/2)}$ at $0$. Thus, $h_0=\left(\chi_{(0,1/2)}\right)_0$. By definition of oscillation blow-up (Definition \ref{def:oscillation blow-up}) and definition of generalized jump point (Definition \ref{def:generalized approximate limit-oscillation points}) we conclude that $0\in \mathcal{J}'_h$. Thus, $\mathcal{H}^0(\mathcal{J}'_h)>0$.
\end{example}

\begin{example}
In this example we show the existence of a function $P\in BV^2(U)$, $U\subset \R^N$, such that $\mathcal{H}^{N-1}(\mathcal{J}_P)=0$ and $\mathcal{H}^{N-1}(\mathcal{J}'_P)>0$.

Let $\phi\in W^{1/2,2}(\Omega)$, where $\phi$ and $\Omega$ are defined in \eqref{eq:definition of phi} and \eqref{eq:definition of the domain of phi} in Example \ref{ex:2}. Let $\gamma\in (0,1)$ be sufficiently small such that $U\subset \Omega$, where $U:=(-\gamma,\gamma)\times (-\gamma,\gamma)^{N-1}$. Here $(-\gamma,\gamma)^{N-1}$ is the $N-1$ Cartesian product of the interval $(-\gamma,\gamma)$. So $\phi\in W^{1/2,2}(U)$, and since $W^{1/2,2}(U)\subset BV^2(U)$, then $\phi\in BV^2(U)$.

Define a function $P(x)=\phi(x)+\varphi(x)$, where $\varphi(x):=\chi_{(0,\gamma)\times (-\gamma,\gamma)^{N-1}}(x)$ is the characteristic function of the product $(0,\gamma)\times (-\gamma,\gamma)^{N-1}$ at the point $x$.  Since
$\varphi\in BV(U)\cap L^\infty(U)$, then $\varphi\in BV^2(U)$. Since $BV^2(U)$ is a vector space, we have $P\in BV^2(U)$.
Set
$$
E:=\left\{(0,w):w\in (-\gamma,\gamma)^{N-1}\right\}\subset U.
$$
Since for every $\nu\in S^{N-1}$ and $(0,w)\in E$
\begin{equation}
\lim_{r\to 0^+}\fint_{B^+_r ((0,w),\nu)}P(x)dx\geq \lim_{r\to 0^+}\fint_{B^+_r ((0,w),\nu)}\phi(x)dx\geq \lim_{r\to 0^+}\log\left|\log(r)\right|=\infty,
\end{equation}
then $(0,w)\notin \mathcal{J}_P$. Therefore, since $P$ is continuous in $U\setminus E$, $\mathcal{J}_P=\emptyset$ and $\mathcal{H}^{N-1}(\mathcal{J}_P)=0$.

Let us show that $\mathcal{H}^{N-1}$ a.e. $x\in E$ is contained in $\mathcal{J}'_P$. By Proposition \ref{prop:oscillation blow-up of sum equals to the sum of the oscillation blow-up}, for $(0,w)\in E$, $P_{(0,w)}=\phi_{(0,w)}+\varphi_{(0,w)}$ if the oscillation blow-ups $\phi_{(0,w)},\varphi_{(0,w)}$ exist at $(0,w)$. In Example \ref{ex:2} it was proved that $\mathcal{H}^{N-1}(\mathcal{S}'_\phi)=0$. So we get that $\mathcal{H}^{N-1}$ a.e. $x\in U$ is an approximate limit-oscillation point of $\phi$, which means that

\begin{equation}
\lim_{\rho\to
0^+}\Bigg(\inf\limits_{z\in\mathbb{R}^d}\fint_{B_\rho(x)}|\phi(y)-z|dy\Bigg)=0.
\end{equation}

By the definition of oscillation blow-up (Definition \ref{def:oscillation blow-up}), we get that the zero function is an oscillation blow-up of $\phi$ at $\mathcal{H}^{N-1}$ a.e. point in $U$. Since every $(0,w)\in E$ satisfies $(0,w)\in \mathcal{J}_{\varphi}$, then we get by Remark \ref{rem:oscillation blow-up in case of jump and generalized jump points} that
\begin{equation}
\varphi_{(0,w)}(y)=
\begin{cases}
1\quad &y_1>0\\
0\quad &y_1<0\\
\end{cases}, \quad y:=(y_1,...,y_N), \nu_{\varphi}((0,w))=e_1,
\end{equation}
is an oscillation blow-up of $\varphi$ at $(0,w)$. Thus, $P_{(0,w)}=\varphi_{(0,w)}$ at $\mathcal{H}^{N-1}$ a.e. $(0,w)\in E$. By definition of oscillation blow-up (Definition \ref{def:oscillation blow-up}) and definition of generalized jump point (Definition \ref{def:generalized approximate limit-oscillation points}), we conclude that $\mathcal{H}^{N-1}$ a.e. point in $E$ lies in $\mathcal{J}'_P$, so $\mathcal{H}^{N-1}(\mathcal{J}'_P)\geq \mathcal{H}^{N-1}(E)>0$.
\end{example}

\section{Appendix}
\subsection{Aspects of differential geometry}
\begin{definition}($C^k$ submanifold of $\mathbb{R}^N$)
\label{de:definition of submanifold of $R^N$}
Let $l,k\in \N, l\leq N$. We say that a set $\Gamma\subset \mathbb{R}^N$ is an $l-$dimensional submanifold of $\mathbb{R}^N$ of class $k$ if and only if for every
$x\in \Gamma$ there exist an open set $U\subset \mathbb{R}^N$ containing $x$, an open set $V\subset \mathbb{R}^l$ and a function $h:V\to \Gamma\cap U$ such that: (1) $h\in C^k(V)$; (2) $h$ is homehomorphism; (3) for every $q\in V$, the differential $dh_q:\mathbb{R}^l\to \mathbb{R}^N$ is one-to-one. The function $h$ is called \textbf{coordinate system around $x$}.
\end{definition}

\begin{remark}
Recall that if $\Gamma$ is an $(N-1)-$dimensional submanifold of $\mathbb{R}^N$ of class $1$, then there exists a continuous unit normal
$\nu$ along $\Gamma$, which means that a continuous mapping
$\nu:\Gamma\to S^{N-1}$, where $S^{N-1}$ is the $(N-1)-$dimensional unit sphere in $\R^N$, such that for every $x\in \Gamma$ the vector $\nu(x)$ is normal to $\Gamma$ at the point $x$. For a non zero vector $v\in \R^N$ we denote by $v^\perp$ the
$(N-1)-$dimensional vector subspace of $\R^N$ orthogonal to the vector $v$:
\begin{equation}
v^\perp:=\left\{y\in \mathbb{R}^N:v\cdot y=0\right\},
\end{equation}
where $v\cdot y$ is the inner product between $v$ and $y$ defined by
$v\cdot y=\sum_{j=1}^Nv_jy_j$ for $v=(v_1,...,v_N),y=(y_1,...,y_N)$.
\end{remark}
\begin{lemma}
\label{lem:the normal is invariant under translations}
Let $\Omega\subset \R^N$ be an open set and $f\in C^1(\Omega)$. Let
$n\in \R^N,s\in \R$. Assume that for $x\in \Omega,y:=f(x)$, the set $\Gamma:=f^{-1}(\{y\})$ is an
$(N-1)-$dimensional submanifold  of class $1$ in $\mathbb{R}^N$. Then, $\nu(x)$ is normal to $\Gamma$ at $x$ if and only if $\nu(x)$ is normal to $\Gamma-sn$ at $x-sn$.
\end{lemma}
\begin{proof}
Notice that $h:\Omega'\subset \R^{N-1}\to \Gamma$ is a coordinate system around
$x\in\Gamma$
if and only if
$g:\Omega'\to \Gamma-sn$, $g(z)=h(z)-sn$, is a coordinate system around $x-sn\in \Gamma-sn$. The derivatives of $h$ and $g$ coincide:
$\partial_{z_i}h(z)=\partial_{z_i}g(z)$, $\forall 1\leq i\leq N-1$ and $\forall z\in \Omega'$. Hence, the tangent space to $\Gamma$ at $x$ and the tangent space to $\Gamma-sn$ at $x-sn$ are parallel, so $\nu(x)$ is normal to $\Gamma$ at $x$ if and only if $\nu(x)$ is normal to $\Gamma-sn$ at $x-sn$.
\end{proof}

\begin{remark}
\label{rem:grad f orthogonal to the manifold}
Recall that if $\Omega\subset\mathbb{R}^N$ is an open set and $f:\Omega\to \mathbb{R}$ is a $C^1$ function such that
$f^{-1}(\{s\})$ is $(N-1)-$dimensional submanifold  of class $1$ in $\mathbb{R}^N$ for some $s\in \R$, then the gradient $\nabla f$ defines a continuous normal along $f^{-1}(\{s\})$.
\end{remark}

\begin{lemma}
\label{lem:parametric surface relative to vector $n$}
Let
$\Gamma$ be an
$(N-1)-$dimensional submanifold of
$\mathbb{R}^N$ of class $1$. Let $\nu:\Gamma\to S^{N-1}$ be a continuous unit normal along $\Gamma$. Then, for every $x_0\in \Gamma$ there exists an open set
$\Gamma_{x_0} \subset\Gamma$(in the induced topology from $\R^N$ to $\Gamma$) containing $x_0$, such that for every
$n_0\notin \nu(x_0)^\perp$, there exist $t_0>0$ and $\delta_0>0$ such that for every $0<t<t_0$ and $n\in \overline{B}_{\delta_0}(n_0)$ there exist an open set
\begin{equation}
O_{t,n}(\Gamma_{x_0}):=\cup_{s\in (-t,t)}(\Gamma_{x_0}-sn)
\end{equation}
and a $C^1$ and Lipschitz function
\begin{equation}
f:O_{t,n}(\Gamma_{x_0})\to (-t,t),
\end{equation}
such that for every
$z:=x-sn\in O_{t,n}(\Gamma_{x_0})$ we have $f(z)=s$ and $|\nabla f(z)|=\frac{1}{|\nu(z+f(z)n)\cdot n|}$.
\end{lemma}

\begin{proof}
Let $x_0\in \Gamma$. Let $h:U\to h(U)$ be a coordinate system around $x_0$ as in Definition $\ref{de:definition of submanifold of $R^N$}$, where $U\subset \R^{N-1}$ is an open set, $x_0'\in U$ and $h(x_0')=x_0$.
Let $n_0\notin \nu(x_0)^\perp$. Let us define a function
\begin{equation}
H:U\times \mathbb{R}\times\mathbb{R}^N\to \mathbb{R}^N\times\mathbb{R}^N,\quad H(x',s,n)=\left(h(x')-sn,n\right).
\end{equation}
Since $dh_{x_0'}$ is one-to-one and $n_0\notin \nu(x_0)^\perp$, the differential $dH_{(x'_0,0,n_0)}$ is an isomorphism. By the inverse function theorem, there exist open set
$W\subset U$, a number
$t_0>0$ and an open set $Y_0\subset \mathbb{R}^N$ such that
$(x'_0,0,n_0)\in W\times(-t_0,t_0)\times Y_0$ and
$H:W\times(-t_0,t_0)\times Y_0\to \R^N\times\R^N$
is a $C^1$ diffeomorphism on its image.
We can assume without loss of generality  that the $C^1$ function
$H^{-1}:H\left(W\times(-t_0,t_0)\times Y_0\right)\to W\times(-t_0,t_0)\times Y_0$
is a Lipschitz function (otherwise we take a smaller neighbourhood of
$\left(x_0,n_0\right)$ in
$H\left(W\times(-t_0,t_0)\times Y_0\right)$ ).
Set $\Gamma_{x_0}:=h(W)$ and for any $0<t<t_0$ and $n\in Y_0$
\begin{equation}
O_{t,n}(\Gamma_{x_0}):=\cup_{s\in (-t,t)}(\Gamma_{x_0}-sn).
\end{equation}

The set $O_{t,n}(\Gamma_{x_0})$ is an open set as an image of the open set
$W\times(-t,t)\times Y_0$ under the open map which is obtained by the composition of the open maps $H$ and the projection
$\pi:\mathbb{R}^N\times\mathbb{R}^N\to \mathbb{R}^N$ on the first $N$ coordinates:
$\pi\circ H(x',s,n)=\pi(h(x')-sn,n)=h(x')-sn$ for $(x',s,n)\in W\times(-t,t)\times Y_0$.

Note that for every
$s_1\neq s_2$, $s_1,s_2\in (-t,t),0<t<t_0$, we have
$\left(\Gamma_{x_0}-s_1n\right)\cap\left(\Gamma_{x_0}-s_2n\right)=\emptyset$. If not, then there exists
$z\in \left(\Gamma_{x_0}-s_1n\right)\cap\left(\Gamma_{x_0}-s_2n\right)$. Thus, there exist $x_1',x_2'\in W$ such that $z=h(x'_1)-s_1n=h(x'_2)-s_2n$. So $H(x_1',s_1,n)=(z,n)=H(x_2',s_2,n)$. Since $H$ is injective on $W\times(-t,t)\times Y_0$, we get $s_1=s_2$ which contradicts the assumption.

Let $\delta_0>0$ be such that $\overline{B}_{\delta_0}(n_0)\subset Y_0\cap\left( \mathbb{R}^N\setminus \nu(x_0)^\perp\right)$. Let $0<t<t_0$ and $n\in \overline{B}_{\delta_0}(n_0)$. Define functions
\begin{equation}
G:\R^N\to \R^N\times\R^N, G(z)=(z,n),\quad  \pi_N:\R^N\times\R^N\to \R,\pi_N(x_1,...,x_N,x_{N+1},...,x_{2N})=x_N,
\end{equation}
and
\begin{equation}
f:=\pi_N\circ H^{-1}\circ G,\quad f:O_{t,n}(\Gamma_{x_0})\to (-t,t).
\end{equation}
The function $f$ is Lipschitz and $C^1$ as a composition of Lipschitz and $C^1$ functions, and for $z\in O_{t,n}(\Gamma_{x_0})$, $z=h(x')-sn$ for some $x'\in W,s\in (-t,t)$, we have $f(z)=s$.
By derivation with respect to $s$ we obtain
\begin{align}
1=\partial_sf(x-sn)=\nabla f(x-sn)\cdot (-n),\quad\forall s\in (-t,t).
\end{align}
For every point
$z=x-sn\in O_{t,n}(\Gamma_{x_0})$, since $\nu(x)$ is normal to
$\Gamma_{x_0}$ at $x$, then, by Lemma
$\ref{lem:the normal is invariant under translations}$, $\nu(x)$ is normal to
$\Gamma_{x_0}-sn$ at the point $x-sn$. Since, by Remark $\ref{rem:grad f orthogonal to the manifold}$, $\nabla f(x-sn)$ is also normal to $\Gamma_{x_0}-sn$ at the point $x-sn$, there exists a scalar $\alpha(x)\neq 0$ such that
$\nabla f(x-sn)=\alpha(x)\nu(x)$. Thus,
\begin{equation}
-1=\nabla f(x-sn)\cdot n=\alpha(x) \nu(x)\cdot n\Longrightarrow \alpha(x)=\frac{-1}{\nu(x)\cdot n},
\end{equation}
and we obtain for every $z=x-sn\in O_{t,n}\left(\Gamma_{x_0}\right)$
\begin{equation}
\nabla f(z)=\nabla f(x-sn)=\frac{-1}{\nu(x)\cdot n}\nu(x)=\frac{-1}{\nu(z+f(z)n)\cdot n}\nu(z+f(z)n),
\end{equation}
so
\begin{equation}
|\nabla f(z)|=\frac{1}{|\nu(z+f(z)n)\cdot n|}.
\end{equation}
\end{proof}

\subsection{Aspects of geometric measure theory}
\begin{definition}($k-$dimensional densities)
\label{def:densities}
Let $\mu$ be a positive Radon measure in an open set $\Omega\subset\mathbb{R}^N$ and $k\geq0$. The upper and lower $k-$dimensional densities of $\mu$ at $x\in \Omega$ are respectively defined by
\begin{equation}
\Theta^*_{k}(\mu,x):=\limsup_{\rho\to 0^+}\frac{\mu\left(B_\rho(x)\right)}{\rho^{k}\alpha(k)},\quad \Theta_{*k}(\mu,x):=\liminf_{\rho\to 0^+}\frac{\mu\left(B_\rho(x)\right)}{\rho^{k}\alpha(k)}.
\end{equation}
If $\Theta^*_{k}(\mu,x)=\Theta_{*k}(\mu,x)$, then their common value is denoted by $\Theta_{k}(\mu,x)$.
For any Borel set
$\Gamma\subset\Omega$ we define also
\begin{equation}
\Theta^*_{k}(\Gamma,x):=\limsup_{\rho\to 0^+}\frac{\Haus^{k}\left(\Gamma\cap B_{\rho}(x)\right)}{\rho^{k}\alpha(k)},\quad
\Theta_{*k}(\Gamma,x):=\liminf_{\rho\to 0^+}\frac{\Haus^{k}\left(\Gamma\cap B_{\rho}(x)\right)}{\rho^{k}\alpha(k)}
\end{equation}
and, if they agree, we denote the common value of these densities by $\Theta_{k}(\Gamma,x)$.
\end{definition}

\begin{theorem}(Theorem 2.56 in \cite{AFP})
(Hausdorff and Radon measures)
\label{thm:Hausdorff and Radon measures}
Let $\Omega\subset\mathbb{R}^N$ be an open set, $\mu$ a positive Radon measure in $\Omega$ and $k\geq 0$. Then, for any $t\in (0,\infty)$ and any Borel set $B\subset\Omega$ the following implication holds:
\begin{equation}
\Theta^*_{k}(\mu,x)\geq t\quad \forall x\in B\quad \Longrightarrow \quad\mu \geq t\mathcal{H}^k\llcorner B.
\end{equation}
\end{theorem}

\begin{proposition}
\label{Prop:controlling one measure by another}(Proposition 2.21 in \cite{AFP})
Let $\mu$ and $\nu$ be positive Radon measures in $\R^N$ and let $t\in [0,\infty)$. For any Borel set $E\subset\Supp(\mu)$ the following two implications hold:
\begin{equation}
D_\mu^-\nu(x):=\liminf_{\rho\to 0^+}\frac{\nu\left(B_\rho(x)\right)}{\mu\left(B_\rho(x)\right)}\leq t\quad \forall x\in E\quad \Longrightarrow \nu(E)\leq t\mu(E),
\end{equation}
\begin{equation}
D_\mu^+\nu(x):=\limsup_{\rho\to 0^+}\frac{\nu\left(B_\rho(x)\right)}{\mu\left(B_\rho(x)\right)}\geq t\quad \forall x\in E\quad \Longrightarrow \nu(E)\geq t\mu(E).
\end{equation}
\end{proposition}

\begin{lemma}
\label{lem:local measure inequality for rectifiable sets gives global ineqality}
Let $\Omega\subset\R^N$ be an open set,
$\mu$ and $\lambda$ are finite positive Radon measures on $\Omega$, $J$ and
$\{\Gamma^i\}_{i=1}^\infty$ are Borel sets in $\Omega$. Assume that $\lambda\left(J\setminus \bigcup_{i=1}^\infty\Gamma^i\right)=0$, and for every $i$ and for
$\lambda$-almost every $x\in \Gamma^i$ there exists a converging to zero sequence $\rho_m>0$ such that
$\lambda\left(B_{\rho_m}(x)\cap\Gamma^i\right)\leq \mu\left(B_{\rho_m}(x)\right)$ for every $m\in\N$. Then
\begin{equation}
\lambda\left(J\right)\leq \mu\left(\Omega\right).
\end{equation}
\end{lemma}

\begin{proof}
Assume without loss of generality that the support of $\mu$ contains
$\Omega$, if not, we can replace $\mu$ by $\mu_0=\mu+\varepsilon f\mathcal{L}^N$, where $\varepsilon>0$ is arbitrary and $f$ is any positive function in $L^1(\Omega)$. Assume also that the sets $\Gamma^i$ are pairwise disjoint, if they are not, then we can replace them by $\Gamma^i_0:=\Gamma^i\setminus \bigcup_{l=1}^{i-1}\Gamma^l$ for $i>1$ and $\Gamma^1_0:=\Gamma^1$. Assume also that $\lambda$ and $\mu$ are finite Radon measures in all of $\R^N$, otherwise we take the restrictions $\lambda\llcorner\Omega,\mu\llcorner\Omega$ which are Radon measures in $\R^N$.

By assumption for each $i$ there exists a set $E^i$ such that $\lambda(E^i)=0$ and for every $x\in \Gamma^i\setminus E^i$ there exists a converging to zero sequence $\rho_m>0$ such that
$\lambda\left(B_{\rho_m}(x)\cap\Gamma^i\right)\leq \mu\left(B_{\rho_m}(x)\right)$ for every $m\in\N$. Thus, for every such $x$ we have
\begin{equation}
\frac{\lambda\llcorner{\Gamma^i}\left(B_{\rho_m}(x)\right)}{\mu\left(B_{\rho_m}(x)\right)}\leq 1, \quad \forall m\in\N.
\end{equation}
Taking the lower limit as $m\to \infty$ we get
\begin{equation}
D_{\mu}^-\left(\lambda\llcorner{\Gamma^i}\right)(x)\leq \liminf_{m\to \infty}\frac{\lambda\llcorner{\Gamma^i}\left(B_{\rho_m}(x)\right)}{\mu\left(B_{\rho_m}(x)\right)}\leq 1
\end{equation}
for every $i$ and every $x\in \Gamma^i\setminus E^i$. Thus, by Proposition $\ref{Prop:controlling one measure by another}$ we get $\lambda(\Gamma^i)=\lambda\llcorner{\Gamma^i}(\Gamma^i\setminus E^i)\leq \mu(\Gamma^i\setminus E^i)\leq \mu(\Gamma^i)$. Thus,
\begin{align}
\lambda\left(J\right)&=\lambda\left(J\cap\left(\bigcup_{i=1}^{\infty}\Gamma^i\right)\right)+\lambda\left(J\setminus \left(\bigcup_{i=1}^{\infty}\Gamma^i\right)\right)=\lambda\left(J\cap\left(\bigcup_{i=1}^{\infty}\Gamma^i\right)\right)\nonumber
\\
&\leq \lambda\left(\bigcup_{i=1}^{\infty}\Gamma^i\right)\leq \mu\left(\bigcup_{i=1}^{\infty}\Gamma^i\right)\leq \mu\left(\Omega\right).
\end{align}
\end{proof}

\begin{definition}(Rectifiable sets)
Let $\Gamma\subset \mathbb{R}^N$ be an $\mathcal{H}^k-$measurable set. We say that $\Gamma$ is \textbf{countably $k-$rectifiable} if and only if there exist countably many Lipschitz functions $f_i:\mathbb{R}^k\to \mathbb{R}^N$ such that
\begin{equation}
\Gamma\subset \bigcup_{i=1}^\infty f_i(\mathbb{R}^k).
\end{equation}
We say that $\Gamma$ is \textbf{countably $\mathcal{H}^k-$rectifiable} if and only if there exist countably many Lipschitz functions $f_i:\mathbb{R}^k\to \mathbb{R}^N$ such that
\begin{equation}
\mathcal{H}^k\left(\Gamma\setminus \bigcup_{i=1}^\infty f_i(\mathbb{R}^k)\right)=0.
\end{equation}
Finally, we say that $\Gamma$ is \textbf{$\mathcal{H}^k-$rectifiable} if and only if $\Gamma$ is countably $\mathcal{H}^k-$rectifiable and $\mathcal{H}^k(\Gamma)<\infty$.
\end{definition}

\begin{theorem}(Theorem 3.2.29. in \cite{F})
\\
\label{thm:rectifiable set is a union of smooth manifold}
A subset $\Gamma\subset\mathbb{R}^N$ is countably $\mathcal{H}^k-$rectifiable if and only if there exist countably many submanifolds $\Gamma^i$ of class $1$ of $\mathbb{R}^N$ such that
\begin{equation}
\mathcal{H}^k\left(\Gamma\setminus \bigcup_{i=1}^\infty\Gamma^i\right)=0.
\end{equation}
We can also choose the $\Gamma^i$ such that $\mathcal{H}^k(\Gamma^i)<\infty$.
\end{theorem}

\begin{theorem}(Theorem 1 in Section 1.5 in \cite{EG}) (Vitali's covering theorem)
\label{thm:Vitali's covering theorem}
Let $\mathcal{F}$ be any collection of nondegenerate closed balls in $\R^n$ with
\begin{equation}
\sup\{diam B:B\in \mathcal{F}\}<\infty.
\end{equation}
Then there exists a countable family $\mathcal{G}$ of disjoint balls in $\mathcal{F}$ such that
\begin{equation}
\bigcup_{B\in \mathcal{F}}B\subset \bigcup_{B\in \mathcal{G}}\hat{B},
\end{equation}
where if $B=\overline{B}_r(x)$, then $\hat{B}=\overline{B}_{5r}(x)$.
\end{theorem}

\begin{theorem}(Corollary 1 in Section 1.5 in \cite{EG}) (Besicovitch's covering theorem)
\label{thm:Besicovitch covering theorem}
Let $\mu$ be a Borel measure on $\mathbb{R}^N$, and $\mathcal{F}$ any collection of nondegenerate
closed balls. Let $A$ denote the set of centers of the balls in $\mathcal{F}$. Assume
$\mu(A)<\infty$ and
$\inf\left\{r:\overline{B}_r(a)\in \mathcal{F}\right\}=0$ for each $a\in A$. Then for each open set $U\subset\mathbb{R}^N$, there exists a countable collection
$\mathcal{G}$ of disjoint balls in $\mathcal{F}$ such that
\begin{equation}
\bigcup_{B\in \mathcal{G}}B\subset U
\end{equation}
and
\begin{equation}
\mu\left(\left(A\cap U\right)\setminus \bigcup_{B\in \mathcal{G}}B\right)=0.
\end{equation}
\end{theorem}

\begin{theorem}(The co-area formula)(For proof see for example \cite{F,EG,AFP})
\label{thm:the co-area formula}
Let $f:\R^N\to \R$ be a Lipschitz function. Then for every function $g\in L^1(\R^N)$ it follows that $g\in L^1\left(f^{-1}(\{s\}),\mathcal{H}^{N-1}\right)$ for $\mathcal{L}^{1}-$almost every $s\in \R$ and
\begin{equation}
\int_{\mathbb{R}^N}g(x)|\nabla f(x)|dx=\int_{\mathbb{R}}\left(\int_{f^{-1}(\{s\})}g(x)d\mathcal{H}^{N-1}(x)\right)ds.
\end{equation}
\end{theorem}

\begin{lemma}(Lemma 2 in section 2.2 in \cite{EG})
\label{lem:properties of Hausdorf measure}
Let $A\subset\R^N$ and $0\leq s<t<\infty$.
\\
(i) If $\mathcal{H}^s(A)<\infty$, then $\mathcal{H}^t(A)=0$.
\\
(ii) If $\mathcal{H}^t(A)>0$, then $\mathcal{H}^s(A)=\infty$.
\end{lemma}

\begin{lemma}
\label{lem:Hausdorff measure of the set of points of positive upper density with lower dimension has measure zero}
Let $\Omega\subset\R^N$ be an open set, and let $\mu$ be a positive Radon measure on $\Omega$ such that $\mu(\Omega)<\infty$. Assume that $\mu$ is absolutely continuous with respect to the $N$-dimensional Hausdorff measure $\mathcal{H}^{N}$. Let $0\leq k<N$ and define a set
\begin{equation}
D:=\left\{x\in \Omega: \Theta^*_k(\mu,x)>0\right\}.
\end{equation}
Then, $\mathcal{H}^{k}(D)=0$.
\end{lemma}

\begin{proof}
Decompose
\begin{equation}
D=\bigcup_{n=1}^\infty D_n,\quad D_n:=\left\{x\in \Omega: \Theta^*_k(\mu,x)\geq \frac{1}{n}\right\}.
\end{equation}
By $\sigma-$subadditivity of the measure $\mathcal{H}^{k}$, we get $\mathcal{H}^{k}(D)\leq \sum_{n=1}^\infty\mathcal{H}^{k}(D_n)$. We have for every $n\in \N$ that $\mathcal{H}^{k}(D_n)=0$ : assume by contradiction there exists $n\in \N$ such that
$\mathcal{H}^{k}(D_n)>0$. So by Theorem $\ref{thm:Hausdorff and Radon measures}$, we have $\mu(D_n)\geq \frac{1}{n}\mathcal{H}^{k}(D_n)>0$, and since $\mu$ is absolutely continuous with respect to $\mathcal{H}^N$ we have $\mathcal{H}^{N}(D_n)>0$. Thus, since $k<N$, we get by Lemma \ref{lem:properties of Hausdorf measure}
$\mathcal{H}^{k}(D_n)=\infty$ which is a contradiction to the finiteness of $\mu$.
\end{proof}

\begin{theorem}(Hardy-Littlewood maximal inequality)
\label{thm:Hardy-Littlewood maximal inequality}
Let $1<p\leq \infty$ and
$f\in L^p(\R^N)$. Then, there exists a constant $C$, depending only on $p$ and $N$, such that
$\|M(f)\|_{L^p(\R^N)}\leq C\|f\|_{L^p(\R^N)}$, where
\begin{align}
M(f)(x):=\sup\limits_{r>0}\left(\fint_{B_{r}(x)}\left|f(y)\right|dy\right)
\end{align}
is the Hardy-Littlewood maximal function.
\end{theorem}

\begin{lemma}
\label{lem: F is at most countable}
Let $(X,\mathcal{E},\sigma)$ be a measure space, where $X$ is a set, $\mathcal{E}$ is a sigma-algebra on $X$ and
 $\sigma:\mathcal{E}\to [0,\infty]$ a measure. Assume that $E\in \mathcal{E}$ is such that $\sigma(E)<\infty$.
Assume $\{E_\alpha\}_{\alpha\in I}$ is a family of sets, where $I$ is a set of indexes, such that for every $\alpha\in I$, $E_\alpha\subset E$,$E_\alpha\in \mathcal{E}$, and $E_\alpha \cap E_{\alpha'}=\emptyset$ for every different $\alpha,\alpha'\in I$. Define the set
\begin{equation}
F:=\left\{\alpha\in I:\sigma(E_\alpha)>0\right\}.
\end{equation}
Then, $F$ is at most countable.
\end{lemma}

\begin{proof}
Let us decompose $F=\cup_{k=1}^\infty F_k,F_k:=\left\{\alpha\in I:\sigma(E_\alpha)>\frac{1}{k}\right\}$. For each $k\in \mathbb{N}$ the set $F_k$ is finite. Otherwise, there exists a sequence
$\{\alpha_j\}_{j=1}^\infty\subset F_k$ of different elements and so
\begin{equation}
\infty>\sigma(E)\geq \sigma\left(\bigcup_{j=1}^\infty E_{\alpha_j}\right)=\sum_{j=1}^\infty \sigma(E_{\alpha_j})\geq \sum_{j=1}^\infty\frac{1}{k}=\infty.
\end{equation}
This contradiction shows that each $F_k$ is a finite set and hence $F$ is at most countable set as a countable union of finite sets.
\end{proof}

\begin{definition}(Definition 3.70 in \cite{AFP})
\label{def:approximate differentiability points}
Let $\Omega\subset\R^N$ be an open set and let $u\in L^1_{loc}(\Omega,\R^d)$. Let $x\in\Omega\setminus \mathcal{S}_u$. We say that $u$ is approximately differentiable at $x$ if there exists a $d\times N$ matrix $L$ such that
\begin{equation}
\label{eq:approximate differentiability}
\lim_{\rho\to 0^+}\fint_{B_\rho(x)}\frac{|u(y)-\tilde{u}(x)-L(y-x)|}{\rho}dy=0.
\end{equation}
If $u$ is approximately differentiable at $x$, the matrix $L$, uniquely determined by \eqref{eq:approximate differentiability}, is called the \textbf{approximate differential} of $u$ at $x$ and denoted $\nabla u(x)$. The set of approximate differentiability points of $u$ is denoted by $\mathcal{D}_u$.
\end{definition}

\begin{theorem}(Calder{\'o}n-Zygmund, Theorem 3.83 in \cite{AFP})
\label{thm:Calderon-Zygmund}
Let $\Omega\subset\R^N$ be an open set. Any function $u\in BV(\Omega,\R^d)$ is approximately differentiable at $\mathcal{L}^N$-a.e. point of $\Omega$. Moreover, the approximate differential $\nabla u$ is the density of the absolutely continuous part of $Du$ with respect to $\mathcal{L}^N$, in particular $\nabla u\in L^1(\Omega,\R^{d\times N})$.
\end{theorem}

\begin{acknowledgement}
The authors would like to express gratitude to Professor Alexander Ukhlov for valuable and insightful discussions on the topic of fine properties of Sobolev functions in general. 
\end{acknowledgement}

\end{document}